\documentclass[envcountsect,orivec,runningheads,11pt]{llncs}
\usepackage{a4wide}
\usepackage{xspace}
\usepackage{amsmath}
\usepackage{amssymb}
\usepackage{url}
\usepackage{graphicx}
\usepackage{footmisc}
\usepackage{stmaryrd}
\graphicspath{{./img/}}
\makeatletter
\DeclareRobustCommand*{\bfseries}{%
  \not@math@alphabet\bfseries\mathbf
  \fontseries\bfdefault\selectfont
  \boldmath
}
\makeatother
\newcommand{\COMMENTED}[1]{}

\newcommand{\lab}[1]{\label{#1}}
\newcommand{\IR}{\mathbb{R}}

\newcommand{\IRc}{\IR_\text{c}}

\newcommand{\IB}{\mathbb{B}}

\newcommand{\ID}{\mathbb{D}}
\newcommand{\IQ}{\mathbb{Q}}
\newcommand{\IO}{\mathbb{O}}
\newcommand{\II}{\mathbb{I}}
\newcommand{\IP}{\mathbb{P}}
\newcommand{\IC}{\mathbb{C}}

\newcommand{\IZ}{\mathbb{Z}}

\newcommand{\IN}{\mathbb{N}}

\newcommand{\dom}{\operatorname{dom}}
\newcommand{\poly}{\operatorname{poly}}
\renewcommand{\Re}{\operatorname{Re}}
\renewcommand{\Im}{\operatorname{Im}}
\newcommand{\Lip}{\operatorname{Lip}}
\newcommand{\Card}{\operatorname{Card}}
\newcommand{\sgn}{\operatorname{sgn}}

\newcommand{\bin}{\operatorname{bin}}

\newcommand{\kernel}{\operatorname{kernel}}

\newcommand{\id}{\operatorname{id}}

\newcommand{\sdzero}{\textup{\texttt{0}}\xspace}
\newcommand{\sdone}{\textup{\texttt{1}}\xspace}

\newcommand{\double}{\textup{\texttt{double}}\xspace}
\newcommand{\Reg}{\calR}
\newcommand{\BCSS}{\BSS}
\newcommand{\BSS}{\text{BSS}\xspace}
\newcommand{\UC}{\operatorname{UC}}
\newcommand{\barR}{\overline{\calR}}
\newcommand{\calR}{\mathcal{R}}
\newcommand{\calA}{\mathcal{A}}
\newcommand{\calB}{\mathcal{B}}
\newcommand{\calO}{\mathcal{O}}
\newcommand{\calF}{\mathcal{F}}

\newcommand{\calM}{\mathcal{M}}

\newcommand{\calT}{\mathcal{T}}

\newcommand{\calP}{\ensuremath{\mathcal{P}}\xspace}
\newcommand{\calUP}{\ensuremath{\mathcal{UP}}\xspace}
\newcommand{\calNP}{\ensuremath{\mathcal{NP}}\xspace}
\newcommand{\sharpP}{\ensuremath{\#\mathcal{P}}\xspace}
\newcommand{\calPSPACE}{\ensuremath{\mathcal{PSPACE}}\xspace}
\newcommand{\calCH}{\ensuremath{\mathcal{CH}}\xspace}
\newcommand{\calEXP}{\text{\sf EXP}\xspace}
\newcommand{\person}[1]{\textsc{#1}}

\newcommand{\mycite}[2]{{\rm\cite[\textsc{#1}]{#2}}}
\newcommand{\Max}{\operatorname{Max}}
\newcommand{\MAX}{\operatorname{MAX}}
\newcommand{\dsolve}{\operatorname{dsolve}}
\newcommand{\myto}{\!\to\!}
\newcommand{\toto}{\rightrightarrows}

\newcommand{\mapstoto}{\Mapsto}
\newcommand{\ball}{\operatorname{ball}}
\newcommand{\cball}{\overline{\operatorname{ball}}}
\newcommand{\Cinfty}{C^\infty}

\newcommand{\Comega}{C^\omega}
\newcommand{\Comegax}[1]{\Comega\!\left(#1\right)}
\newcommand{\Gevrey}{G}
\newcommand{\Cheby}{T}
\newcommand{\OCHEBY}[1]{\IP_{\calT,#1}}
\newcommand{\ICHEBY}[1]{\II_{\calT,#1}}

\newcommand{\myrho}{\rho}

\newcommand{\binary}{\operatorname{binary}}
\newcommand{\unary}{\operatorname{unary}}

\newcommand{\rhody}{\myrho_{\text{\rm dy}}}

\newcommand{\tilderhorho}{[\widetilde{\rhody\!\!\!\to\!\!\rhody}]}
\newcommand{\dyrho}{\rhody^{\ID}}
\newcommand{\dyrhoLip}{\dyrho+\binary(\Lip)}
\newcommand{\deltabox}{\delta_{_{\Box}}}
\newcommand{\entire}{\tilde\varepsilon}
\newcommand{\iRRAM}{\texttt{iRRAM}\xspace}

\newcommand{\MatLab}{\texttt{MatLab}\xspace}

\newcommand{\reduce}{\preceq}
\newcommand{\reducep}{\preceq_{\text{\rm p}}}
\newcommand{\reduceP}{\preceq_{\text{\rm P}}^2}

\newcommand{\Baire}{\IB}
\newtheorem{fact}[theorem]{Fact}
\newtheorem{myremark}[theorem]{Remark}
\newtheorem{myexample}[theorem]{Example}
\newtheorem{myproblem}[theorem]{Problem}
\newtheorem{myquestion}[theorem]{Question}
\newtheorem{mydefinition}[theorem]{Definition}
\newtheorem{mylemma}[theorem]{Lemma}
\newtheorem{myproposition}[theorem]{Proposition}
\newtheorem{observation}[theorem]{Observation}
\newtheorem{goals}[theorem]{Goals}
\makeatletter
\def\institutename{\par
 \begingroup
 \parskip=\z@
 \parindent=\z@
 \setcounter{@inst}{1}%
 \def\and{\qquad\stepcounter{@inst}%
 \noindent$^{\the@inst}$\enspace\ignorespaces}%
 \setbox0=\vbox{\def\thanks##1{}\@institute}%
   \xdef\fnnstart{\c@@inst}%
   \setcounter{@inst}{1}%
   \noindent$^{\the@inst}$\enspace
 \ignorespaces
 \@institute\par
 \endgroup}
\makeatother
\begin{document}
\addtocounter{page}{-1}%
\title{%
Parameterized Uniform Complexity in Numerics: \\
from Smooth to Analytic, from $\calNP$--hard to Polytime\thanks{%
The first author is supported in part by \emph{Kakenhi} 
(Grant-in-Aid for Scientific Research) \texttt{23700009}; 
the others by the \emph{Marie Curie International Research
Staff Exchange Scheme Fellowship} \texttt{294962}
within the 7th European Community Framework Programme.
We gratefully acknowledge
seminal discussions and suggestions from \person{Ulrich Kohlenbach},
\person{Eike Neumann},
\person{Robert Rettinger}, \person{Matthias Schr\"{o}der},
and \person{Florian Steinberg}.}}
\titlerunning{Parameterized Uniform Complexity in Numerics}
\author{Akitoshi Kawamura\inst{1}, \;\;Norbert Th. M\"{u}ller\inst{2}, \;\;Carsten R\"{o}snick\inst{3}, \;\;Martin Ziegler\inst{3}}
\institute{University of Tokyo \qquad \and Universit\"{a}t Trier \qquad \and TU Darmstadt}
\date{}
\makeatletter
\renewcommand\maketitle{\newpage
  \refstepcounter{chapter}%
  \stepcounter{section}%
  \setcounter{section}{0}%
  \setcounter{subsection}{0}%
  \setcounter{figure}{0}
  \setcounter{table}{0}
  \setcounter{equation}{0}
  \setcounter{footnote}{0}%
  \begingroup
    \parindent=\z@
    \renewcommand\thefootnote{\@fnsymbol\c@footnote}%
    \if@twocolumn
      \ifnum \col@number=\@ne
        \@maketitle
      \else
        \twocolumn[\@maketitle]%
      \fi
    \else
      \newpage
      \global\@topnum\z@   
      \@maketitle
    \fi
    \thispagestyle{empty}\@thanks
    \def\\{\unskip\ \ignorespaces}\def\inst##1{\unskip{}}%
    \def\thanks##1{\unskip{}}\def\fnmsep{\unskip}%
    \instindent=\hsize
    \advance\instindent by-\headlineindent
    \if@runhead
       \if!\the\titlerunning!\else
         \edef\@title{\the\titlerunning}%
       \fi
       \global\setbox\titrun=\hbox{\small\rm\unboldmath\ignorespaces\@title}%
       \ifdim\wd\titrun>\instindent
          \typeout{Title too long for running head. Please supply}%
          \typeout{a shorter form with \string\titlerunning\space prior to
                   \string\maketitle}%
          \global\setbox\titrun=\hbox{\small\rm
          Title Suppressed Due to Excessive Length}%
       \fi
       \xdef\@title{\copy\titrun}%
    \fi
    \if!\the\tocauthor!\relax
      {\def\and{\noexpand\protect\noexpand\and}%
      \protected@xdef\toc@uthor{\@author}}%
    \else
      \def\\{\noexpand\protect\noexpand\newline}%
      \protected@xdef\scratch{\the\tocauthor}%
      \protected@xdef\toc@uthor{\scratch}%
    \fi
    \if@runhead
       \if!\the\authorrunning!
         \value{@inst}=\value{@auth}%
         \setcounter{@auth}{1}%
       \else
         \edef\@author{\the\authorrunning}%
       \fi
       \global\setbox\authrun=\hbox{\small\unboldmath\@author\unskip}%
       \ifdim\wd\authrun>\instindent
          \typeout{Names of authors too long for running head. Please supply}%
          \typeout{a shorter form with \string\authorrunning\space prior to
                   \string\maketitle}%
          \global\setbox\authrun=\hbox{\small\rm
          Authors Suppressed Due to Excessive Length}%
       \fi
       \xdef\@author{\copy\authrun}%
       \markboth{\@author}{\@title}%
     \fi
  \endgroup
  \setcounter{footnote}{\fnnstart}%
  \clearheadinfo}

\def\institutename{\par
 \begingroup
 \parskip=\z@
 \parindent=\z@
 \setcounter{@inst}{1}%
 \def\and{\qquad\stepcounter{@inst}%
 \noindent$^{\the@inst}$\enspace\ignorespaces}%
 \setbox0=\vbox{\def\thanks##1{}\@institute}%
 \ifnum\c@@inst=1\relax
   \gdef\fnnstart{0}%
 \else
   \xdef\fnnstart{\c@@inst}%
   \setcounter{@inst}{1}%
   \noindent$^{\the@inst}$\enspace
 \fi
 \ignorespaces
 \@institute\par
 \endgroup}

\makeatother

\maketitle
\begin{abstract}
The synthesis of classical Computational Complexity Theory
with Recursive Analysis provides a quantitative
foundation to reliable numerics. Here
the operators of maximization, integration, and
solving ordinary differential equations are known
to map (even high-order differentiable) polynomial-time 
computable functions to instances which are `hard'
for classical complexity classes $\calNP$, $\sharpP$,
and $\calCH$; but, restricted to analytic 
functions, map polynomial-time computable ones to
polynomial-time computable ones --- non-uniformly!

We investigate the uniform parameterized complexity 
of the above operators in the setting of Weihrauch's
TTE and its second-order extension due to 
Kawamura\&Cook (2010). That is, we explore which 
(both continuous and discrete, first and second order) 
information and parameters
on some given $f$ is sufficient to obtain similar 
data on $\Max f$ and $\int f$;
and within what running time,
in terms of these parameters and 
the guaranteed output precision $2^{-n}$.

It turns out that Gevrey's hierarchy of functions
climbing from analytic to smooth corresponds to
the computational complexity of maximization
growing from polytime to $\calNP$-hard.
Proof techniques involve mainly the Theory of (discrete)
Computation, Hard Analysis, and Information-Based Complexity.
\end{abstract}
\setcounter{footnote}{0}
\renewcommand{\thefootnote}{\fnsymbol{footnote}}

\setcounter{tocdepth}{3}
\renewcommand{\contentsname}{}
\begin{center}
\begin{minipage}[c]{0.95\textwidth}\vspace*{-8ex}%
\tableofcontents
\end{minipage}
\end{center}

\pagebreak

\section{Motivation and Introduction}
Numerical methods provide practical solutions to many, and particularly
to very large, problems arising for instance in Engineering.
Nonlinear partial differential equations for instance are usually
treated by discretizing the domain of the solution function space,
i.e. approximating the latter by a high but finite dimensional
space. Due to the nonlinearity, sub-problems may involve (low-dimensional) 
numerical integration and maximization 
and are regularly handled by standard methods.
For instance Newton Iterations locally converge quadratically 
to a root $x_0$ of $f'$, that is, a candidate optimum.

We thus record that numerical science has devised a variety of 
impressive methods working in practice very well --- in terms 
of an intuitive conception of efficiency. 
A notion capturing this formally, on the other hand, is
at the core of the Theory of Computation and has led to
standard complexity classes like $\calP$, $\calNP$, $\calUP$,
$\sharpP$, $\calCH$, and $\calPSPACE$: for discrete problems,
that is, over sequences of bits encoding, say, 
integers or graphs \cite{Papadimitriou}. To quote from \cite[\S1.4]{BCSS}:

\begin{quote}\it
 The developments described in the previous section
 (and the next) have given a firm foundation to computer
 science as a subject in its own right. Use of the Turing
 machines yields a unifying concept of the algorithm
 well formalized. [\ldots] The situation in numerical analysis
 is quite the opposite. Algorithms are primarily a means to
 solve practical problems. There is not even a formal definition
 of algorithm in the subject. [\ldots] Thus we view numerical
 analysis as an eclectic subject with weak
 foundations; this certainly in no way denies its great
 achievements through the centuries.
\end{quote}

Recursive Analysis is the theory of computation over real numbers 
by rational approximations
up to prescribable absolute error $2^{-n}$. Initiated by Alan Turing
(in the very same work that introduced the machine now named after him 
\cite{Turing37}) it provides a computer scientific foundation 
to reliable numerics \cite{Alefeld} and computer-assisted proofs
\cite{Rump} in unbounded precision; 
cmp., e.g., \cite{Kreinovich,Siam100Digits,Braverman2}. 

\begin{myremark} \lab{r:Numerics}
  Mainstream numerics is usually scrupulous about constant 
  factor gains or losses (e.g. $5\times$) in running time.
  This generally means an implicit restriction to hardware supported 
  calculations, that is, to \double and in particular to
  fixed-precision arithmetic.
\begin{enumerate}
\item[a)]
  Alternative approaches that, in order to approximate the result with \double
  accuracy, use \emph{un}bounded precision for intermediate calculations,
  therefore reside in a `blind spot' --- which 
  Recursive Analysis may shed some light on {\rm\cite{Korovina}}.
\item[b)] Inputs $x$ are in classical numerics generally considered `exact';
  equivalently: both subtraction and the Heaviside function regarded as computable.
\item[c)] 
  Outputs $y$, on the other hand, constitute mere approximations
  to the `true' values $f(x)$. In consequence,
  this notion of real computation necessarily
  lacks closure under composition {\rm\cite[p.325]{YapGuaranteedAccuracy}}.
\end{enumerate}
The latter is more than just a conceptual annoyance:
the (approximate) result of one subroutine cannot in general 
be fed as argument to another in order to obtain a (provably always) correct algorithm,
thus spoiling the modular approach to software development!
\end{myremark}
Rooted in the Theory of Computation, Recursive Analysis on the
other hand is closed under composition --- and interested only
in asymptotic algorithmic behaviour, that is, ignoring constant
factors in running times.

\begin{myremark} \lab{r:iRRAM}
Libraries based on 
this model of computation {\rm\cite{iRRAM,realLib}}, on calculations 
that suffice with \double precision (and thus do not actually 
make use of the enhanced power) can generally be expected to
run 20 to 200 times slower than a direct implementation. 
Even in this limited realm they are useful for fast numerical 
prototyping, that is, for a first quick-and-dirty coding of 
some new algorithmic approach to empirically explore its 
typical running time behaviour and stability properties.
Indeed,
\begin{enumerate}
\item[i)] accuracy issues are taken care
of by the machine automatically
\item[ii)] closure under composition
allows for modularly combining subroutines 
\item[iii)] chosen from, or contributing to,
a variety of standard real functions and non-/linear operators.
\end{enumerate}
Their full power of course lies in computations involving
intermediate or final results with unbounded guaranteed precision
{\rm\cite{RobertBloch}}.
\end{myremark}
Recursive Analysis has over the last few decades
evolved into a rich and flourishing theory with many classical results 
in real and complex analysis investigated for their computability.
That includes, in addition to numbers, also `higher-type' objects 
such as (continuous) functions, (closed and open) subsets, 
and operators thereon \cite{MLQ2,SchroederProbability}: 
by fixing some suitable encoding of the arguments 
(real numbers, continuous functions, closed/open subsets) 
into infinite binary strings. More precisely the 
\emph{Type-2 Theory of Effectivity} (TTE, cmp. Remark~\ref{r:TTE} below) 
studies and compares such encodings (so-called representations) \cite{Weihrauch}:
mostly \emph{qualitatively} in the sense of which mappings they 
render computable and which not.

\begin{fact} \lab{f:Nonunif}
Concerning complexity theory, \person{Ker-I Ko} and \person{Harvey Friedman} 
have constructed in {\rm\cite{KoFriedman}} certain smooth 
(i.e. $C^\infty$) functions $f,g:[0;1]\to[0;1]$
computable in time polynomial in the output precision $n$
(short: polytime) and proved the following:
\begin{enumerate}
\item[a)] The function
$\Max(f):=\big([0;1]\ni t\mapsto\max\{f(x):0\leq x\leq t\}\big)\in\Cinfty[0;1]$
is again polytime 
iff $\calP=\calNP$ holds; cmp. \mycite{Theorem~3.7}{Ko91}.
\item[b)] The function
$\int g:=\big([0;1]\ni t\mapsto\int_0^t g(x)\,dx\big)\in\Cinfty[0;1]$
  is polytime for every polytime $f\in C[0;1]$
  ~iff~ $\calP=\sharpP$; cmp. \mycite{Theorem~5.33}{Ko91}.
\item[c)] 
More recently, one of us succeeded in constructing a
polytime Lipschitz--continuous function
$h:[0;1]\times[-1;1]\to[-1;1]$ such that the unique solution
$u:=\dsolve(h):[0;1]\to[-1;1]$ to the ordinary differential equation
\begin{equation} \label{e:ODE}
  u'(t) \;=\; h\big(t,u(t)\big), \quad u(0)=0
\end{equation}
is again polytime iff $\calP=\calPSPACE$ {\rm\cite{AkiODE}}.
\item[d)]
Restricted to polytime right-hand sides $h\in C^k$,
$\dsolve(h)$ is again polytime iff $\calP=\calCH$ {\rm\cite{Ota}}.
\end{enumerate}
\end{fact}
Put differently: If numerical methods could indeed calculate
maxima (or even integrals or solve ODEs)
efficiently \emph{with prescribable 
error in the worst case}, this would mean a positive answer
to the first Millennium Prize Problem (and beyond); cmp. \cite{Smale}!
The proof of Fact~\ref{f:Nonunif}a)
will be recalled in Example~\ref{x:maxNP} below.

\begin{myremark} \lab{r:tallyUP}
\begin{enumerate}
\item[a)]
We consider here the real problems as operators $\IO$ 
mapping functions $f$ to functions $\IO(f)$, that is, depending
on the upper end $t$ of the interval $[0;t]$ which the given $f$
is to be maximized, integrated, or the ODE solved on.
Fixing $t:=1$ amounts to functionals 
$f\mapsto\Lambda(f):=\IO\big(f\big)(1)$;
and the worst-case complexity of the single 
real numbers $\|f\|=\Max\big(f\big)(1)$ and
$\int_0^1 f(x)\,dx$ for polytime $f\in\Cinfty[0;1]$
corresponds to famous open questions concerning unary
complexity classes $\calP_1$, $\calNP_1$, and $\sharpP_1$
\mycite{Theorems~3.16+5.32}{Ko91}; recall \textsf{Mahaney's Theorem}.
\item[b)]
It has been conjectured {\rm\cite{Shary}}
that the difficulty of optimizing some (smooth)
function $f:[0;1]\to[0;1]$ may arise from it having many maxima;
and indeed this is the case for the `bad' function $f$
according to Fact~\ref{f:Nonunif}a). However a modification 
of the construction proving \mycite{Theorem~3.16}{Ko91}
yields a polytime computable smooth $\tilde f$ such that,
for every $x\in[0;1]$, $\tilde f\big|_{[0;x]}$ attains its  
maximum in precisely one point while relating the complexity
of the single real $\Max\big(\tilde f\big)(1)$ to 
the complexity class $\calUP_1$ of unary decision problems
accepted by a nondeterministic polytime Turing machine with
at most one accepting computation for each input.
Recall the \textsf{Valiant--Vazirani Theorem} and 
that the open question ``$\calP\neq\calUP$'' is equivalent
to that of the existence of cryptographic one-way functions
\mycite{Theorem~12.1}{Papadimitriou}.
\end{enumerate}
\end{myremark}
Indeed, the numerics community seems little aware of these connections.
And, as a matter of fact, such ignorance may almost have some justification:
The above functions $f$ and $g$ and $h$, although satisfying 
strong regularity conditions, are `artificial\footnote{\label{f:Schwarz}%
Theoretical physicists, used to regularly invoking Schwarz's Theorem,
discard the following `counter-example' as \emph{artificial}:
$$ f:(-1;1)^2\to\IR, \qquad
(x,y)\mapsto\left\{ \begin{array}{ll} 
\tfrac{xy(x^2-y^2)}{x^2+y^2} & (x,y)\neq(0,0) \\
0 & (x,y)=(0,0) \end{array}\right. $$
Mathematicians on the other hand point out that
Schwarz's Theorem requires as a prerequisite the
\emph{continuity} of the second derivatives.}' in any intuitive 
sense and clearly do not arise in practical applications. 
This, on the other hand, raises 

\begin{myquestion} \lab{q:Main}
\begin{enumerate}
\item[a)]
Which functions are the ones numerical practitioners regularly
and implicitly allude to when claiming to be able to efficiently
calculate their maximum, integral, and ODE solution? 
\item[b)]
More formally, on which (classes $\calF_k$ of) functions do 
these operations become computable in polynomial time
$\calO(n^d)$ for some (and, more precisely, for which) 
$d=d(k)\in\IN$?
\end{enumerate}
\end{myquestion}
We admit that numerics as pursued in Engineering might not really
need to answer this question but be happy as long as, say, \MatLab
readily solves those particular instances encountered everyday.
Such a pragmatic\footnote{not unsimilar to \person{Paracelsus}'
``\emph{He who heals is right}'' in medicine} approach seems indeed
compatible for instance with the spirit of the \texttt{NAG} library:

\begin{quote} \it
{\tt nag\_opt\_one\_var\_deriv(e04bbc)} \underline{normally} computes a 
sequence of $x$ values which \underline{tend in the limit} to a minimum of 
$F(x)$ subject to the given bounds.
\end{quote}
Both Mathematics and (theoretical) Computer Science, however,
do come with a tradition of explicitly 
\emph{stating prerequisites}\footref{f:Schwarz}
to theorems and \emph{specifying the domains and running-time bounds} 
to (provably correct) algorithms, respectively:
crucial for modularly combining known results/algorithms
to obtain new ones; recall the end of Remark~\ref{r:Numerics}. 
In the case of real computation, finding such a specification
for maximization and integration, say, amounts to answering
Question~\ref{q:Main}b). Moreover, the answer cannot include
all polytime $C^k$--functions for any $k\in\IN$
according to Fact~\ref{f:Nonunif}. On the other hand, we record

\begin{myexample} \lab{x:Upper} 
One important (but clearly too small) class $\calF$ of functions $f:[-1;1]\to\IR$,
for which $\Max f$, $\int f$, and $\dsolve(f)$ 
as well as derivatives $f^{(j)}$ are known
polytime whenever $f$ is, consists of the real 
analytic ones: those locally admitting power series 
expansions, cmp. {\rm\cite[p.208]{Ko91}} and 
see {\rm\cite{Mueller87,Moiske,Mueller95,Bournez}};
equivalently \mycite{Proposition~1.2.12}{RealAnalytic}: 
those satisfying, for some $A,K\in\IN$,
\begin{equation} \label{e:Hadamard1}
\forall |x|\leq1, \quad
\forall j\in\IN: \qquad |f^{(j)}(x)|\;\leq \; A\cdot K^j\cdot j! \enspace .
\end{equation}
Indeed, the following is the standard example of a smooth but non-analytic function:
\[ h\;: [-1;1] \;\ni\; x\;\mapsto\; 
\left\{ \begin{array}{ll} \exp(-1/x) & x>0 \\ 0 & x\leq0  \end{array}\right. \]
Based on l'H\^{o}spital's rule it is easy to verify that $h$ is continuous 
at $x_0=0$ and differentiable 
arbitrarily often with $h^{(j)}(0)=0$, hence having
Taylor expansion around $x_0$ disagree with $h$.
\end{myexample}
Note that both Fact~\ref{f:Nonunif} and Example~\ref{x:Upper}
are stated \emph{non-}uniformly in the sense of ignoring how
$f$ is algorithmically transformed into, say, $\Max f$.
Consisting of lower complexity bounds, this makes Fact~\ref{f:Nonunif}
particularly strong; whereas as upper bounds Example~\ref{x:Upper} 
merely asserts, whenever there exists some $d\in\IN$ and an algorithm 
$\calA$ approximating $f$ within time $\calO(m^d)$ up to error $2^{-m}$,
the existence of some $e\in\IN$ and an algorithm $\calB$ approximating
$\Max f$ up to error $2^{-n}$ within time $\calO(n^e)$:
both the dependence of $\calB$ on $\calA$ and that 
of $e$ on $d$ are ignored. Put differently,
the proofs of Example~\ref{x:Upper} may and do implicitly
make use of many integer parameters of $f$ (such as
numerator and denominator of starting points for Newton's
Iteration converging to a root $x_0$ of $f'$ of known 
multiplicity where $f$ attains its maximum) since,
nonuniformly, they constitute simply constants. 
Useful uniform (computability and) upper complexity bounds 
like in \cite{Bournez}, on the other hand, fully specify 
\begin{enumerate}
\item[i)] which information on $f$ is employed as input
\item[ii)] which asymptotic running times are met in the worst case 
\item[iii)] in terms of which parameters (in addition to the output precision $n$).
\end{enumerate}
The present work extends the Type-2 Theory of Effectivity 
and Complexity \cite{WeihrauchComplexity}
to such \emph{parameterized, uniform} claims.
We focus here on operators and upper complexity bounds,
that is, fully specified provably correct algorithms.

Let us illustrate the relevance of integer parameters, 
and its difference from integral advice \cite{Advice}, 
to real computation of (multivalued) functions:

\begin{myexample} \lab{x:Param1}
\begin{enumerate}
\item[a)]
On the entire real line,
the exponential function and binary addition
and multiplication are computable --- 
but not within time bounded in terms of the output precision $n$ only:
Because already the integral part of
the argument $x$ requires time reading and printing depending 
on (any integer upper bound $k$ on) $x\to\pm\infty$. 
\item[b)]
Restricted to real intervals $[-k;k]$, 
on the other hand, $\exp:|_{[-k;k]}$ is computable
in time polynomial in $\calO(n+k)$,
noting that $\lfloor\exp(k)\rfloor$ has 
binary length of order $k\cdot\log k$. \\
Similarly, binary addition and multiplication
on $[-k;k]\times[-k;k]$ are computable in 
time polynomial in $\calO(n+\log k)$. \\
Here and in the sequel, abbreviate $\log(k):=\lceil\log_2(k+1)\rceil$.
\item[c)]
Given a real symmetric $d\times d$--matrix $A$,
finding (in a multivalued way) some eigenvector is uncomputable.
\item[d)]
However when given, in addition to $A$, also the integer
$\ell(A):=\min_{\lambda\in\sigma(A)}
\lfloor\log_2\dim\kernel(A-\lambda\id)\rfloor
\in\{0,\ldots,\lfloor1+\log_2 d\rfloor\}$,
some eigenvector can be found computably and,
\item[e)]
given the integer $L(A):=\Card\sigma(A)\in\{1,\ldots,d\}$, 
even an entire basis
of eigenvectors, i.e. the matrix can be diagonalized
computably.
\end{enumerate}
\end{myexample}
The rest of this section recalls the state of the art on
real complexity theory with its central notions and definitions.

\subsection{Computability and Complexity Theory of Real Numbers and Sequences}
A real number $x$ is considered \emph{computable}
(that is an element of $\IRc$) if
any (equivalently: all) of the following conditions hold
\mycite{\S4.1+Lemma~4.2.1}{Weihrauch}:
\begin{description}
\item[a1)] The (or any) binary expansion $b_n\in\{0,1\}$
of $x=\sum_{n\geq-N} b_n2^{-n}$ is recursive.
\item[a2)] A Turing machine can produce dyadic 
 approximations to $x$ up to prescribable error $2^{-n}$,
 that is, compute an integer sequence 
 $a:\IN\ni n\mapsto a_n\in\IZ$
 (output encoded in binary) with $|x-a_n/2^{n+1}|\leq2^{-n}$;
 cmp. e.g. \mycite{Definition~3.35}{Eigenwillig}.
\item[a3)] A Turing machine can output rational sequences
 $(c_n)$ and $(\epsilon_n)$ with $|x-c_n|\leq\epsilon_n\to0$;
 cmp. e.g. \mycite{Definition~0.3}{PER}.
\item[a4)] $x$ admits a recursive \emph{signed digit expansion},
 that is a sequence $s_n\in\{\sdzero,-\sdone,+\sdone\}$ with $x=\sum_{n=-N} s_n2^{-n}$.
\end{description}
Moreover the above conditions are strictly stronger than the following
\begin{description}
\item[a5)] A Turing machine can output rational sequences
 $(c'_n)$ with $c'_n\to x$.
\end{description}
In fact relaxing to (a5) is equivalent to permitting 
oracle access to the Halting problem \cite{Ho}!
We also record that the equivalence among (a2) to (a4),
but not to (a1), holds even uniformly.

\begin{myremark} \lab{r:TTE}
TTE constitutes a convenient framework for naturally inducing, 
combining, and comparing encodings of general
continuous universes $X$ like $\IR$:
\begin{enumerate}
\item[a)]
Formally, a \emph{representation} $\xi$ of $X$
is a surjective partial mapping from
Cantor space $\{\sdzero,\sdone\}^\omega$ to $X$
\mycite{\S3}{Weihrauch}. 
\item[b)]
And $\xi$--computing some $x$
means to output a $\xi$--name (out of possibly many) of $x$,
that is, print onto some non-rewritable \cite{Revising} 
tape an infinite binary sequence 
$\bar\sigma=(\sigma_n)_{_n}\in\{\sdzero,\sdone\}^\omega$
with $\xi(\bar\sigma)=x$.
\item[c)]
Representations $\xi$ of $X$ and $\upsilon$ of $Y$
canonically give rise to one, called $\xi\times\upsilon$, of $X\times Y$:
A $\xi\times\upsilon$--name of $(x,y)$ is an infinite binary sequence
$(\sigma_0,\tau_0,\sigma_1,\tau_1,\ldots)$ 
where $\bar\sigma$ constitutes a $\xi$--name of $x$
and $\bar\tau$ an $\upsilon$--name of $y$. 
\item[d)]
Similarly, a countable family of representations $\xi_j$ of $X_j$
canonically induces the representation $\prod_j\xi_j$ of $\prod_j X_j$
as follows: For a fixed computable bijection 
$\langle\,\cdot\,,\,\cdot\,\rangle:\IN\times\IN\to\IN$,
$\bar\sigma=(\sigma_k)_{_k}$ is a $\prod_j\xi_j$--name of 
$\prod_j x_j$ iff
$(\sigma_{\langle n,j\rangle})_{_n}$ is a $\xi_j$--name
of $x_j$ for every $j$.
\item[e)]
The representation $\binary$ of $\IN$ encodes integers in
binary self-delimited with trailing zeros:
$\binary(\sdone\,b_0\,\sdone\,b_1\,\cdots\,\sdone\,b_n\,\sdzero^\omega):=\sum_{j=0}^n b_j2^j$. \\
Representation $\unary$ of $\IN$ encodes integers in unary:
$\unary(\sdone^n\,\sdzero^\omega):=n$.
\item[f)]
Note that the bijection $\bin:\{\sdzero,\sdone\}^*\to\IZ$ with
\begin{eqnarray*}
(\sdone\,w_{n-1}\cdots w_1w_0)&\mapsto& \phantom{1-(}w_0+2w_1+\cdots+2^{n-1}w_{n-1}+2^n, \\
(\sdzero\,w_{n-1}\cdots w_1w_0)&\mapsto& 1-(w_0+2w_1+\cdots+2^{n-1}w_{n-1}+2^n)
\end{eqnarray*}
itself is technically not a representation -- but 
useful as a building block. By abuse of name we
denote its inverse also by $\bin$.
\end{enumerate}
\end{myremark}
Concerning real numbers, the equivalences between (a1) to (a4)
refer to computability.
Not too surprisingly, they break up under the finer perspective of complexity; 
some notions even become useless \mycite{Example~7.2.1}{Weihrauch}. 
On the other hand, (a2) and (a4) do lead to uniformly 
quadratic-time equivalent notions \mycite{Example~7.2.14}{Weihrauch}.
Somehow arbitrarily, but in agreement with \iRRAM's internal data
representation, we focus on the following formalization in TTE:

\begin{mydefinition} \label{d:NumberComplex} 
\begin{enumerate}
\item[a)]
Call $x\in\IR$ \emph{computable in time $t(n)$}
if a Turing machine can, given $n\in\IN$, produce
within $t(n)$ steps $\vec w\in\{\sdzero,\sdone\}^*$
with $|x-\bin(\vec w)/2^{n+1}|\leq2^{-n}$.
\item[b)]
A $\rhody$--name of $x\in\IR$ is (an infinite sequence of \sdzero's and 
\sdone's encoding similarly to Remark~\ref{r:TTE}f)
a sequence $a_n\in\IZ$ such that $|x-a_n/2^{n+1}|\leq2^{-n}$ holds.
\item[c)]
A $\rhody^2$--name of $z\in\IC$ consists of a $\rhody$--name
according to b) of $\Re(z)$ and one of $\Im(z)$,
both infinite sequences interleaved into a single
(Remark~\ref{r:TTE}c).
\item[d)]
A sequence $(x_j)\subseteq\IR$ is \emph{computable in time $t(n,j)$}
if a Turing machine can, given $n,j\in\IN$, produce
within $t(n,j)$ steps some
$a\in\IZ$ (in binary) with $|x_j-a/2^{n+1}|\leq2^{-n}$.
\item[e)]
A $\rhody^\omega$--name of a real sequence $(x_j)$ is a
sequence $(a_j)\in\IZ$, encoded into an infinite binary string
as above, such that $|x_j-a_{\langle n,j\rangle}|\leq2^{-n}$.
\item[f)]
Here and in the sequel fix the \emph{Cantor pairing function}
$\langle n,j\rangle=n+(n+j)\cdot(n+j+1)/2$.
\end{enumerate}
\end{mydefinition}
Note that the running time bound in a) is expressed in terms of
the output precision $n$; which corresponds to classical complexity
theory based on input length by having the argument $n$ encoded in unary.
Indeed, the output $a\in\IZ$ will have length $\Theta(n+\log|x|)$
in Landau notation.
Within a $\rhody$--name of $x$ according to b),
the $\calO(\log|x|+n)$ digits of $a_n$ similarly start
roughly at position $n\cdot(n+\log|x|)$.
Similarly, since $\langle m,j\rangle\leq\calO(m^2+j^2)$,
the digits of $a_{\langle m,j\rangle}$ start
within a $\rhody^\omega$--name of $(x_j)$
at position $n\approx (m^2+j^2)\cdot(m+\max_{i\leq j}\log|x_{i}|)$.
We also record that both the mapping
$(m,j)\mapsto\langle m,j\rangle$ and its inverse 
are classically quadratic-time computable.

\subsection{Type-2 Computability and Complexity Theory: Functions} \lab{s:FuncComplex}
Concerning a real function $f:[0;1]\to\IR$, the following 
conditions are well-known equivalent and thus a reasonable notion
of computability
\cite{Grzegorczyk}, \mycite{\S0.7}{PER}, \mycite{\S6.1}{Weihrauch}, 
\mycite{\S2.3}{Ko91}, \mycite{D\'{e}finition~2.1.1}{Lombardi}.
\begin{description}
\item[b1)] A Turing machine can output a sequence of 
 (degrees and lists of coefficients of) univariate dyadic
 polynomials $P_n\in\ID[X]$ such that
 $\forall n\in\IN:\|f-P_n\|\leq2^{-n}$ holds,
 where $\ID:=\bigcup_n\ID_n$ and $\|f\|:=\max\{|f(x)|:0\leq x\leq 1\}$.
\item[b2)] A Turing machine can, upon input of 
 every $\rhody$--name of some $x\in[0;1]$,
 output a $\rhody$--name of $f(x)$.
\item[b3)] There exists an oracle Turing machine 
 $\calM^?$ which, upon input of each $n\in\IN$ and for every
 (discrete function) oracle answering queries ``$m\in\IN$'' 
 with some $a\in\ID_{m+1}$ such that $|x-a|\leq2^{-m}$,
 prints some $b\in\ID_{n+1}$ with $|f(x)-b|\leq2^{-n}$.
\end{description}
In particular every (even relatively, i.e. oracle) computable 
$f:[0;1]\to\IR$ is necessarily continuous. This gives rise to
two causes for noncomputability: a (classical) recursion theoretic
and a topological one; cmp. Figure~\ref{f:superpol}b).
More precisely, every $f\in C[0;1]$ is computable relative so some 
appropriate oracle; cmp. \mycite{Theorem~3.2.11}{Weihrauch}.

\begin{myremark} \label{r:Uniformity}
\begin{enumerate}
\item[a)]
The above (equivalent) notions of computability are uniform 
in the sense that $x$ is considered given as argument
from which the value $f(x)$ has to be produced.
The \emph{non}uniform relaxation asks of whether,
for every $x\in[0;1]\cap\IRc$, 
$f(x)$ is again computable:
possibly requiring separate algorithms 
to do so for each of the (countably many) $x\in[0;1]\cap\IRc$.
For instance the totally discontinuous Dirichlet function is
trivially computable in this (thus too weak) sense.
\item[b)]
Defining computable functions also on non-computable arguments
(as above) is known to avoid certain pathologies in the partial
case \mycite{Example~9.6.5}{Weihrauch}
while making no difference on well-behaved domains
\mycite{Theorem~9.6.6}{Weihrauch}.
\item[c)]
For spaces $X,Y$ with respective representations $\xi,\upsilon$,
TTE defines a $(\xi,\upsilon)$--realizer of a (possibly partial 
and/or multivalued) function $f:\subseteq X\toto Y$ to be 
a partial mapping $F:\subseteq\{\sdzero,\sdone\}^\omega\to\{\sdzero,\sdone\}^\omega$
such that, for every $\xi$--name $\bar\sigma$ of every $x\in\dom(f)$,
$\tau:=F(\bar\sigma)$ is a $\upsilon$--name of some $y\in f(x)$. 
And $f$ is called $(\xi,\upsilon)$--computable if it admits
a computable $(\xi,\upsilon)$--realizer.
Notion (b2) thus amounts to $(\rhody,\rhody)$--computability.
\end{enumerate}
\end{myremark}
Concerning complexity, (b2) and (b3) have turned out as 
polytime equivalent but different from (b1);
cf. \mycite{\S8.1}{Ko91} and \mycite{Theorem~9.4.3}{Weihrauch}.
So as opposed to discrete complexity theory, 
running time is not measured in terms of the input length
(which is infinite anyway) but in terms of the output
precision parameter $n$ in (b3); and for (b2)
in terms of the time until the $n$-th digit of the infinite
binary output string appears:
in both cases uniformly in (i.e. w.r.t. the worst-case over all)
$x\in[0;1]$. Indeed the dependence on $x$
can be removed by taking the 
maximum running time over $[0;1]$, a compact set
\mycite{Theorem~7.2.7}{Weihrauch}.
On non-compact domains, on the other hand, the running time will in 
general admit no bound depending on the output precision $n$ only;
recall Example~\ref{x:Param1}a) and see \mycite{Exercise~7.2.10}{Weihrauch}. 

The qualitative topological condition of continuity is a 
prerequisite to computability of a function. Its quantitative
refinement capturing sensitivity on the other hand corresponds
to a lower bound on the complexity \mycite{Theorem~2.19}{Ko91}:

\begin{fact} \lab{f:Modulus}
If $f:[0;1]\to\IR$ is computable in the sense of (b3)
within time $t(n)$, then the function 
$\mu:\IN\ni n\mapsto t(n+2)\in\IN$
constitutes a \textsf{modulus of continuity} to $f$
in the sense that the following holds:
\begin{equation} \label{e:Modulus}
\forall x,y\in\dom(f): \quad |x-y|\leq2^{-\mu(n)} \;\Rightarrow\;
|f(x)-f(y)|\leq2^{-n} \enspace . 
\end{equation}
In particular every polytime computable function
necessarily has a polynomial modulus of continuity; 
and, conversely, each $f:[0;1]\to\IR$ satisfying 
Equation~(\ref{e:Modulus}) for polynomial $\mu$
is polytime computable relative to some oracle.
\end{fact}
The \textsf{Time Hierarchy Theorem} in classical complexity theory
constructs a set $E\subseteq\IN$ decidable in exponential but not in
polynomial time. Appropriately encoded, this in turn yields 
a real number $y[E]$ of similar complexity
\mycite{Lemma~7.2.9}{Weihrauch}; and the same holds
for the constant function $[0;1]\ni x\mapsto y[E]$
\mycite{Corollary~7.2.10}{Weihrauch}:
computable within exponential but not in polynomial time.
We remark that, again and as opposed to 
this recursion theoretic construction, 
also quantitative topological reasons
lead to an (even explicitly given) function
with such properties; cmp. Figure~\ref{f:superpol}b):

\begin{myexample} \lab{x:Modulus}
The following function is computable in exponential
time, but not in polynomial time --- and oracles do not help:
\[ f:(0;1]\;\ni\; x\;\mapsto\; 1/\ln(e/x)\;\in\;(0;1], \qquad
f(0)=0 \enspace . \]
\end{myexample}
\begin{proof}
Note that monotone $f$ is
obviously computable on $(0;1]$, 
and continuous in 0;
in fact effectively continuous but 
with exponential modulus of continuity.
More precisely, $f(2^{-m})=1/(1+m\cdot\ln2)\overset{!}{\leq}2^{-n}
\Leftrightarrow m\geq(2^n-1)/\ln2$.
Therefore arguments $x=2^{-m}$ must be known from zero
(i.e. and with precision at least $m$ exponential in $n$)
in order to assert the value $f(x)$ 
different from $f(0)=0$ up to error $2^{-n}$.
\qed\end{proof}

\begin{figure}[htb]
\begin{minipage}[c]{0.52\textwidth}
\includegraphics[width=0.99\textwidth]{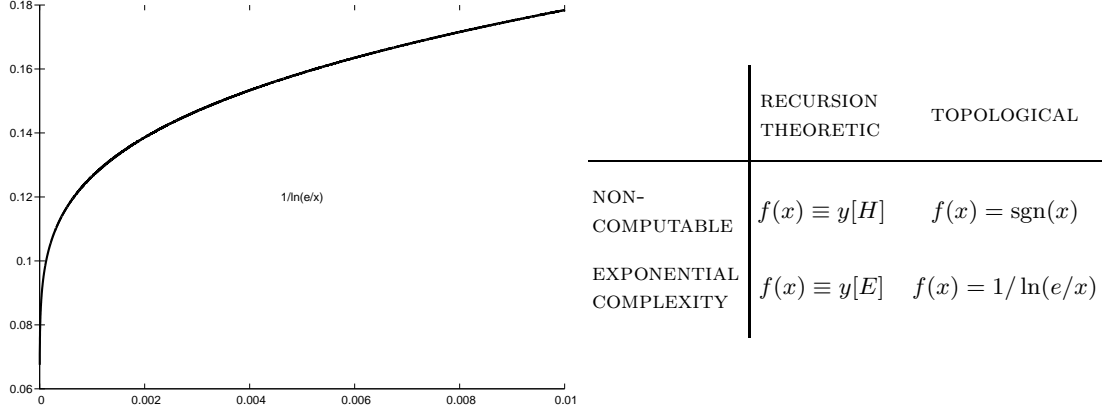}
\end{minipage}%
\begin{minipage}[c]{0.43\textwidth}
\sc
\begin{tabular}{c|@{\;}c@{\quad}c}
\\
 & \begin{minipage}[c]{12ex}
recursion theoretic \end{minipage} 
& topological   \\[3ex]
\hline \\
\begin{minipage}[c]{15ex}
non-computable 
\end{minipage}
&  $f(x)\equiv y[H]$ &  $f(x)=\sgn(x)$  \\[4ex]
\noindent
\begin{minipage}[c]{15ex}
\noindent
exponential complexity \end{minipage} 
&   $f(x)\equiv y[E]$ &  $f(x)=1/\ln(e/x)$ \\ ~
\end{tabular}
\end{minipage}
\caption{\label{f:superpol}a)~The graph of $f(x)=1/\ln(e/x)$ 
from Example~\ref{x:Modulus}
demonstrating its exponential rise from 0. \newline
b)~Lower bound techniques in real function computation;
$H\subseteq\IN$ is the Halting problem
and $\IN\supseteq E\in\calEXP\setminus\calP$.}
\end{figure}
Concerning computational complexity on spaces other than $\IR$, consider

\begin{mydefinition} \lab{d:MetricComplex}
\begin{enumerate}
\item[a)] A partial function $F:\subseteq\{\sdzero,\sdone\}^\omega\to\{\sdzero,\sdone\}^\omega$
is \emph{computable in time} $t:\IN\to\IN$ if a Type-2 Machine can,
given $\bar\sigma\in\dom(F)$, produce $\bar\tau=F(\bar\sigma)$
such that the $n$-th symbol of $\bar\tau$ appears within
$t(n)$ steps.
\item[b)] For representations $\xi$ of $X$ and $\upsilon$ of $Y$,
a (possibly partial and multivalued) function $f:\subseteq X\toto Y$
is \emph{computable in time} $t:\IN\to\IN$ if it has a 
$(\xi,\upsilon)$--realizer $F$ according to a).
\end{enumerate}
\end{mydefinition}
\begin{figure}[htb]
\begin{minipage}[c]{0.44\textwidth}
\includegraphics[width=0.97\textwidth,height=0.33\textheight]{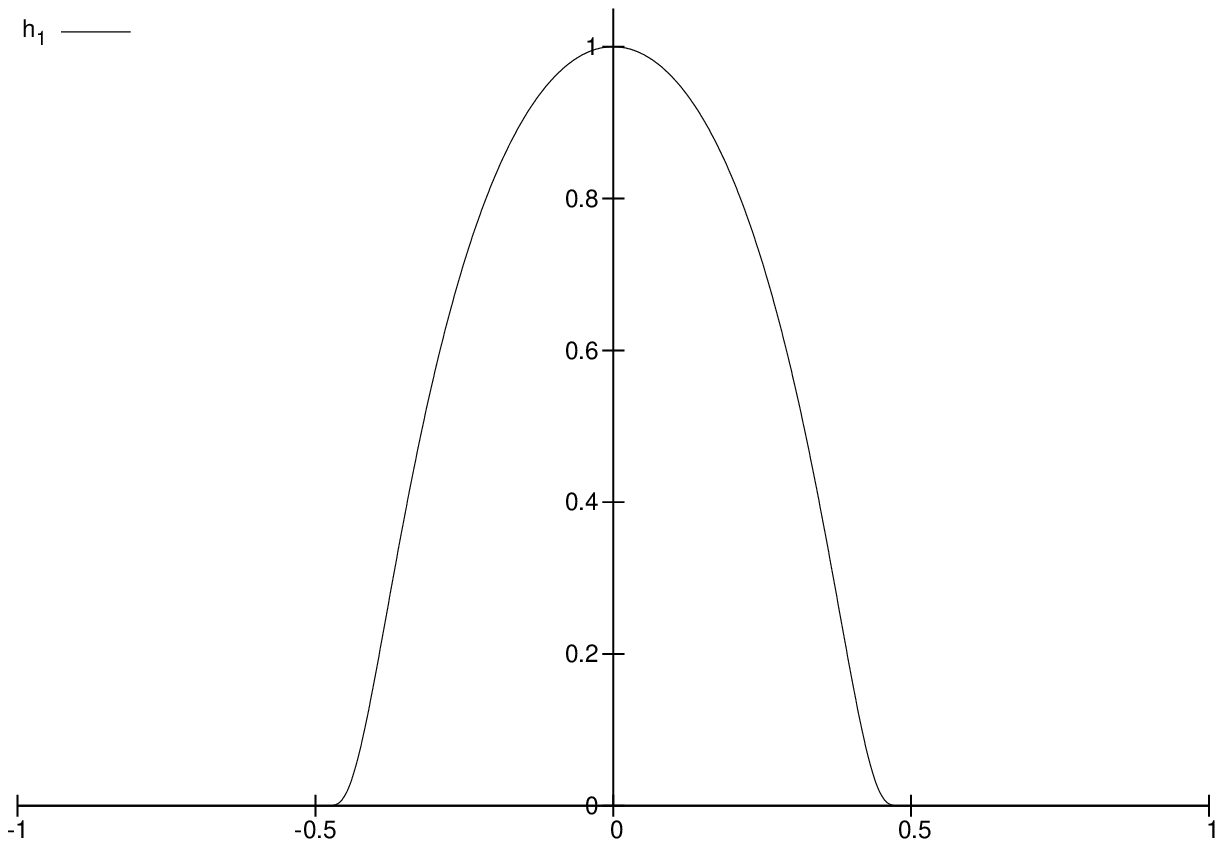}
\end{minipage} \hfill
\begin{minipage}[c]{0.55\textwidth}
\includegraphics[width=0.98\textwidth]{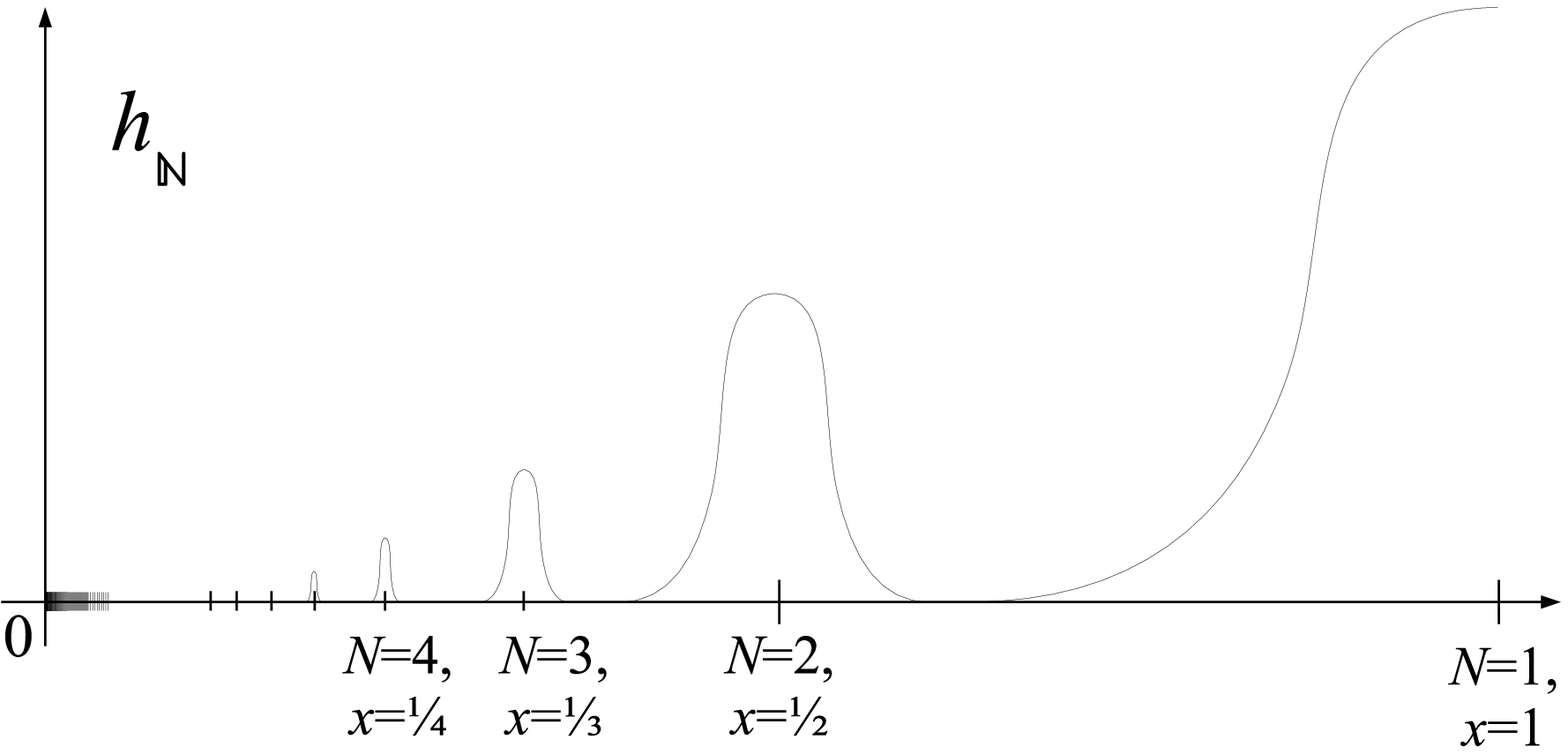} \\[2ex]
\includegraphics[width=0.98\textwidth]{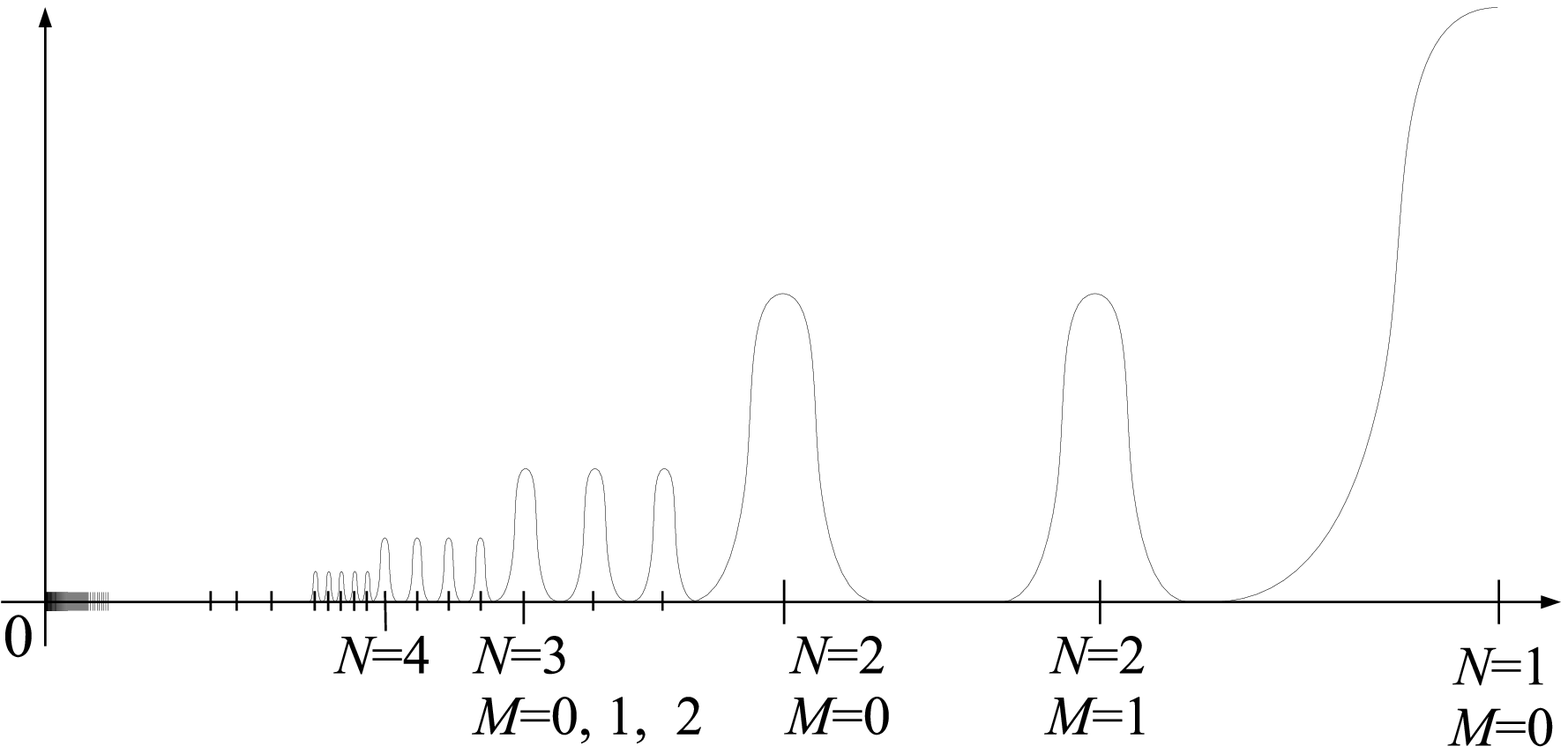}
\end{minipage}
\caption{\label{f:maxNP}
Encoding decision problems into
smooth functions: Illustrating the proof of Example~\ref{x:maxNP}.
}
\end{figure}
As already pointed out, this notion of complexity may be trivial
for some $\xi$ and $\upsilon$. (Meta-)Conditions on the representations
under consideration are devised in \cite{WeihrauchComplexity}.
We conclude this subsection by recalling the construction 
underlying \mycite{Theorem~3.7}{Ko91}
of a polytime computable smooth $h:[0;1]\to[0;1]$ 
for which the function $\Max(h):[0;1]\to[0;1]$ is not polytime
unless $\calP=\calNP$:

\begin{myexample} \lab{x:maxNP}
\begin{enumerate}
\item[a)]
Recall from Example~\ref{x:Upper} that the function 
\[  h_1(x)\;:=\;\exp\big(\tfrac{4x^2}{4x^2-1}\big) \quad\text{ for }|x|\leq\tfrac{1}{2},
\qquad h_1(x):=0 \quad\text{ for }|x|\geq\tfrac{1}{2} \]
is smooth on $[-1;1]$ and vanishes outside $\big[-\tfrac{1}{2};\tfrac{1}{2}\big]$.
Similarly, the scaled and shifted
$h_{N}(x):=h(2N^2\cdot x-2N)/N^{\ln N}$ is smooth 
with maximum $\exp(-\ln^2 N)$ attained at $x=1/N$
and vanishes outside 
$\big[\tfrac{1}{N}-\tfrac{1}{4N^2};\tfrac{1}{N}+\tfrac{1}{4N^2}\big]$.
Note that $\tfrac{1}{N}-\tfrac{1}{4N^2}\geq\tfrac{1}{N+1}+\tfrac{1}{4(N+1)^2}$ 
holds for $N\in\IN$, hence
$h_{N}$ has support disjoint to that of $h_{N'}$ 
whenever $N,N'\in\IN$ are distinct.
Moreover, for every $j\in\IN_0$, 
$h_{N}^{(j)}(x)=h_1^{(j)}(2N^2\cdot x-2N)\cdot N^{2j}/N^{\ln N}\to0$
uniformly in $x$ as $N\to\infty$.
Therefore (and abusing notation)
the function $h_{L}:=\sum_{N\in L} h_{N}:[0;1]\to[0;1]$
is smooth for every $L\subseteq\IN$; 
and the question ``$N\in L$?'' reduces to
evaluating $h_{L}(1/N)$ up to absolute error $2^{-n}$
for $n:=(\ln N)^2\cdot\log e$ polynomial in the length of $N$:
Modulo polytime, the computational complexity of $h_{L}$ 
coincides with that of $L$ as a discrete binary decision problem.
\item[b)]
We argue that also the computational complexity of $\Max(h_L)$ 
coincides (modulo polytime) with that of $L$.
To this end note that, for $x\leq 1$,
$\exp(x/2)\leq 1+x$ and thus $\ln(1+x)\geq x/2$:
\begin{eqnarray*} \ln^2(N+1)\ln^2(N)
&=& \big(\ln(N+1)+\ln(N)\big)\cdot\ln\big(1+\tfrac{1}{N}\big)
\;\geq\; 2\ln(N)/(2N)\geq 1/N \enspace , \\[1ex]
e^{-\ln^2(N)}-e^{-\ln^2(N+1)}
&=& e^{-\ln^2N}\cdot\Big(1-e^{-\big(\ln^2(N+1)-\ln^2(N)\big)}\Big)  \\
&\geq& e^{-\ln^2N}\cdot\big(1-e^{1/N}\big) \;\geq\; e^{-\ln^2N-\ln N-1}
\end{eqnarray*}
because $1-e^{-1/N}\geq\tfrac{1}{eN}=e^{-\ln N-1}$.
Note observe that 
$\Max\big(h_L\big)(1/N)=\exp\big(-\ln^2(N)\big)$ holds if $N\in L$,
whereas $\Max\big(h_L\big)(1/N)\leq\exp\big(-\ln^2(N+1)\big)$ if $N\not\in L$.
And the above estimates demonstrate that both cases can be distinguished
by evaluating $\Max\big(h_L\big)(1/N)$ up to absolute error $2^{-n}$
for $n:=\calO(\ln^2N)$ polynomial in the binary length of $N$.
\item[c)]
Slightly more generally consider 
\[
L \;\subseteq\; \{ (N,M) : N\in\IN, \IN_0\ni M<N \}, \qquad
K \;:=\; \{ N : \exists M: (N,M)\in L \}  \]
and observe that the condition $M<N$ on integer values
implies the binary length of $M$ to be bounded by that of $N$
from which in turn follows $M<2N$; hence this condition
can be met by `padding' $N$ with polynomially many dummy digits
in order to obtain a $\calNP$--complete $K$ with $L\in\calP$.
Now let 
\[
h_{N,M}(x)\;:=\;h_1(3N^3\cdot x-3N^2-M)/N^{\ln N}, \qquad
h_L:=\sum\nolimits_{(N,M)\in L} h_{N,M}  \]
and note as in a) that $h_{N,M}$ is smooth with maximum $\exp(-\ln^2 N)$
attained at $x_{N,M}:=\tfrac{1}{N}+\tfrac{M}{3N^3}$ and vanishes outside
$\Big[\frac{1}{N}+\frac{M-1/2}{3N^3};\frac{1}{N}+\frac{M+1/2}{3N^3}\Big]$.
Moreover $\tfrac{1}{N}-\tfrac{1}{6N^3}\geq\tfrac{1}{N+1}+\tfrac{N}{3(N+1)^3}+\tfrac{1}{6(N+1)^3}$
shows that $h_{N,M}$ has support disjoint to that of $h_{N',M'}$ for 
distinct $(N,M),(N',M')$ with $M<N$ and $M'<N'$.
Hence $h_L$ is polytime computable iff $L\in\calP$.
On the other hand it holds
$\Max\big(h_L\big)(x_{N,N-1})=\exp\big(-\ln^2(N)\big)$ in case $N\in K$
and $\Max\big(h_L\big)(x_{N,N-1})\leq\exp\big(-\ln^2(N+1)\big)$ in case $N\not\in K$:
both distinguishable by evaluating $\big(h_L\big)(x_{N,N-1})$ up to absolute
error polynomial in the binary length of $N$.
Therefore polytime evaluation of $\Max(h_L)$ implies $K\in\calP=\calNP$.
\end{enumerate}
\end{myexample}

\subsection{Uniformly Computing Real Operators} \lab{s:Oracle} 
In common terminology, an operator maps functions to 
functions; and a functional maps (real) functions to (real) numbers.
In order to define computability (Remark~\ref{r:Uniformity}c)
and complexity (Definition~\ref{d:MetricComplex}b),
it suffices to choose some representation of
the function space to operate on: 
for integration and maximization 
that is the set $C[0;1]$ of continuous $f:[0;1]\to\IR$.
In view of Section~\ref{s:FuncComplex}, 
(b1) provides a reasonable such encoding 
(which we choose to call $[\rhody\myto\rhody]$):
concerning computability, but not for the refined view of complexity.
In fact any uniform complexity theory of real operators is faced with 

\begin{myproblem}  \lab{p:Problem}
\begin{enumerate}
\item[a)] Even restricted to $C([0;1],[0;1])$,
 that is to continuous functions $f:[0;1]\to[0;1]$,
 the maps $f:[0;1]\ni x\mapsto f(x)\in[0;1]$
  are not computable within time uniformly bounded 
  in terms on the output error $2^{-n}$ only.
\item[b)] Restricted further to the subspace $\Lip_1([0;1],[0;1])$
  of non-expansive such functions
  (i.e. additionally satisfying $|f(x)-f(y)|\leq|x-y|$), there is no
  representation (i.e. a surjection $\delta$ defined on some subset of $\{\sdzero,\sdone\}^\omega$) 
  rendering evaluation or integration computable in uniform time subexponential in $n$.
\end{enumerate}
\end{myproblem}

\begin{figure}[htb]
\includegraphics[width=0.95\textwidth]{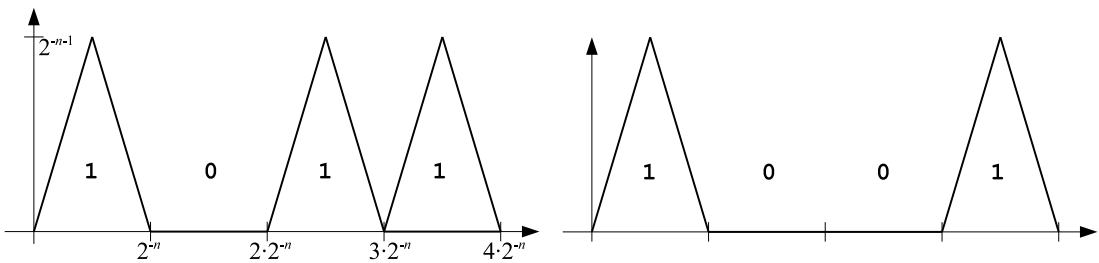}%
\caption{\label{f:manyfunc}Encoding the binary strings \texttt{1011}
and \texttt{1001} into 1-Lipschitz functions $f,g$
with $\|f-g\|$ not too small}
\end{figure}

\begin{proof}
\begin{enumerate}
\item[a)] follows from Fact~\ref{f:Modulus}
since there exist arbitrarily
`steep' continuous functions $f:[0;1]\to[0;1]$.
\item[b)]
Technically speaking, the so-called \emph{width}
of this function space is exponential 
\mycite{Example~6.10}{WeihrauchComplexity}.
More elementarily,
Figure~\ref{f:manyfunc} demonstrates how to encode 
each $\vec b\in\{\sdzero,\sdone\}^n$ into a 
non-expansive function $f_{\vec b}:[0;1]\to[0;1]$
with $\|f_{\vec b}-f_{\vec a}\|\geq2^{-n}$ for $\vec b\neq\vec a$.
These $2^{2^n}$ functions thus can be told apart
by evaluation up to error $2^{-n-2}$ at arguments $(k+1/2)\cdot 2^{-n}$,
$0\leq k<2^n$; alternatively by integrating from
$k\cdot 2^{-n}$ to $(k+1)\cdot2^{-n}$.
Any algorithm doing so within time $\leq t(n)$ can read at most the
first $N:=t(n)$ bits of its given $\delta$--name. However there are
only $2^N$ different initial segments of putative such names.
It thus follows $t(n)\geq2^{\Omega(n)}$.
\qed\end{enumerate}\end{proof}
\mycite{Definition~2.37}{Ko91} solves Problem~\ref{p:Problem}a) 
by permitting the running time
to depend polynomially on both $n$ and $f$'s modulus of continuity
as a second (order) parameter, similar to Example~\ref{x:Param1}b).
We provide an alternative perspective to this resolution
in Subsection~\ref{s:Parameterized} below.

\subsection{Type-3 Complexity Theory} \lab{s:SecondOrder}
Concerning Problem~\ref{p:Problem}b), encodings as strings with sequential access
seem a restriction compared to actual computers random access memory
and subroutine calls providing on-demand information:
Given (an algorithm computing) some $f$,
this permits to approximate the value $f(a)$ for $a\in\ID_{n+1}$ 
up to error $2^{-n}$ with\emph{out} having to (calculate and) 
skip over all order $2^n$ values $f(a')$, $a'\in\ID_{n'}$, $n'<n$.
This can be taken into account by modelling access to $f$ via
oracles that, given $a\in\ID_{n+1}$, report $b\in\ID_{n+1}$
with $|b-f(a)|\leq2^{-n}$; equivalently: 
using not infinite strings (i.e. total mappings $\{1\}^*\to\{\sdzero,\sdone\}$)
but string functions (i.e. total mappings $\{\sdzero,\sdone\}^*\to\{\sdzero,\sdone\}^*$)
to encode $f$. And indeed \cite{AkiSTOC} resolves Problem~\ref{p:Problem}b) 
by extending TTE (Remark~\ref{r:TTE}) and generalizing Cantor space 
as the domain of a representation to (a certain subset of) Baire space:

\begin{mydefinition} \lab{d:Representation2}
\begin{enumerate}
\item[a)]
Let $\Reg$ denote the set of all total functions
$\psi:\{\sdzero,\sdone\}^*\to\{\sdzero,\sdone\}^*$
length-monotone in the sense of verifying
\begin{equation} \label{e:Regular}
|\vec v|\;\leq\;|\vec w|\quad\Rightarrow\quad|\psi(\vec v)|\;\leq\;|\psi(\vec w)| \enspace .
\end{equation}
Write $|\psi|:\IN\to\IN$ for the (thus well-defined)
mapping $|\vec w|\mapsto|\psi(\vec w)|$. 
\item[b)]
A \text{second-order representation} for a space $X$ is a
surjective partial mapping $\tilde\xi:\subseteq\Reg\to X$.
\item[c)]
Any ordinary representation $\xi:\subseteq\{\sdzero,\sdone\}^\omega\to X$ 
\text{induces} a second-order representation $\tilde\xi$ as follows:
Whenever $\bar\sigma$ is a $\xi$--name of $x$, then
$\psi:\{\sdzero,\sdone\}^*\ni\vec v\mapsto\sigma_{|\vec v|}\in\{\sdzero,\sdone\}$ 
is a $\tilde\xi$--name of said $x$.
\item[d)]
Let $\tilde\xi:\subseteq\Reg\to X$ and
$\tilde\upsilon:\subseteq\Reg\to Y$ be second-order representations.
The \text{product} $\tilde\xi\times\tilde\upsilon$ is the second-order representation
whose names for $(x,y)\in X\times Y$ consist of mappings 
\[
\sdzero\,\vec w \;\mapsto\; \sdzero^{m-|\psi_0(\vec w)|}\,\sdone\,\psi_0(\vec w), \qquad
\sdone\,\vec w \;\mapsto\; \sdzero^{m-|\psi_1(\vec w)|}\,\sdone\,\psi_1(\vec w) 
\]
where $\psi_0$ denote a $\tilde\xi$--name of $x$ 
and $\psi_1$ a $\tilde\upsilon$--name of $y$  and
$m:=\max\{|\psi_0(\vec w)|,|\psi_1(\vec w)|\}$.
\item[e)]
More generally, fix a injective linear-time bi-computable length-monotone mapping
$\langle\,\cdot\,,\,\cdot\,\rangle:
\{\sdzero,\sdone\}^*\times\{\sdzero,\sdone\}^*\to\{\sdzero,\sdone\}^*$
and an arbitrary index set $A\subseteq\{\sdzero,\sdone\}^*$
as well as second-order representations $\tilde\xi_{\vec a}:\subseteq\Reg\to X_{\vec a}$,
$\vec a\in A$.
The product $\prod_{\vec a\in A}\tilde\xi_{\vec a}$
is the second-order representation of $\prod_{\vec a\in A} X_{\vec a}$
whose names $\psi$ of $(x_{\vec a})_{_{\vec a\in A}}$ satisfy that 
$\{\sdzero,\sdone\}^*\ni\vec v\mapsto\psi(\langle\vec a,\vec v\rangle)\in\{\sdzero,\sdone\}^*$
constitutes a $\tilde\xi_{\vec a}$--name of $x_{\vec a}$
for each $\vec a\in A$, padded with some initial $\sdzero^*\,\sdone$
to attain common length $|\psi|(n)$ on all arguments 
of length $n=|\langle\vec a,\vec v\rangle|$.
\item[f)]
An oracle Type-2 Machine $\calM^\psi$ may write onto its query tape 
some $\vec w\in\{\sdzero,\sdone\}^*$ which,
when entered the designated query state, will be replaced
with $\vec v:=\psi(\vec w)$. \\ (We implicitly employ some
linear-time bicomputable self-delimited encoding on this tape such as
$(w_1,\ldots,w_n)\mapsto\sdone\,w_1\,\sdone\,w_2\,\ldots\sdone\,w_n\,\sdzero$.)
\item[g)]
$\calM^?$ \textsf{computes} a partial mapping
$\tilde F:\subseteq\{\sdzero,\sdone\}^\omega\times\Reg\to\{\sdzero,\sdone\}^\omega$ 
if, for every $(\bar\sigma,\psi)\in\dom(\tilde F)$,
$\calM^{\psi}$ on input $\bar\sigma$ produces
$\tilde F(\bar\sigma,\psi)$. \\
$\calM^?$ \textsf{computes} a partial mapping
$\tilde G:\subseteq\{\sdzero,\sdone\}^\omega\times\Reg\to\Reg$ if, 
for every $(\bar\sigma,\psi)\in\dom(\tilde G)$ 
and every $\vec v\in\{\sdzero,\sdone\}^*$, 
$\calM^{\psi}$ on input $(\bar\sigma,\vec v)$, $\vec v\in\{\sdzero,\sdone\}^*$,
produces $\tilde G\big(\bar\sigma,\psi\big)(\vec v)\in\{\sdzero,\sdone\}^*$;
cmp. Figure~\ref{f:TTE2}a). 
\item[h)]
For ordinary representation $\xi$ of $X$
and second-order representation $\tilde\upsilon$ of $Y$
and ordinary representation $\zeta$ of $Z$,
$\tilde F:\subseteq\{\sdzero,\sdone\}^\omega\times\Reg\to\{\sdzero,\sdone\}^\omega$ 
is a $(\xi,\tilde\upsilon,\zeta)$--realizer
of a (possibly partial and multivalued)
function $f:\subseteq X\times Y\toto Z$ ~iff,
for every $(x,y)\in\dom(f)$ and every 
$\xi$--name $\bar\sigma$ of $x$
and every $\tilde\upsilon$--name of $y$,
$\tilde F(\psi)$ is a $\zeta$--name of some $z\in f(x,y)$. \\
For second-order representation $\tilde\zeta$ of $Z$,
$\tilde G:\subseteq\{\sdzero,\sdone\}^\omega\times\Reg\to\Reg$ 
is a $(\xi,\tilde\upsilon,\tilde\zeta)$--realizer of $f$ if,
for every $(x,y)\in\dom(f)$ and every 
$\xi$--name $\bar\sigma$ of $x$
and every $\tilde\upsilon$--name of $y$,
$\tilde G(\psi)$ is a $\tilde\zeta$--name of some $z\in f(x,y)$.
\end{enumerate}
\end{mydefinition}
From a mere computability point of view, each representation on
Baire space is equivalent to one on Cantor space
\mycite{Exercise~3.2.17}{Weihrauch}. Concerning complexity,
an approximation of $\psi\in\Reg$ up to error $2^{-n}$
naturally consists of the restriction $\psi|_{\{0,1\}^n\}}$;
cmp. \mycite{Definition~4.1}{Kapron}.
Recall that any $\psi:\{\sdzero,\sdone\}^*\to\{\sdzero,\sdone\}^*$ 
may w.l.o.g. be presumed self-delimiting and then
redefined such as to satisfy Equation~(\ref{e:Regular})
by appropriately `padding'. We require every $\tilde\xi$--name $\psi$ 
to be total, but often the information on $x$ is contained in some
restriction of $\psi$.
Note that an oracle query $\vec w\mapsto\vec v:=\psi(\vec w)$
according to Definition~\ref{d:Representation2}f) may return
a (much) longer answer for some argument $\psi$ than for some 
other $\psi'$; so in order to be able to even read such a reply
for some fixed $n$, the permitted running time bound should 
involve both $n$ and $|\psi|$. Since only the former is a number,
this naturally involves a concept already introduced in  
\cite{Mehlhorn}:

\begin{mydefinition} \lab{d:Poly2}
\begin{enumerate}
\item[a)]
A \emph{second-order} polynomial $P=P(n,\ell)$
is a term composed from variable symbol $n$,
unary function symbol $\ell()$, 
binary function symbols $+$ and $\times$, and positive integer constants.
\item[b)]
Let $T:\IN\times\IN^\IN\to\IN$ be arbitrary.
Oracle machine $\calM^?$ computing 
$\tilde F:\subseteq\{\sdzero,\sdone\}^\omega\times\Reg\to\{\sdzero,\sdone\}^\omega$ 
according to Definition~\ref{d:Representation2}f+g)
operates in \emph{time $T$} if, 
for every $(\bar\sigma,\psi)\in\dom(\tilde F)$ and every $n\in\IN$,
$\calM^{\psi}$ on input $\bar\sigma$ produces
the $n$-th output symbol within at most $T(n,|\psi|)$ steps. \\
$\calM^?$ computing $\tilde G:\subseteq\{\sdzero,\sdone\}^\omega\times\Reg\to\Reg$
operates in \emph{time $T$} if, 
for every $(\bar\sigma,\psi)\in\dom(\tilde G)$,
$\calM^{\psi}$ on input $\vec v\in\{\sdzero,\sdone\}^n$
makes at most $T(n,|\psi|)$ steps.
\item[c)]
For ordinary representations $\xi$ of $X$ and $\zeta$ of $Z$
as well as second-order representations $\tilde\upsilon$ of $Y$
and $\tilde\zeta$ of $Z$, a (possibly partial and multivalued)
function $f:\subseteq X\times Y\toto Z$ is $(\xi,\tilde\upsilon,\zeta)$--computable
in time $T$ ~iff~ $f$ has a $(\xi,\tilde\upsilon,\zeta)$--realizer
$\tilde F$ computable in this time; \\
similarly for $(\xi,\tilde\upsilon,\tilde\zeta)$--computability.
\item[d)]
\emph{Second-order polytime} computability means computability
in time $P$ for some second-order polynomial $P$.
\item[e)]
An ordinary representation $\xi$ of $X$
\emph{polytime reduces} to an ordinary or second-order
representation $\tilde\upsilon$ of $X$
(written $\tilde\xi\reducep\tilde\xi$)
if the identity $\id:X\to X$ is 
$(\tilde\xi,\tilde\upsilon)$--computable within (first-order) polytime. \\
A second-order representation $\tilde\xi$ of $X$
(second-order) \emph{polytime reduces} to an ordinary or second-order
representation $\tilde\upsilon$ of $X$
(written $\tilde\xi\reduceP\tilde\xi$)
if the identity $\id:X\to X$ is 
$(\tilde\xi,\tilde\upsilon)$--computable in second-order polytime. \\
\emph{(Second-order) polytime equivalence} means that 
also the converse(s) hold(s).
\end{enumerate}
\end{mydefinition}
`Long' arguments $\psi$ are thus granted more time to operate on;
see Remark~\ref{r:FunctionAsPoint} below.
Every ordinary (i.e. first-order) bivariate polynomial $p(n,k)$
obviously can be seen as a second-order polynomial 
by identifying $k\in\IN$ with the constant function $\ell(n)\equiv k$;
but not conversely, as illustrated with the example
$n\cdot\ell\Big(n^5+\big(\ell(n^2)\big)^2\Big)$.
Second-order polynomials also constitute the second-level of a 
Grzegorczyk Hierachy on functionals of finite type arising
naturally as bounds in proof-mining \cite[middle of p.33]{PolyProofs};
cmp. also \mycite{\S A}{ProofMining} and \cite{Tent}.
In fact on arguments $\ell$ polynomial of unbounded degree, 
only second-order polynomials satisfy closure under 
both kinds of composition:
\begin{equation} \label{e:Composition2}
\big(Q\circ P\big)(n,\ell) \;\; :=\;\;
Q\big(P(n,\ell),\ell\big) \quad\text{ and }\quad
\big(Q\bullet P\big)(n,\ell) \;\;:=\;\;
Q\big(n,P(\cdot,\ell)\big) \enspace .
\end{equation}
Second-order representations induced by ordinary ones on
the other hand basically use only string functions $\psi$ 
on unary arguments with binary values, i.e. of length 
$|\psi|\equiv1$; and these indeed recover first-order TTE complexity:

\begin{observation} \lab{o:SecondVsFirst}
\begin{enumerate}
\item[a)]
Any bivariate ordinary polynomial $p(n,m)$ can be
bounded by some univariate polynomial in $n+m$. 
Any bivariate second-order polynomial $P\big(n,\ell(),k()\big)$
--- that is a term composed from $n$, $\ell()$, $k()$, $+$, $\times$, and $1$ ---
can be bounded by some $Q\big(n,\ell()+k()\big)$.
\item[b)]
Let $P=P(n,\ell)$ denote a second-order polynomial
and $\ell$ a monic linear function with offset $k$,
i.e., $\ell(n)=n+k\in\IN$.
Then $P(n,\ell)$ boils down to an ordinary bivariate polynomial
in $n$ and $k$.
In particular, $P(0,\equiv k)$ is a polynomial in $k$.
\item[c)]
More generally, fix $d\in\IN$ and consider 
the module $\IN_d$ of polynomials 
$\ell$ over $\IN$ of degree less than $d$:
$\ell(n)=\sum_{j<d} k_j n^j$. \\ Subject to this restriction,
every second-order polynomial $P=P(n,\ell)$ can be bounded
by a univariate polynomial in $n+k_0+\cdots+k_{d-1}$.
\item[d)]
Let $\xi$ denote an ordinary representation of $X$
and $\tilde\xi$ its induced second-order representation. 
Then $\xi$ reduces to $\tilde\xi$ within (first-order) polytime:
$\xi\reducep\tilde\xi$; \\
and $\tilde\xi$ reduces to $\xi$ within first-order polytime:
$\tilde\xi\reduceP\xi$.
\item[e)] 
Let $\xi$ and $\upsilon$ denote ordinary representations of $X$ and $Y$
with induced second-order representations $\tilde\xi$ and $\tilde\upsilon$, respectively.
Then the second-order representation $\widetilde{\xi\times\upsilon}$ of $X\times Y$
induced by $\xi\times\upsilon$ is polytime equivalent to $\tilde\xi\times\tilde\upsilon$.
\item[f)]
For each $j\in\IN$ let $\xi_j$ denote an ordinary representation of $X_j$
and $\tilde\xi_{\sdone^j}$ its induced second-order representation.
Then the second-order representation $\widetilde{\prod_j\xi_j}$ of $\prod_j X_j$
is polytime equivalent to $\prod_{j}\tilde\xi_{\sdone^j}$, that is,
with respect to indices encoded in unary.
\end{enumerate}
\end{observation}
\begin{proof}
a+b+c) are immediate. 
\begin{enumerate}
\item[d)]
Given $\bar\sigma$ and $\vec v$, it is easy to return $\sigma_{|\vec v|}$,
thus $(\xi,\tilde\xi)$--computing $\id_X$ within polynomial time; \\
similarly for the converse, observing that $\tilde\xi$--names $\psi$
have constant length, see b).
\item[e)]
Observe that a $\widetilde{\xi\times\upsilon}$--name $\psi$ of $(x,y)$ has
$|\psi|\equiv1$ and $\xi\big(\psi(\sdone^{2n})_{_n}\big)=x$ 
and $\upsilon\big(\psi(\sdone^{2n+1})_{_n}\big)=y$;
while a $\tilde\xi\times\tilde\upsilon$--name $\psi$ has 
$|\psi|\equiv1$ and
$\xi\big(\psi(\sdzero\,\sdone^n)_{_n}\big)=x$ and
$\upsilon\big(\psi(\sdone\,\sdone^n)_{_n}\big)=y$.
\item[f)]
A $\widetilde{\prod_j\xi_j}$--name of $(x_j)$ 
has $|\psi|\equiv1$ and 
$\xi_j\big(\psi(\sdone^{\langle n,j\rangle})_{_n}\big)=x_j$;
while a $\big(\prod_{j}\tilde\xi_{\sdone^j}\big)$--name $\psi$ 
has $|\psi|\equiv1$ and 
$\xi_j\big(\psi(\langle\sdone^j,\sdone^n\rangle)_{_n}\big)=x_j$
\qed\end{enumerate}\end{proof}
As with ordinary representations, not every choice of $\tilde\xi$ 
and $\tilde\upsilon$ leads to a sensible notion of complexity.
Concerning the case of continuous functions on Cantor space
and on the real interval $[0;1]$, we record and
report from \mycite{\S4.3}{AkiSTOC}:

\begin{myexample} \lab{x:Representation2}
\begin{enumerate}
\item[a)]
Define a second-order representation 
$\dyrho:\subseteq\Baire\to C[0;1]$ as follows: \\
Length-monotone $\psi:\{\sdzero,\sdone\}^*\to\{\sdzero,\sdone\}^*$ 
is a $\dyrho$--name of bounded partial $f:\subseteq[0;1]\to\IR$ 
\begin{multline*}
\text{iff} \quad \forall\vec w\in\{\sdzero,\sdone\}^*: \qquad \Big(
\dom(f)\cap\ball\big(\bin(\vec w)/2^{|\vec w|+1},2^{-|\vec w|}\big)\neq\emptyset
\quad \Rightarrow \\
\exists x\in\dom(f)\cap\ball\big(\bin(\vec w)/2^{|\vec w|+1},2^{-|\vec w|}\big): \;
\big|\bin\big(\psi(\vec w)\big)/2^{|\vec w|+1}-f(x)\big|\leq 2^{-|\vec w|} 
\Big) \enspace . \end{multline*}
It holds $n+\log\|f\|_\infty\leq|\psi|(n)\leq C_\psi+n$
for all $n\in\IN$ and some $C_{\psi}\in\IN$.
\item[b)]
For $\ell\geq0$, let $\Lip_\ell[0;1]:=\big\{f:[0;1]\to\IR, |f(x)-f(y)|\leq \ell\cdot|x-y|\big\}$
and $\Lip[0;1]:=\bigcup_\ell\Lip_\ell[0;1]$ the class of (globally) Lipschitz-continuous functions.
Define a $\big(\dyrhoLip\big)$--name of $f\in\Lip[0;1]$ to be a mapping
$\{\sdzero,\sdone\}^*\ni\vec w\mapsto\langle\bin(\ell),\psi(\vec w)\rangle\in\{\sdzero,\sdone\}^*$,
where $\psi$ denotes a $\dyrho$--name of $f$ 
and $\ell\in\IN$ some Lipschitz constant to it. 
It thus holds $n+\log\|f\|_\infty+\log\ell\leq|\psi|(n)\leq n+C'_{\psi}$.
\item[c)]
Define a $\tilderhorho$--name 
(in {\rm\cite{AkiSTOC}} called a $\deltabox$--name)
$\psi$ of $f\in C[0;1]$ to be a mapping
$\{\sdzero,\sdone\}^*\;\ni\;\vec w
\;\mapsto\;\sdone^{\mu(|\vec w|}\,\sdzero\,\psi(\vec w)\;\in\{\sdzero,\sdone\}^*$,
where $\psi$ denotes a $\dyrho$--name of $f$
and $\mu:\IN\to\IN$ is a modulus of uniform continuity to it.
It thus holds $\mu(n)+n+\log\|f\|_\infty\leq|\psi|(n)\leq n+C''_{\psi}$.
\item[d)]
Recall that the identity $\id$ is the standard representation of
Cantor space. Inspired by \mycite{Definition~3.6.1}{Roesnick},
consider the following 
second-order (multi\footnote{Since we choose not to include
a description of the functions' domains, a name of some $F$ 
constitutes also one of every restriction of $F$.}-)
representation of the space 
$\UC(\subseteq\{\sdzero,\sdone\}^\omega\to\{\sdzero,\sdone\}^\omega)$
of uniformly continuous partial functions on Cantor space: \\
A $\tilde\eta$--name of $F$ maps 
$\sdone^n\,\sdzero\,\vec x\in\{\sdzero,\sdone\}^*$
to some $\sdone^m\,\sdzero\,\vec y$ such that
$d\big(F(\bar x),\vec y\big)\leq 2^{-n}$
for all $\bar x\in\dom(F)$ with $d(\bar x,\vec x)\leq 2^{-m}$.
Here $d(\bar x,\vec y)=\max\{2^{-j}:x_j\neq y_j\;\vee\; j>|\vec y|\}$ 
induces the the metric on Cantor space. 
\item[e)]
$\tilderhorho$ is computably equivalent to $[\rhody\myto\rhody]$. \\
More precisely the evaluation functional $C[0;1]\times[0;1]\ni (f,x)\mapsto f(x)\in\IR$
is $\big(\tilderhorho\times\widetilde\rhody,\widetilde\rhody\big)$--computable
in second-order polytime; \\
and the evaluation functional 
$\UC(\subseteq\{\sdzero,\sdone\}^\omega\to\{\sdzero,\sdone\}^\omega)\times\{\sdzero,\sdone\}^\omega
\ni (F,\bar x)\mapsto F(\bar x)\in\{\sdzero,\sdone\}^\omega$
is $\big(\tilde\eta\times\tilde\id,\tilde\id\big)$--computable in second-order polytime.
\item[f)]
It holds $\dyrhoLip\reduceP\tilderhorho\big|^{\Lip[0;1]}$ 
with $\Lip[0;1]\subseteq C[0;1]$; \\
but $\tilderhorho\big|^{\Lip[0;1]}\not\reduceP\dyrhoLip$.
\item[g)]
For arbitrary closed $K\subseteq\{\sdzero,\sdone\}^\omega$,
every total polytime-computable $F:K\to\{\sdzero,\sdone\}^\omega$ 
admits a polytime-computable $\tilde\eta$--name $\psi$. 
Similarly, to every `family' of total functionals
$\Lambda:C[0;1]\times[0;1]\to C[0;1]$
$(\tilderhorho\times\widetilde\rhody,\widetilde\rhody)$--computable in second-order polytime,
there exists a second-order polytime 
$(\tilderhorho,\tilderhorho)$--computable operator
$\IO:C[0;1]\to C[0;1]$
such that $\Lambda(f,x)=\calO\big(f\big)(x)$
holds for all $f$ and all $x$.
\end{enumerate}
\end{myexample}
Concerning a), each $f\in C[0;1]$ is indeed bounded;
hence the output of $\psi$ on $\{\sdzero,\sdone\}^n$ can be padded with
leading zeros (as opposed to Remark~\ref{r:TTE}e)
to $\calO(\log\|f\|_\infty)$ binary digits before the point and $n$ after.
On the other hand, $\dyrho$ alone lacks a bound on the slope 
of $f$ necessary to evaluate it on non-dyadic arguments
and leads to the representations in b) and c).
\\
Regarding d), note that the mapping $\sdone^n\,\sdzero\,\vec x\to \sdone^m$
yields a local modulus of continuity to $F$; and the uniform
continuity prerequisite again ensures that $m=m(n,\vec x)$ 
can be chosen to depend only on $n$ and $|\vec x|$, 
thus yielding a length-monotone string function $\psi$.
\\
In e),
functions with `large' modulus of continuity 
and/or `large' values that might take `long' to evaluate
(recall Example~\ref{x:Modulus}) necessarily have 
$\tilderhorho$--names $\psi$ with `large' $|\psi|$.
Explicitly on Cantor space, given $\bar x$ and $n$,
invoke the oracle on $\sdone^n$ and on finite initial segments $\vec x$ 
of $\bar x$ of increasing length $m'=1,2,\ldots$ until the reported 
$\sdone^m\,\sdzero\,\vec y$ satisfies $m\leq m'$, then print $\vec y$:
This asserts $d\big(F(\bar x),\vec y\big)\leq2^{-n}$
and takes $\poly(m,n)$ steps, that is, second-order polytime.
The real case proceeds even more directly by querying
oracle $\psi$ for $m:=\mu(n+1)$ and then for $f(q)$ 
with $q\in\ID_{m}$ and $|x-q|\leq2^{-n-1}$.
\\
For the polytime reduction in f),
verify that $\mu(n)=n+\log\ell$ is a modulus of continuity to
$f\in\Lip_{\ell}[0;1]$. Concerning failure of the converse, 
$\ell:=2^{\max_n \mu(n)-n}$
could (if finite) serve a Lipschitz constant to $f$ but is
clearly not computable from finitely many queries to $\mu$.
We postpone the formal proof to Example~\ref{x:Param3}h). 
\\
Turning to g), computing $F$ in polytime means calculating the first 
$n$ digits $\vec y$ of $F(\bar x)$ from $\bar x$ in time polynomial in $n$;
and simulating this computation while keeping track of the number
$m$ of digits $\vec x$ of $\bar x$ thus read essentially yields 
the $\tilde\eta$--name $\psi$.
The difficulty consists in algorithmically finding some appropriate
padding to obtain $\psi$ length-monotone in the sense of Equation~(\ref{e:Regular}).
For co-r.e. $K$, such a bound is at least computable
\mycite{Theorem~5.5}{WeihrauchComplexity};
however in our non-uniform polytime setting, 
it constitutes some fixed polynomial 
and can be stored in the algorithm.
For fixed $f$, the real case proceeds similarly
and is easily seen to run in second-order polytime
uniformly in $f$.
\begin{myremark} \lab{r:FunctionAsPoint}
\begin{enumerate}
\item[a)]
We follow {\rm\cite{Mehlhorn,Kapron,AkiSTOC}} in defining the complexity 
of operators such that computations on `long' arguments
are granted more running time and still be considered polynomial,
thus avoiding Problem~\ref{p:Problem}a).
This approach resembles \mycite{D\'{e}finition~2.1.7}{Lombardi} where, 
too, computations on `long' arguments are granted more time to achieve a desired
output precision of $2^{-n}$; cmp. \mycite{\S7}{WeihrauchComplexity}.
Note, however, that \mycite{D\'{e}finition~2.2.9}{Lombardi} 
refers only to sequences of functions (cmp. Definition~\ref{d:Representation2}d)
and, in spite of \mycite{Theor\'{e}me~5.2.13}{Lombardi}, is therefore not a
fully uniform notion of complexity for operators.
A related concept, \emph{hereditarily polynomial bounded analysis} 
\cite{ProofTheory,Oliva} permits polynomially bounded quantification
and thus climbing up \person{Stockmeyer}'s polynomial hierarchy.
\item[b)]
Some $f\in\Lip_1[0;1]$ can be regarded 
\begin{enumerate}
\item[i)] on the one hand as 
a transformation on real numbers  \quad and
\item[ii)] on the other hand as a point in a separable metric space.
\end{enumerate}
\noindent
Both views induce `natural' notions of computability and complexity
(Definition~\ref{d:MetricComplex}).
Now i) and ii) are equivalent concerning the former
\mycite{\S6.1}{Weihrauch}; and so are they when
equipping $\Lip_1[0;1]$ with the second-order representation $\dyrho$;
cmp. Lemma~4.9 in the journal version of {\rm\cite{AkiSTOC}}.
Indeed it seems desirable that type conversion
(the \textsf{utm} and \textsf{smn} properties)
be not just computable \mycite{Theorem~2.3.5}{Weihrauch} 
but within uniform polytime so (Example~\ref{x:Representation2}e+f)
--- and, in view of Remark~\ref{r:Uniformity}b), not merely for 
polytime--computable functions \mycite{Theorem~8.13}{Ko91}.
\end{enumerate}
\end{myremark}
\begin{figure}[htb]
\begin{center}
\includegraphics[width=0.25\textwidth]{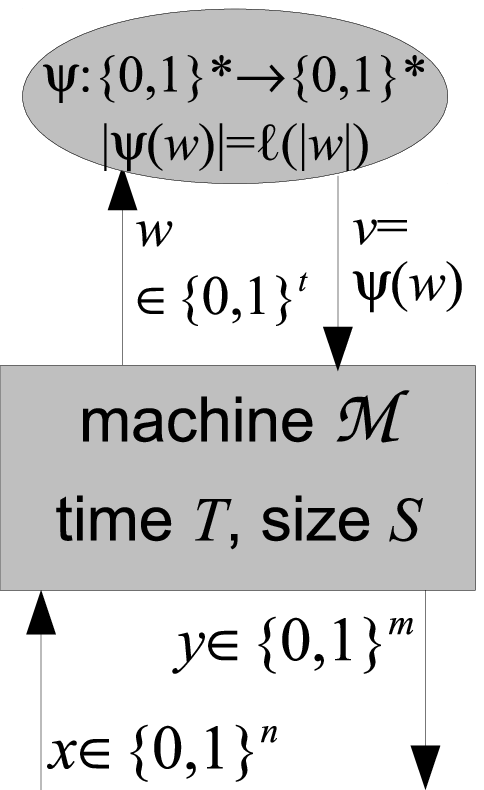}\qquad\qquad\qquad
\includegraphics[width=0.22\textwidth]{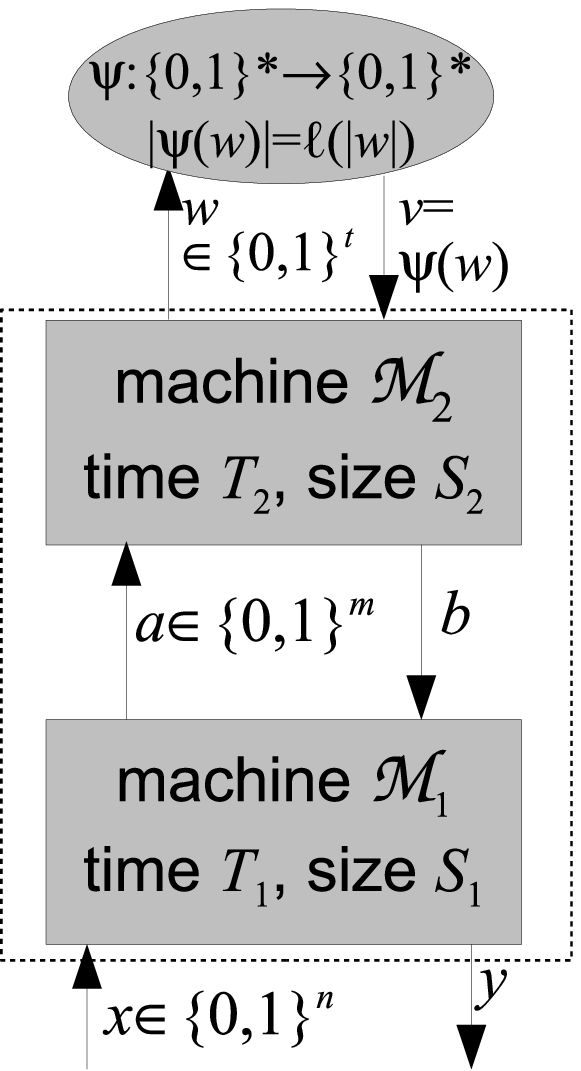}
\caption{\label{f:TTE2}a) Computation of $\calM^{\psi}$.
\quad b) Composition of two oracle machines.}
\end{center}
\end{figure}


\section{Parameterized Type-2 and Type-3 Complexity} \lab{s:Parameterized}
As illustrated in Example~\ref{x:Param1}a+b), 
real number computations sometimes may not admit running times
bounded in terms of the output precision only; 
cmp. also the case of inversion $(0;1]\ni x\mapsto 1/x$ 
\mycite{Exercise~7.2.10+Theorem~7.3.12}{Weihrauch}
and of polynomial root finding \cite{Hotz}.
Such effects are ubiquitous in numerics and captured
quantitatively for instance in so-called \emph{condition numbers} 
$k$ of matrices or, more generally, of
partial functions $f:\subseteq\IR^d\to\IR$ 
with singularities/diverging behaviour on
$\vec x\not\in\dom(f)$ \cite{Buergisser}:
in order to express and bound in terms of both $n$ and $k$
the number of iterations, i.e. basically 
the running time it takes in order to attain a prescribed (although,
pertaining to the \BCSS model with equality decidable,
usually relative) precision $2^{-n}$.
Put differently, $k$ serves in addition to $n$
as a second (but still first-order) parameter:
just like in classical complexity theory \cite{FlumGrohe}.
We combine both TTE complexity \mycite{Definition~2.1}{WeihrauchComplexity}
and the discrete notion of \emph{fpt--reduction}:

\begin{mydefinition} \lab{d:FPTTTE}
\begin{enumerate}
\item[a)]
Fix $F:\subseteq\{\sdzero,\sdone\}^\omega\to\{\sdzero,\sdone\}^\omega$ 
and $k:\dom(F)\to\IN$, called \emph{parameterization of $F$}. 
The pair $(F,k)$ is \emph{fixed-parameter tractable} if
there exists some recursive function $p:\IN\to\IN$ 
such that a Type-2 Machine can compute 
$\dom(F)\ni\bar\sigma\mapsto F(\bar\sigma)$ within
at most $n^{\calO(1)}\cdot p\big(k(\bar\sigma)\big)$ steps. 
\item[b)]
$(F,k)$ is \emph{fully polytime} computable
if the above running time is bounded by a polynomial
in both $n$ and $k(\bar\sigma)$, that is, if $p\in\IN[K]$.
\item[c)]
A \emph{parameterized representation} or
\emph{representation with parameter} of a space $X$
is a tuple $(\xi,k)$ where $\xi$ denotes a
representation of $X$ and $k:\dom(\xi)\to\IN$ some function.
\item[d)]
For a parameterized representation $(\xi,k)$ of $X$
and one $(\upsilon,\ell)$ of $Y$, let
$(\xi,k)\times(\upsilon,\ell):=(\xi\times\upsilon,k+\ell)$
denote the parameterized representation of $X\times Y$.
Here, $\xi\times\upsilon$ is the representation
according to Remark~\ref{r:TTE}c); and 
$k+\ell$ formally denotes the mapping
$(\sigma_0,\tau_0,\sigma_1,\tau_1,\ldots)\mapsto
k(\sigma_0,\sigma_1,\ldots)+\ell(\tau_0,\tau_1,\ldots)$.
\item[e)]
For a parameterized representation $(\xi,k)$ of $X$
and a representation $\upsilon$ of $Y$,
a (possibly partial and multivalued) function 
$f:\subseteq X\toto Y$ is \emph{fully polytime} 
$\big((\xi,k),\upsilon\big)$--computable
if it admits a $(\xi,\upsilon)$--realizer $F$
such that $(F,k)$ is \emph{fully polytime} computable.
\item[f)]
If in e) the representation $\upsilon$ for $Y$ is 
equipped with a parameter $\ell$ as well, call
$f:\subseteq X\toto Y$ a \emph{fixed-parameter reduction} 
if it admits a $(\xi,\upsilon)$--realizer $F$
such that $(F,k)$ is fixed-parameter tractable
and it holds $\ell\big(F(\bar\sigma)\big)\leq k(\bar\sigma)$.
\item[g)]
Again for parameterized representations $(\xi,k)$ of $X$
and $(\upsilon,\ell)$ of $Y$, 
a \emph{fully polytime} $\big((\xi,k),(\upsilon,\ell)\big)$--computable
$f:\subseteq X\toto Y$ must admit a $(\xi,\upsilon)$--realizer $F$
such that $(F,k)$ is \emph{fully polytime} computable
while increasing the parameter at most polynomially:
$\ell\circ F \leq p\circ k$ for some $p\in\IN[K]$.
\item[h)]
A parameterized representation $(\xi,k)$ of $X$
is \emph{fully polytime reducible} to 
another parameterized representation 
$(\upsilon,\ell)$ of $X$
(written $(\xi,k)\reducep(\upsilon,\ell)$)
if $\id:X\to X$ is fully polytime 
$\big((\xi,k),(\upsilon,\ell)\big)$--computable. 
\emph{Fully polytime equivalence} means that 
also the converse holds.
\end{enumerate}
\end{mydefinition}
Since the same point $x\in X$ may have several $\xi$--names 
(some perhaps more difficult to parse or process than others),
also the parameter may have different values for these names in e+f+g).
Note that both fixed-parameter reductions (Definition~\ref{d:FPTTTE}f)
and fully polytime computable functions (Definition~\ref{d:FPTTTE}g) 
are closed under composition. 

We did not define parameterized running times for 
second-order maps $\tilde F:\subseteq\Reg\to\Reg$:
because $\Reg$, as opposed to infinite binary strings,
already comes equipped with a notion of size as parameter
entering in running time bounds: recall Definition~\ref{d:Poly2}b).

\begin{myremark}
The above notions make also sense
for other complexity classes such as polynomial space. 
Giving up closure under composition,
Definitions~\ref{d:FPTTTE}a+b+e+f)
can be refined quantitatively to,
say, quadratic-time computability ---
but Definition~\ref{d:FPTTTE}g) becomes ambiguous:
just like in both the discrete case and 
\mycite{D\'{e}finition~2.1.7}{Lombardi};
cmp. \mycite{\S7}{WeihrauchComplexity}.
\end{myremark}
Also, as in unparameterized TTE complexity theory, 
the above notion of (parameterized) complexity may be 
meaningless for some representations --- that can be 
avoided by imposing additional (meta-) conditions 
\mycite{\S4+\S6}{WeihrauchComplexity}.
Example~\ref{x:Param1}b) is now rephrased as Item~a)
of the following

\begin{myexample} \lab{x:Param2}
\begin{enumerate}
\item[a)]
The exponential function on the entire real line is fully polytime 
$\big((\rhody,|\rhody|),\rhody\big)$--computable,
where $|\rhody|:=|\,\cdot\,|\circ\rhody$;
recall Definition~\ref{d:NumberComplex}. 
\item[b)]
A constant parameter has no effect asymptotically: For any fixed $c$,
polytime $(\xi,\upsilon)$--computability is equivalent to
fully polytime $\big((\xi,c),\upsilon\big)$--computability
and to fully polytime $\big((\xi,c),(\upsilon,\ell)\big)$--computability. \\
On the other hand, and as opposed to the discrete case,
not every (even total) computable
$F:\{\sdzero,\sdone\}^\omega\to\{\sdzero,\sdone\}^\omega$
admits some parameterization $k$ rendering
$(F,k)$ is fixed-parameter tractable.
\item[c)]
Suppose $\tilde F:\subseteq\Reg\to\Reg$ is second-order polytime computable
and has $\dom(\tilde F)$ consisting only of string functions of linear length
in the following sense: there exists $c\in\IN$
and $k:\dom(\tilde F)\to\IN$ such that every $\psi\in\dom(\tilde F)$ 
satisfies $\forall n:|\psi|(n)\leq c\cdot n+k(\psi)$.
Then $\tilde F$ is fully polytime computable. \\
The hypothesis is for instance satisfied for 
any $\big(\dyrhoLip,\tilde\upsilon\big)$--realizer $\tilde F$
of some $\Lambda:\subseteq\Lip[0;1]\toto Y$.
\item[d)]
Evaluation of a given power series 
$\big((a_j)_{_j},z\big)\mapsto \sum_{j=0}^\infty a_j z^j$
is not $\big((\rhody^2)^\omega,\rhody^2\big)$--computable,
even restricted to real arguments $|x|\leq1$ and real
coefficient sequences $(a_j)_{_j}=:\bar a$
of radius of convergence $R(\bar a):=1/\limsup_j|a_j|^{1/j}>1$. \\
However when given, in addition to approximations to
$\bar a$ and $z$, also integers $K,A$ with 
\begin{equation} \label{e:PowerSeries}
R\;>\;\sqrt[K]{2}\;=:\;r \qquad\text{and}\qquad 
\forall j:\;|a_j|\;\leq\;A/r^j \enspace , 
\end{equation}
evaluation up to error $2^{-n}$ becomes uniformly computable 
within time polynomial in $n$, $K$, and $\log A$. \\ 
Formally let $\pi$ denote the following representation of 
$\Comega_1:=\big\{\bar a\subseteq\IC, R(\bar a)>1\big\}$:
A $\pi$--name of $\bar a$ is a
$\big((\rhody^2)^\omega\times\binary\times\unary\big)$--name
(recall Definition~\ref{d:NumberComplex}e+Remark~\ref{r:TTE}c+d)
of $(\bar a,A,K)$ satisfying Equation~(\ref{e:PowerSeries}).
Equipping $\pi$ with parameterization $K+\log A$
renders evaluation
$(\Comega_1,[-1;1])\ni(\bar a,x\big)\mapsto\sum_j a_jx^j$
fully polytime $\big((\pi,K+\log A)\times\rhody,\rhody)$--computable.
\end{enumerate}
\end{myexample}
Note that the (questionable) fully polytime computability of $\exp(x)$ in a) 
hinges on using as parameter the value rather than, perhaps more naturally, 
the binary length of (the integral part of) the argument $x$.
Indeed, similarly to the discrete function $\IN\ni x\mapsto 2^x$,
an output having length exponential in that of the input 
otherwise prohibits polytime computability. A notion of parameterized
complexity taking into account the output size is suggested in
Definition~\ref{d:dualuse}j) below.

\begin{proof}[Example~\ref{x:Param2}]
\begin{enumerate}
\item[a)] immediate.
\item[b)] A polynomial evaluated on a constant argument is again a constant. \\
On the other hand the Time Hierarchy Theorem
yields a binary sequence $\IN\ni n\mapsto\tau_n\in\{\sdzero,\sdone\}$ 
computable but not within time polynomial in $n$; now consider the 
constant function $F:\{\sdzero,\sdone\}^\omega\ni\bar\sigma\mapsto\bar\tau$.
\item[c)]
According to Observation~\ref{o:SecondVsFirst}a+b+c),
second-order polynomials $P(n,\ell)$ 
on linearly bounded second-order arguments $\ell(n)\leq c\cdot n+k$
with fixed $c$ can be bounded by an ordinary polynomial in $n+k$. \\
Moreover recall (Example~\ref{x:Representation2}b) that
$\big(\dyrhoLip\big)$--names have linear length.
\item[d)]
The negative claim is folklore; cmp. e.g. \cite{Mueller95}.
We defer the proof of the positive claim to
Theorem~\ref{t:PowerSeries}a).
\qed\end{enumerate}\end{proof}
Note that, for a parameterized representation $(\xi,k)$ of $X$
according to Definition~\ref{d:FPTTTE}d),
a Type-2 Machine $\big((\xi,k),\upsilon\big)$--computing some 
$f:X\ni x\mapsto f(x)\in Y$ is provided merely with a $\xi$--name
$\bar\sigma$ of $x$ but not with the value $k(\bar\sigma)$ of the
parameter entering in the running time bound it is to obey.
In the case of the global exponential function (Example~\ref{x:Param2}a),
an upper bound to this value is readily available as part of the
given $\rhody$--name of $x$. In the case of power series evaluation
(Example~\ref{x:Param2}d), on the other hand, the values
of parameters $K$ and $A$ had to be explicitly provided
by means of the newly designed representation $\pi$,
that is by `enriching' \cite[p.238/239]{KreiselMacIntyre}
$(\rhody^2)^\omega$, in order
to render an otherwise discontinuous operation computable;
recall also Example~\ref{x:Param1}.
Such simultaneous use of integers as both complexity parameters
and discrete advice will arise frequently in the sequel
and is worth a generic

\begin{mydefinition} \lab{d:dualuse}
\begin{enumerate}
\item[a)]
Let $\xi$ denote an ordinary representation of $X$
and $L:X\toto\IN$ some total multivalued function.
Then ``\emph{$\xi$ with advice parameter $L$ in unary}''
means the following parameterized representation of $X$,
denoted as $\xi+\unary(L)=:(\upsilon,k)$: 
an $\upsilon$--name of $x\in X$ is an infinite binary string 
$\bar\sigma=\sdone^\ell\,\sdzero\,\bar\tau$
where $\bar\tau$ is a $\xi$--name of $X$ and $\ell\in L(x)$;
and $k(\bar\sigma):=\ell$.
\item[b)]
Let $\xi$ denote an ordinary representation of $X$
and $L:X\toto\IN$ some total multivalued function.
Then ``\emph{$\xi$ with advice parameter $L$ in \emph{bi}nary}''
means the following parameterized representation of $X$,
denoted as $\xi+\binary(L)=:(\upsilon,k)$: 
an $\upsilon$--name of $x\in X$ is an infinite binary string 
$\bar\sigma=\langle\bin(\ell),\bar\tau\rangle$
where $\bar\tau$ is a $\xi$--name of $X$ and $\ell\in L(x)$;
and $k(\bar\sigma):=\log\ell$.
Here, $\langle\,\cdot\,,\,\cdot\,\rangle:
\{\sdzero,\sdone\}^*\times\{\sdzero,\sdone\}^\omega\to\{\sdzero,\sdone\}^\omega$
denotes some fixed injective linear-time bi-computable mapping.
\item[c)]
Let $\tilde\xi:\subseteq\Reg\to X$ denote a second-order representation
and $L:X\toto\IN$ some total multivalued function.
Then ``\emph{$\tilde\xi$ with advice parameter $L$ in unary}''
means the following second-order representation,
denoted as $\tilde\xi+\unary(L)$: 
a name of $x\in X$ is a mapping
$\{\sdzero,\sdone\}^*\ni\vec w\mapsto\sdone^\ell\,\sdzero\,\psi(\vec w)\in\{\sdzero,\sdone\}^*$
where $\psi$ is a $\tilde\xi$--name of $x$ and $\ell\in L\big(\tilde\xi(\psi)\big)$.
\item[d)]
Let $\tilde\xi:\subseteq\Reg\to X$ denote a second-order representation
and $L:X\toto\IN$. Then ``\emph{$\tilde\xi$ with advice parameter $L$ in \emph{bi}nary}''
means the following second-order representation, denoted as $\tilde\xi+\binary(L)$: 
a name of $x\in X$ is a mapping
$\{\sdzero,\sdone\}^*\ni\vec w\mapsto\langle\bin(\ell),\psi(\vec w)\rangle\in\{\sdzero,\sdone\}^*$
where $\psi$ is a $\tilde\xi$--name of $x$ and $\ell\in L\big(\tilde\xi(\psi)\big)$.
\item[e)]
For a function
$\tilde F:\subseteq\{\sdzero,\sdone\}^\omega\times\Reg\to\{\sdzero,\sdone\}^\omega$ 
with \emph{parameterization} $k:\dom(\tilde F)\to\IN$,
the pair $(\tilde F,k)$ is \emph{fully polytime} if some 
oracle machine $\calM^?$ can compute $\tilde F$ 
according to Definition~\ref{d:Poly2}b)
within time a second-order polynomial in $n+k$ and $|\psi|$; 
\\
similarly for functions
$\tilde G:\subseteq\{\sdzero,\sdone\}^\omega\times\Reg\to\Reg$
with parameterization $k:\dom(\tilde G)\to\IN$. 
\item[f)]
For a parameterized representation $(\xi,k)$ of $X$
and second-order representation $\tilde\upsilon$ of $Y$
and ordinary representation $\zeta$ of $Z$,
call $f:\subseteq X\times Y\toto Z$ 
\emph{fully polytime} 
$\big((\xi,k),\tilde\upsilon,\zeta\big)$--computable
if it admits a $(\xi,\tilde\upsilon,\zeta)$--realizer $\tilde F$
such that $(\tilde F,k)$ is fully polytime in the sense of e). 

If $(\zeta,\ell)$ is a parameterized representation of $Z$,
call $f$ \emph{fully polytime} 
$\big((\xi,k),\tilde\upsilon,(\zeta,\ell)\big)$--computable
if in addition $\ell$ is bounded by a second-order polynomial
$P$ in $k$ and $|\psi|$:
$\ell\big(\tilde F(\bar\sigma,\psi)\big)
\leq P\big(k(\bar\sigma),|\psi|\big)$.

If $\tilde Z$ is a second-order representation of $Z$,
call $f$ \emph{fully polytime} 
$\big((\xi,k),\tilde\upsilon,\tilde\zeta\big)$--computable
if it admits a $(\xi,\tilde\upsilon,\tilde\zeta)$--realizer $\tilde G$
such that $(\tilde G,k)$ is fully polytime in the sense of e). 
\item[g)]
Fully polytime reduction of a parameterized representation
$(\xi,k)$ of $X$ to a second-order representation $\tilde\zeta$
of $X$ is written as $(\xi,k)\reducep\tilde\zeta$ and
means fully polytime $\big((\xi,k),\tilde\zeta\big)$--computability
of $id:X\to X$; 
similarly for $\tilde\zeta\reduceP(\xi,k)$.
\item[h)]
$(\tilde F,k)$ and $(\tilde G,k)$ as in e)
are \emph{fixed-parameter tractable}
if the computation time is
$\leq P(n,|\psi|)\cdot p(k)$
for a second-order polynomial $P$ and some
arbitrary function $p:\IN\to\IN$.

In the setting of f), 
call $f:\subseteq X\times Y\toto Z$ 
\emph{fixed-parameter} 
$\big((\xi,k),\tilde\upsilon,\zeta\big)$--computable
if it admits a $(\xi,\tilde\upsilon,\zeta)$--realizer $\tilde F$
such that $(\tilde F,k)$ is fixed-parameter tractable; 

and \emph{fixed-parameter}
$\big((\xi,k),\tilde\upsilon,\tilde\zeta\big)$--computable
if it admits a $(\xi,\tilde\upsilon,\tilde\zeta)$--realizer $\tilde G$
such that $(\tilde G,k)$ is fixed-parameter tractable.
\item[j)]
For a parameterized representation $(\xi,k)$ of $X$
and second-order representation $\tilde\upsilon$ of $Y$
and a parameterized representation $(\zeta,\ell)$ of $Z$,
call $f:\subseteq X\times Y\toto Z$
\emph{output-sensitive polytime}
$\big((xi,k),\tilde\upsilon,(\zeta,\ell)\big)$--computable
if it admits a 
$\big((xi,k),\tilde\upsilon,(\zeta,\ell)\big)$--realizer
computable within time a second-order polynomial in $n+k+\ell$ 
and the length of the given $\tilde\upsilon$--name. 

For second-order representation $\tilde\zeta$ of $Z$,
call $f:\subseteq X\times Y\toto Z$
\emph{output-sensitive polytime}
$\big((xi,k),\tilde\upsilon,\tilde\zeta\big)$--computable
if it admits a 
$\big((xi,k),\tilde\upsilon,\tilde\zeta\big)$--realizer
computable within time a second-order polynomial in $n+k+\ell$ 
and the lengths of the given $\tilde\upsilon$--names
and of the produced $\tilde\zeta$--names;
recall Observation~\ref{o:SecondVsFirst}a).
\end{enumerate}
\end{mydefinition}
A more relaxed notion of second-order fixed-parameter tractability (Item~h)
might allow for running times polynomial in $n$ multiplied with
some arbitrary second-order function of both $k$ and $|\psi|$.
Output-sensitive running times (Item~j) are common, e.g., in Computational Geometry. 
As usual, careless choices of the output parameter $\ell$ or output
representation $\tilde\zeta$ may lead to useless notions of
output-sensitive polytime computations.
\\
The representation of $\Lip[0;1]$ from Example~\ref{x:Representation2}b) 
is an instance of Definition~\ref{d:dualuse}d).
Further applications will appear in Definition~\ref{d:PowerSeries} below
to succinctly rephrase the parameterized representation of
$\Comega_1$ from Example~\ref{x:Param2}d).
Items~e)+f) extend Definition~\ref{d:FPTTTE}e+g+h).

\begin{mylemma} \lab{l:dualuse}
\begin{enumerate}
\item[a)]
Fixing $d\in\IN$ and first-order representations 
$\xi,\upsilon$ of $X,Y$, a total
$f:X\toto Y$ is $(xi,\upsilon)$--computable
\emph{with $d$--wise advice} in the sense of 
\mycite{Definition~8}{Advice}
~iff~ there exists some $L:X\toto\{1,\ldots,d\}$
such that $f$ is $(\xi+\unary(L),\upsilon)$--computable.
\item[b)]
$\xi+\unary(L)$ and $\xi+\binary(2^L)$ are 
fully polytime equivalent;  \\
$\tilde\xi+\unary(L)$ and $\tilde\xi+\binary(2^L)$
are second-order polytime equivalent.
\item[c)]
Extending Observation~\ref{o:SecondVsFirst}d+e), 
let $\xi$ denote a representation of $X$ 
with induced second-order representation $\tilde\xi$
and fix $K:X\toto\IN$.
Then it holds $\xi+\binary(K)\reducep\tilde\xi+\binary(K)$
and $\tilde\xi+\binary(K)\reduceP\xi+\binary(K)$.
\end{enumerate}
\end{mylemma}
\begin{proof}
\begin{enumerate}
\item[a)] immediate. 
\item[b)]
It is easy to decode a given $\big((\xi,K)+\unary(L)\big)$--name
$\bar\sigma=\sdone^\ell\,\sdzero\,\bar\tau$ into $\ell$ and $\bar\tau$
and to recode it into $\langle\bin(2^\ell),\bar\tau\rangle$
as well as back, both within time $\calO(n+\ell)$;
similarly for computing
$\{\sdzero,\sdone\}^*\ni\vec w\mapsto\langle\bin(2^\ell),\psi(\vec w)\rangle\in\{\sdzero,\sdone\}^*$
by querying 
$\{\sdzero,\sdone\}^*\ni\vec w\mapsto\sdone^\ell,\,\sdzero\,\psi(\vec w)\rangle\in\{\sdzero,\sdone\}^*$.
and vice versa.
\item[c)]
Recall that a $\big(\xi+\binary(K)\big)$--name of $x$
is an infinite string of the form
$\langle\bin(k),\bar\sigma\rangle$ with $\xi(\bar\sigma)=x$ and $k\in K(x)$.
It corresponds to a 
$\big(\tilde\xi+\binary(K)\big)$--name 
$\psi:\sdone^n\mapsto\langle\bin(k),\sigma_{n'}\rangle$
where $n\approx {n'}+\log k$ and $|\psi|\approx1+\log k$ has length independent of $n$.
This leads to conversion in both directions,
computable within time polynomial in $n+\log k$.
\qed\end{enumerate}\end{proof}
\COMMENTED{
Fix ordinary parameterized representation
$(\xi,K)$ of $X$ and second-order representation
$\tilde\xi$ of $X$ as well as $L,M:X\toto\IN$.
\\
$\big((\xi,K)+\binary(L)\big)+\binary(M)$ and
$\big((\xi,K)+\binary(M)\big)+\binary(L)$ and
$\big((\xi,K)+\binary(\langle M,L\rangle\binary)$
are fully polytime equivalent;  \\
$(\tilde\xi+\binary(L))+\binary(M)$ and
$(\tilde\xi+\binary(M))+\binary(L)$ and
$(\tilde\xi+\binary(\langle M,L\rangle\binary)$
are second-order polytime equivalent.
\begin{proof}
We mutually decode and encode
$\big\langle\bin(m),\langle\bin(\ell),\bar\tau\rangle\big\rangle$
and
$\big\langle\bin(\ell),\langle\bin(m),\bar\tau\rangle\big\rangle$
and
$\big\langle\bin(\langle m,\ell\rangle),\bar\tau\big\rangle$
for any $(\xi,K)$--name $\bar\tau$ and any $m,\ell$;
which is easy in time $\Theta\big(n+\log(m)+\log(\ell)\big)
=\Theta\big(\log\langle m,\ell\rangle\big)$ independent
of the value or encoding of $K$ into $\bar\tau$. \\
The second-order case proceeds similarly.
\qed\end{proof}} 
Without discrete advice, maximization remains computable
but not within second-order polytime even on analytic functions:

\begin{myexample} \lab{x:Param3}
\begin{enumerate}
\item[a)]
Evaluation (i) on $\Lip[0;1]$, that is
the mapping $(f,x)\mapsto f(x)$, is uniformly polytime
$\big(\dyrhoLip,\rhody,\rhody\big)$--computable; 
\\
addition (ii), and multiplication (iii)  
on $\Lip[0;1]$ are uniformly polytime 
$\big(\dyrhoLip,\dyrhoLip\big)$--computable
within second-order polytime
\item[b)]
and so is composition (vii) when defined, that is, the partial operator
\[ \Lip[0;1]\times \Lip\big([0;1],[0;1]\big)\;\ni\; (g,f)\;\mapsto\; g\circ f\;\in\; \Lip[0;1] \enspace .\]
\item[c)]
Differentiation (iv) is 
$([\rhody\myto\rhody],[\rhody\myto\rhody])$--discontinuous 
(and hence \linebreak \mbox{$\big(\dyrhoLip,\dyrhoLip\big)$--uncomputable})
even restricted to $\Cinfty[0;1]$. 
\item[d)]
On the other hand differentiation does become computable when given, 
in addition to approximations to $f\in C^2[0;1]$,
an (integer) upper bound on $\|f''\|_\infty$. \\
More precisely, $\partial:C^2[0;1]\to C^1[0;1]$ 
is polytime $\big((\dyrhoLip)\mathbf{'},\dyrhoLip\big)$--computable,
where a $\big(\dyrhoLip\big)'$--name of $f\in C^2[0;1]$ is defined to be
a $\big(\dyrhoLip\big)$--name of $f'\in C^1[0;1]$.
\item[e)]
Parametric maximization (vi), namely the operator
\begin{equation} \label{e:Max} 
\MAX:\Lip[0;1]\times[0;1]^2
\;\ni\; (f,u,v)\;\mapsto\; 
\max\big\{f(x):\min(u,v)\leq x\leq\max(u,v)\big\} , \enspace
\end{equation}
is $\big(\dyrhoLip,\rhody^2,\rhody\big)$--computable,
\item[f)]
but not within subexponential time, even restricted to analytic real 1-Lipschitz functions
$f:[0;1]\to[0;1]$.
\item[g)]
Similarly for parametric integration (v), that is the operator
\begin{equation} \label{e:Int}
\int:\Comegax{\cball(0,1)}\times[0;1]^2
\;\ni\; (f,u,v)\;\mapsto\; \int\nolimits_{\min(u,v)}^{\max(u,v)} f(x)\,dx\;\in\;\IR \enspace . 
\end{equation}
\item[h)]
It holds $\tilderhorho\big|^{\Lip[0;1]}\not\reduce\dyrhoLip$.
\end{enumerate}
\end{myexample}
\begin{figure}[htb]
\includegraphics[width=0.48\textwidth]{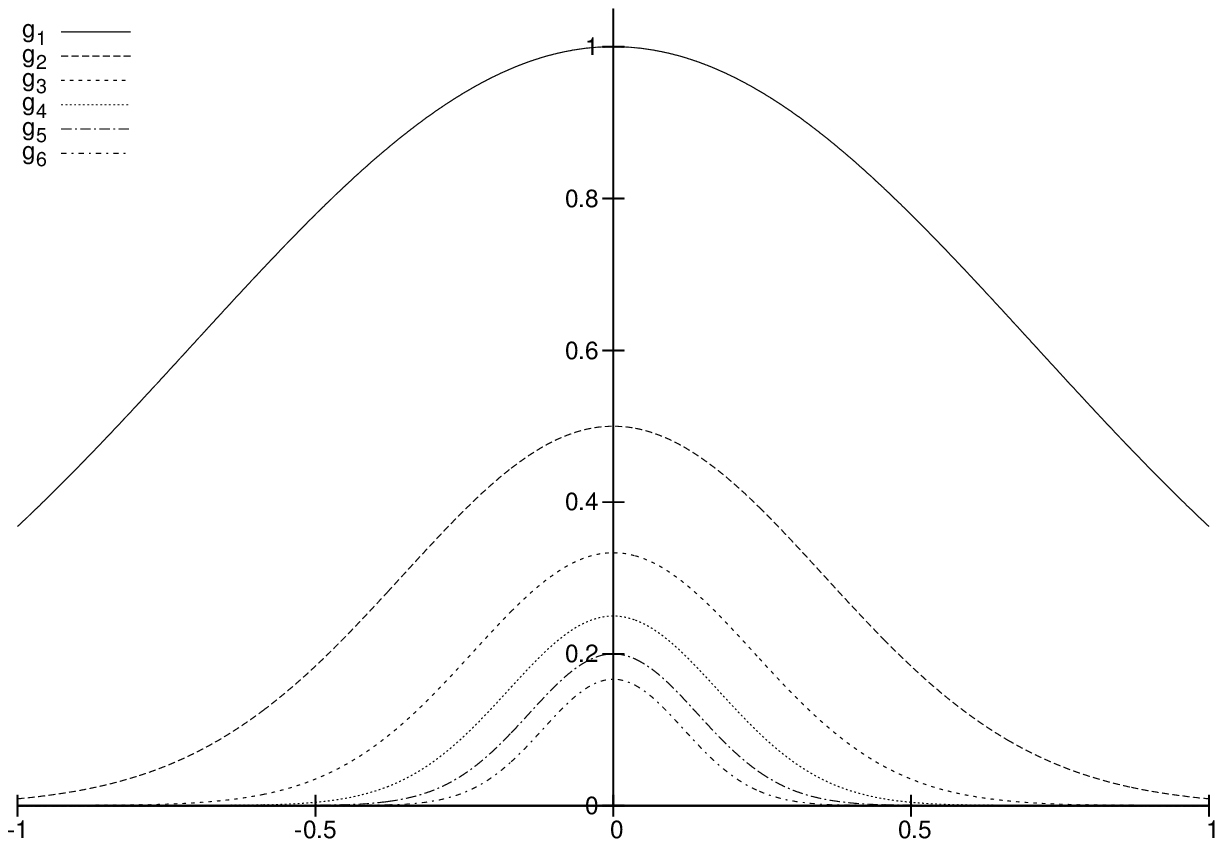}\hfill
\includegraphics[width=0.48\textwidth]{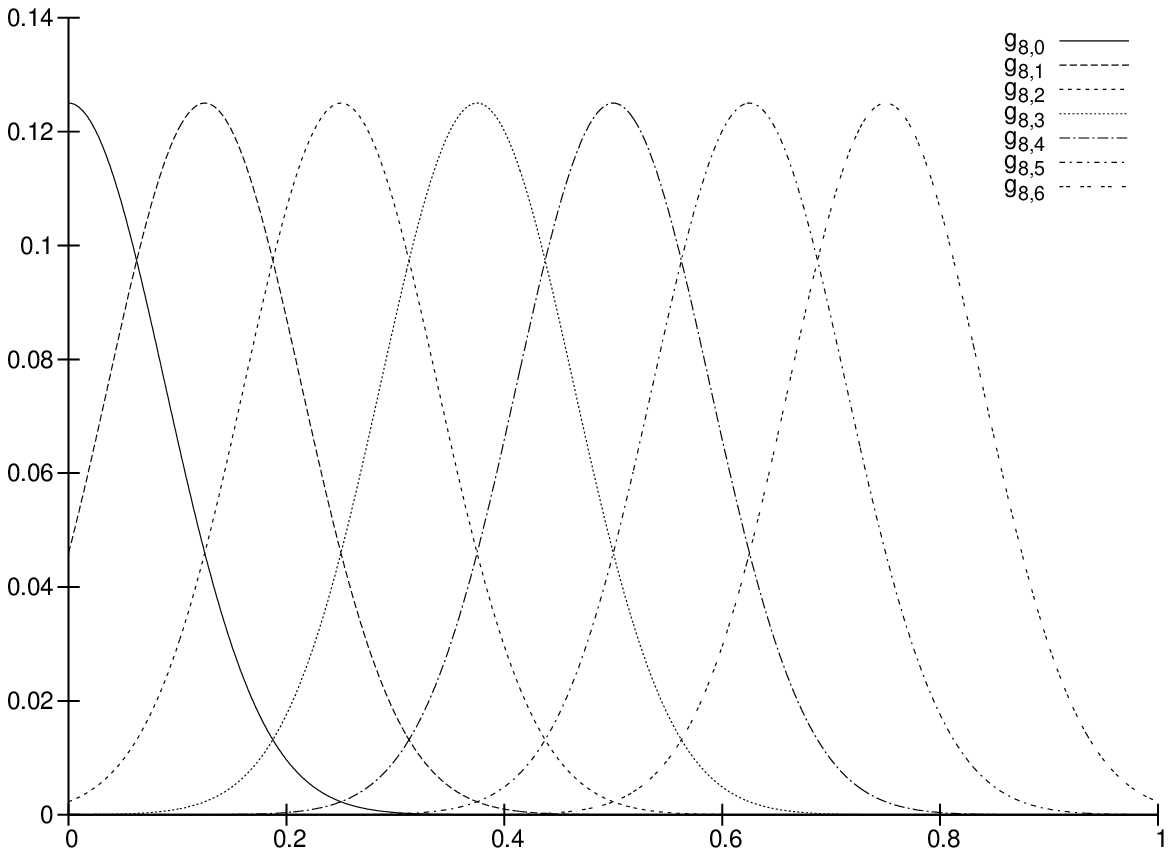}
\caption{\label{f:Gauss}a) Gaussian functions $g_K$ and b) their shifts
$g_{K,k}$ as employed in the proof of Example~\ref{x:Param3}f).}
\end{figure}
\begin{proof}
Concerning a\,i) and b) recall 
Example~\ref{x:Representation2}e+f);
a\,ii) and a\,iii) follow immediately from the uniform pointwise
polytime computability of addition and multiplication on single reals
(Example~\ref{x:Param1}b), taking into account that 
the parameterized time-dependence of the latter for large arguments
is covered in second-order polytime by the length of a name of $f$.
The proof of \mycite{Theorem~6.4.3}{Weihrauch} 
also establishes c); for d) refer
for instance to the proof of \mycite{Theorem~6.2}{Ko91}
or \mycite{Theorem~6.4.7}{Weihrauch}.
Claim~e) and the first part of g) is established,
e.g., in \mycite{Corollary~6.2.5+Theorem~6.4.1}{Weihrauch}.
\\
For f) 
consider the family of analytic Gaussian functions 
\[ g_1(x):=\exp(-x^2), \qquad g_K(x):=g_1(K\cdot x)/K=\exp(-K^2x^2)/K, \quad K\in\IN  \]
depicted in Figure~\ref{f:Gauss}a).
For $K\gg 1$ these are `high' and `thin'
but not too `steep' in the sense that
$\forall x:|g'_K(x)|\leq1$, that is,
1-Lipschitz and thus admit a linear-size $\dyrho$--name $\psi$;
similarly for their shifts $g_{K,k}:[0;1]\ni x\mapsto g_{K}(x-k/K)$, $0\leq k<K$;
cf. Figure~\ref{f:Gauss}b).
Now any algorithm computing $f\mapsto\Max\big(f\big)(1)=\max\{f(x)\mid0\leq x\leq 1\}$ 
on $\{0,g_{K,0},\ldots,g_{K,K-1}\}\subseteq\Comega[0;1]\cap\Lip_1[0;1]$ 
up to error $2^{-n}=:1/(2K)$
must distinguish (every name of) the identically zero function from 
(all names of) some of the $g_{K,k}$ ($0\leq k<K$)
because the first has $\max(0)=0$ and the others $\max(g_{K,k})=1/K$.
Yet, since the $g_{K}$ are `thin', any evaluation
up to error $2^{-m}$ at some $x$ with $|x-k/K|\geq m/K$ 
(i.e. a query to the given name $\psi$) 
may return $0$ as approximation to $g_{K,k}(x)$.
For a sequence $(x_j,2^{-m_j})$ of queries that unambiguously
distinguishes the zero function from the $g_{K,k}$, the intervals
$\big[x_j-\tfrac{m_j}{K};x_j+\tfrac{m_j}{K}\big]$ therefore must necessarily
cover $[0;1]$ and in particular satisfy $\sum_j m_j\geq K/2=2^{n-2}$.
On the other hand each such query takes $\Omega(m_j)$ steps.
\\
For the second part of g) similarly observe
$\int 0=0$ and $\int_0^1 g_{K,k}\geq 
\int_0^{1/K} g_K=
\int_0^1 g_1/K^2\geq 1/(2K^2)=:2^{-n}$.
\\
Turning to h), and on a more refined level, 
consider a hypothetical oracle machine $\calM^{\psi}$
converting a $\tilderhorho$--name $\psi$ of $f(x)\equiv0$ into 
a Lipschitz constant $\ell$ to $f$, necessarily so within finite time
and when knowing finitely many values of $f|_{\ID}$ and 
of a (w.l.o.g. non-decreasing) modulus $\mu$ of continuity to $f$. 
Let $n$ be so large
that no $\mu(n')$ with $n'\geq n$ has thus been queried 
nor the values of $f$ on any pair of arguments closer than $2^{-n}$.
It is no loss of generality to suppose $\ell\geq\mu(n)$.
Then it is easy (but tedious) to add to $f$ a scaled and 
shifted Gaussian function 
$2\ell\cdot g_{2^n,k}$ for some (not necessarily integral) $k$ 
such that the resulting $\tilde f$ 
coincides with $f$ on the arguments queried and has a modulus 
of continuity $\tilde\mu$ coinciding with $\mu$ on $\{1,\ldots,n\}$
and is still $2\ell$--Lipschitz but not $\ell$-Lipschitz.
\qed\end{proof}
Note the similarity of our lower bound proof of
Example~\ref{x:Param3}f+g+h) to arguments in 
\emph{information-based complexity} \cite{TraubIBC,HertlingIBC}
generally pertaining to the \BSS model.

\COMMENTED{
$([\rhody\myto\rhody],[\rhody\myto\rhody])$--discontinuity 
and, in view of Example~\ref{x:Representation2}e), also
$(\tilderhorho,\tilderhorho)$--computability,
is well-known; see for instance the proof of \mycite{Theorem~6.4.3}{Weihrauch}
or \mycite{Theorem~I.1.3}{PER}. \\
Concerning $\big(\dyrhoLip\big)'$,
names provide a Lipschitz constant $\ell'$ to $f'$;
which also constitutes an upper bound to $\|f'\|_\infty$
because of $|f'(x)-f'(0)|\leq\ell'\cdot x$.
Now the proof of \cite[\textsc{Lemma~3.4}b]{Mueller86}
shows differentiation to be uniformly computable in
time polynomial in $n$ and the logarithm of a given (!)
Lipschitz constant $\ell'$ to $f'$.
Finally observe,
recalling the proof of Example~\ref{x:Representation2}a),
that $\dyrho'$--names have length roughly 
$n+\log\|f\|_\infty+\log\ell'$.
}
\begin{myremark} \lab{r:Advantages}
While aware of the conceptual and notational barriers
to these new notions, we emphasize their benefits:
\begin{itemize}
\item They capture numerical practice with various generalized condition numbers as parameters
\item based on, and generalizing, TTE to provide a formal foundation to uniform computation
\item on spaces of `points' as well as of (continuous) functions 
\item by extending discrete parameterized complexity theory
\item with runtime bounds finer than the global worst-case ones
\item while maintaining closure under composition
\item and thus the modular approach to software development by combining subroutines.
\end{itemize}
\end{myremark}
In the sequel we shall apply these concepts to present and 
analyze uniform algorithms receiving analytic functions as inputs.

\section{Uniform Complexity of Operators on Analytic Functions}
For $z\in\IC$ and $r>0$, abbreviate $\ball(z,r):=\{w\in\IC:|w-z|<r\}$
and $\cball(z,r):=\{w\in\IC:|w-z|\leq r\}$.
For $U\subseteq\IC$ a non-empty open set of complex numbers,
let $\Comegax{U}$ denote the class of functions $g:U\to\IC$ 
complex differentiable in the sense of Cauchy-Riemann;
for a closed $A\subseteq\IC$, define $\Comegax{A}$ to consist
of precisely those functions $g:\subseteq\IC\to\IC$ 
with open $\dom(g)\supseteq A$. 

A real function $f:[0;1]\to\IR$ thus belongs to $\Comegax{[0;1]}$
if it is the
restriction of a complex function $g$ differentiable on some 
open complex neighbourhood $U$ of $[0;1]$; cmp. \cite{RealAnalytic}.
By Cauchy's Theorem, each such $g$ can be represented locally around
$z_0\in U$ by some power series $f_{\bar a}(z-z_0):=\sum_{j=0}^\infty a_j (z-z_0)^j$.
More precisely, Cauchy's Differentiation Formula yields
\begin{equation} \label{e:Cauchy}
a_j \;=\; f^{(j)}(z_0)/j! \; =\; 
\frac{1}{2\pi i} \int\nolimits_{|z-z_0|=r} \frac{f(z)}{(z-z_0)^{j+1}}\,dz, 
\qquad \cball(z_0,r)\subseteq U
\end{equation}
Now for a fixed power series with polytime computable
coefficient sequence $\bar a=(a_j)_{_j}$, its anti-derivative and 
ODE solution and even maximum\footnote{Note that
anti-/derivative and ODE solution of an analytic function
is again analytic but parametric maximization in general is not.}
are polytime computable;
see Theorem~\ref{t:PowerSeries} below.
And since $[0;1]$ is compact, finitely many such 
power series expansions with rational centers $z_0$
suffice to describe $f$ ---
and yield the drastic improvements to Fact~\ref{f:Nonunif}
mentioned in Example~\ref{x:Upper}.

On the other hand we have already pointed out there
many deficiencies of nonuniform complexity upper bounds.
For example the mere evaluation of a power series 
requires, in addition to
the coefficient sequence $\bar a$, further information;
recall Example~\ref{x:Param2}d) and 
see also \mycite{Theorem~6.2}{ArithHierarchy}.

The present section presents, and analyzes the parameterized running times of, 
uniform algorithm for primitive operations on analytic functions.
It begins with single power series, w.l.o.g. around 0 with radius
of convergence $R>1$; then proceeds to globally convergent power series
such as the exponential function; and finally to real functions analytic
on $[0;1]$.

\medskip
Uniform algorithms and parameterized upper running time bounds 
for evaluation have been obtained for instance as \mycite{Theorem~28}{YapHypergeom}
on a subclass of power series, namely the hypergeometric ones
whose coefficient sequences obey an explicit recurrence relation
and thus can be described by finitely many real parameters.
Further complexity considerations, and in particular lower bounds,
are described in \cite{RettingerHabil}, \cite{Rettinger08,Rettinger09}.
There is a vast literature on computability in complex analysis. 
For practicality issues refer, e.g., to 
\cite{JorisAnalytic,JorisContinuation,JorisComposition}.
\cite{Gaertner} treats computability questions in the 
complementing, algebraic (aka \BSS) model of real number computation \cite{BCSS};
see also \cite{Braverman1} concerning their complexity theoretic relation.
For some further recursivity investigations in complex analysis refer,
e.g., to \cite{HertlingRiemann,McNicholl} or \mycite{\S6}{Escardo}.

\subsection{Representing, and Operating on, Power Series on the Closed Unit Disc} 
\mycite{Theorem~4.3.11}{Weihrauch} asserts complex power series evaluation
$(\bar a,z)\mapsto f_{\bar a}(z)$ to be uniformly computable
when providing, in addition to a $\rhody^2$--name of $z$
and a $(\rhody^2)^\omega$--name of $\bar a=(a_j)_{_j}\subseteq\IC$,
some $r\in\IQ$ with $|z|<r<R$ and some $A\in\IN$ such that it holds
\begin{equation} \label{e:Hadamard}
\forall j:\;\; |a_j|\;\leq\; A/r^j \enspace ,
\end{equation}
where $R(\bar a):=1/\limsup_j|a_j|^{1/j}$ denotes the
coefficient sequence's radius of convergence.
Note that such $A$ exists for $r<R$ but not necessarily 
for $r=R$ (consider $a_j=j$).
Now Equation~(\ref{e:Hadamard}) yields the tail 
estimate $\big|\sum_{j\geq n} a_j z^j\big|\leq A \frac{(|z|/r)^n}{1-|z|/r}$:
non-uniform in $|z|\to r\to R$. Indeed,
any power series is known to have a singularity
somewhere on its complex circle of convergence
(cf. Figure~\ref{f:domain}a), hence
its rate of convergence must deteriorate as $|z|\to R$;
and evaluation does not admit a uniform complexity bound 
in this representation. Instead we shall replace $r\in\IQ$
by an integer $K$ describing how `close' $r$ is to $R$.
For $R<\infty$, by scaling the argument $z$ 
it suffices to treat the case $|z|\leq 1<R$;
(The case $R=\infty$ will be the subject 
of Section~\ref{s:Entire} below.)

So consider the space $\Comegax{\cball(0,1)}$ of functions
holomorphic on some open neighbourhood of the closed
complex unit disc; put differently: functions $g$ whose
sequence $\bar a$ of Taylor coefficients
around 0 according to Equation~(\ref{e:Cauchy})
have radius of convergence $R(\bar a)>1$.
$\Comegax{\cball(0,1)}$ may thus be identified with $\Comega_1$ 
from Example~\ref{x:Param2}d) in the following

\begin{mydefinition} \lab{d:PowerSeries} 
On the space 
$\Comega_1=\big\{\bar a\subseteq\IC, R(\bar a)>1\big\}$,
consider the multivalued mapping
$(\bar a)\mapstoto\langle 2^K,A\rangle\in\IN$ with
$r:=\sqrt[K]{2}<R(\bar a)$ and $\forall j:|a_j|\leq A/r^j$
and let $\pi:=(\rhody^2)^\omega+\binary\langle 2^K,A\rangle$
denote the representation of $\Comegax{\cball(0,1)}$
enriching a $(\rhody^2)^\omega$--name of $\bar a$ 
with advice parameters $K$ in unary and $A$ in binary
(but not with $R$).
\end{mydefinition}
Note how $K$ encodes a lower bound on $r<R$.
More precisely, large values of $K$ mean $r$ may be close to 1,
i.e. the series possibly converging slowly as $|z|\nearrow1$.
Together with the upper bound $A$ on all $|a_j|$,
this serves both as discrete advice and as a parameter
governing the number of terms of the series to 
evaluate in order to assert a tail error $<2^{-n}$;
see the proof of Theorem~\ref{t:PowerSeries}a) below.
Lemma~\ref{l:Derivative}d) shows $K$ to be of asymptotic order
$\tfrac{1}{r-1}$; and provides bounds on how to transform
$K$ computationally when operating on $\bar a$. For 
example the coefficient sequence $a_j'=(j+1)\cdot a_{j+1}$ 
corresponding to the derivative has the same
radius of convergence $R(\bar a)=R(\bar a')$, classically,
but does not permit to deduce a bound $A'$ as in
Equation~(\ref{e:Hadamard}) from $A$ without increasing $r$.

\begin{figure}[htb]
\includegraphics[height=0.17\textheight]{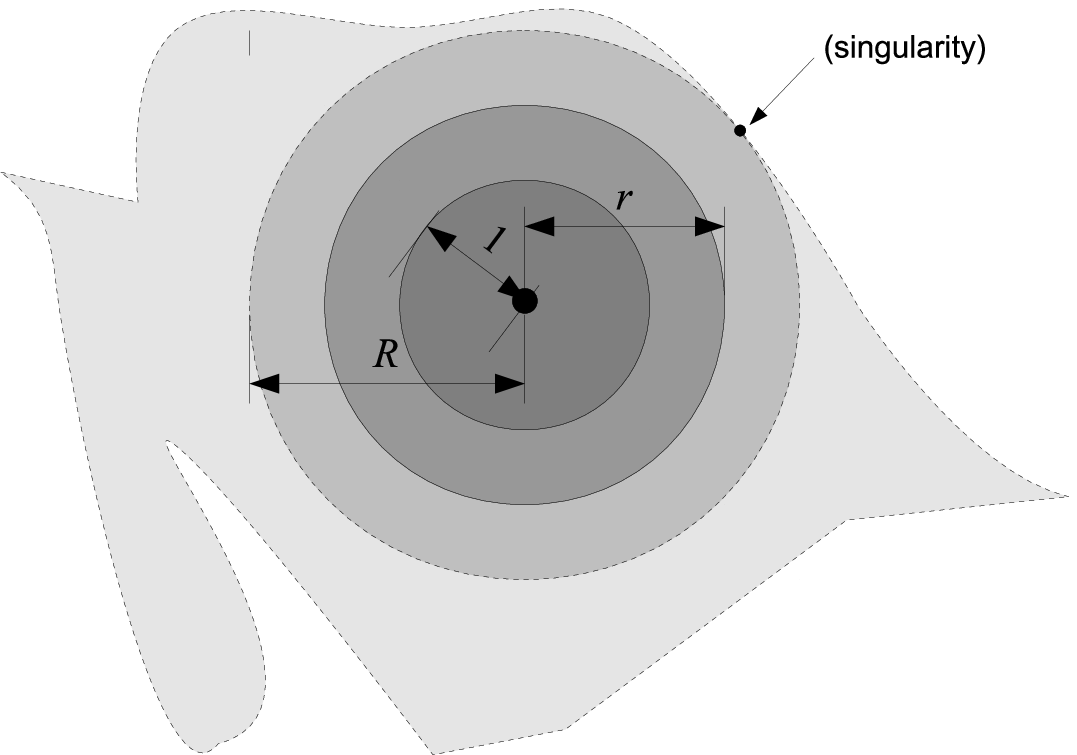}\hspace*{-6ex}%
\includegraphics[height=0.15\textheight]{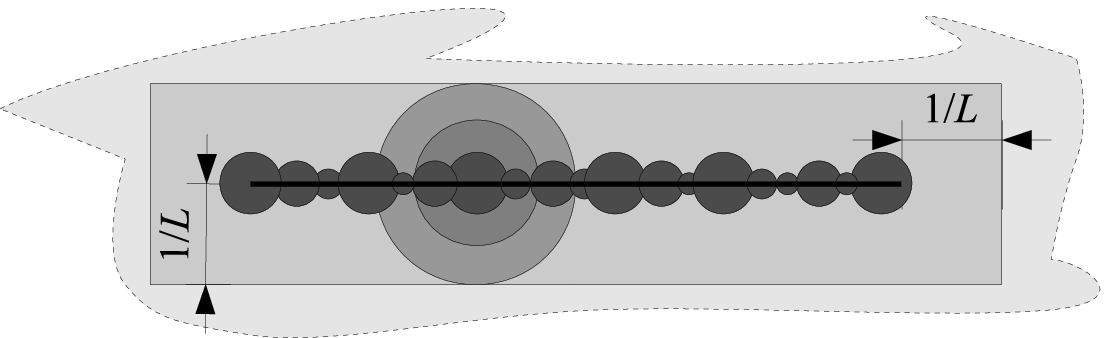}%
\caption{\label{f:domain}Geometry of the parameters underlying a) representation $\pi$
and b) representations $\alpha,\tilde\beta,\tilde\gamma$.}
\end{figure}

\begin{mylemma} \lab{l:Derivative}
\begin{enumerate}
\item[a)]
Let $r>1$. Then it holds $t\leq C\cdot r^t$ for all $t>0$,
where $C:=1/(e\cdot\ln r)$.
\item[b)]
More generally, $t^s\leq (\tfrac{s}{e\ln r})^{s}\cdot r^t$ for all $t>0$
and $s\geq0$ with the convention of $0^0=1$.
\item[c)] 
For all $N\in\IN$ it holds 
$1+\tfrac{\ln 2}{N}\leq\sqrt[N]{2}\leq 1+\tfrac{1}{N}$.
\item[d)]
Asymptotically $1/\ln(r)=\Theta(\tfrac{1}{r-1})$ as $r\searrow1$.
\item[e)]
With $\bar a\in\Comega_1$ and 
$(K,A)$ as in Definition~\ref{d:PowerSeries},
for every $|z|\leq r':=\sqrt[2K]{2}=\sqrt{r}$
it holds $|f_{\bar a}(z)|\leq A\cdot\tfrac{r'}{r'-1}$ and
$|f_{\bar a}^{(d)}(z)|\leq A\cdot d!/(r'-1)^{d+1}$, $d\in\IN$.
\item[f)]
For all $k\geq2$ and $x\geq k^2$ it holds $x^{k}\leq\exp(x)$. \\
For all $a,k\geq1$ and $b>0$ and $x\geq k^2\cdot a^{1/k}/b^2\geq4$
it holds $a\cdot x^k\leq\exp(x\cdot b)$.
\end{enumerate}
\end{mylemma}
\begin{proof}
Any local extreme point point $x_0$ of $0<x\mapsto x\cdot r^{-x}$ 
is a root of $\tfrac{d}{dx}x\cdot r^{-x}=r^{-x}-\ln(r)\cdot x\cdot r^{-x}$,
i.e. $t\cdot r^{-t}\leq C:=\max\{x\cdot r^{-x}:x>0\}$ attained at $x_0:=1/\ln(r)$.
Replacing $r$ with $r^{1/s}$ yields b). For c) observe 
$\sqrt[N]{2}=\exp(\ln 2/N)$ and analyze $\exp(x)$ on $[0;\ln2]\ni x:=1/N$.
Taylor expansion yields $\ln(1-x)\approx-\tfrac{1}{x}-\tfrac{1}{2x^2}-\cdots$ 
for $x:=1-r$. \\
Regarding e),
$|f_{\bar a}(z)|\leq\sum_{j} |a_j|\cdot|z|^j
\leq A\cdot\sum_{j} (r'/r)^j
A\tfrac{1}{1-r'/r}=A\cdot\tfrac{r'}{r'-1}$ and
$|f^{(d)}(z)|\leq A\cdot\sum_{j\geq d} j\cdot (j-1)\cdots (j-d+1)(|z|/r')^{j-d}
=A\tfrac{d^d}{dt^d}\sum_j q^j\big|_{q=|z|/r'}
\leq A\tfrac{d^d}{dq^d}\tfrac{1}{1-q}\big|_{q=|z|/r'\leq r'}
\leq A\cdot d!/(r'-1)^{d+1}$. \\
Turning to f), 
first record that $k^2\leq\exp(k)$
holds for all $k\geq2$. In particular
$x\leq\exp(x/k)$ is true for $x=k^2$;
and monotone in $x$ because of
$\partial_x x=1\leq 1/k\exp(x/k)=\partial_x \exp(x/k)$
for all $x\geq k^2$. \\
Concerning the second claim substitute $y:=x\cdot a^{1/k}$
and conclude from the first that
$y^{k\cdot a^{1/k}/b}\leq\exp(y)$
for all $y\geq k^2\cdot a^{2/k}/b^2\geq4$.
\qed\end{proof}

\begin{theorem} \lab{t:PowerSeries}
\begin{enumerate}
\item[a)] Evaluation 
~$\Comegax{\cball(0,1)}\times\cball(0,1)\ni (f,z)\mapsto f(z)\in\IC$~
is fully polytime $(\pi\times\rhody^2,\rhody^2)$--computable,
that is within time polynomial in $n+K+\log A$
for the parameters $(K,A)$ according to Definition~\ref{d:PowerSeries}.
\item[b)] Addition 
~$\Comegax{\cball(0,1)}
\times \Comegax{\cball(0,1)}
\ni (f_1,f_2)\mapsto f_1+f_2\in \Comegax{\cball(0,1)}$~
is fully polytime $(\pi\times\pi,\pi)$--computable, 
that is within time polynomial in $n+K_1+K_2+\log A_1+\log A_2$.
\item[c)] Multiplication
~$\Comegax{\cball(0,1)}
\times \Comegax{\cball(0,1)}
\ni (f_1,f_2)\mapsto f_1\cdot f_2\in \Comegax{\cball(0,1)}$~
is fully polytime $(\pi\times\pi,\pi)$--computable.
\item[d)] Differentiation
~$\Comegax{\cball(0,1)}\ni f\mapsto f'\in \Comegax{\cball(0,1)}$~
is fully polytime $(\pi,\pi)$--computable. \\
More generally $d$-fold differentiation
~$\Comegax{\cball(0,1)}\times\IN\ni (f,d)\mapsto f^{(d)}\in \Comegax{\cball(0,1)}$~
is fully polytime $\big(\pi+\unary(d),\pi\big)$--computable,
that is within time polynomial in $n+K+\log A+d$.
\item[e)] Anti-differentiation
~$\Comegax{\cball(0,1)}\ni f\mapsto \int f=F\in \Comegax{\cball(0,1)}$~
is fully polytime $(\pi,\pi)$--computable;
and $d$-fold anti-differentiation
is fully polytime $\big(\pi+\unary(d),\pi\big)$--computable.
\item[f)] 
Parametric maximization, that is both the mappings
$\MAX\circ\Re$ and $|\MAX|$ from \linebreak
$\Comegax{\cball(0,1)}\times[-1;1]^2$ to $\IR$ 
are fully polytime $(\pi\times\rhody^2,\rhody)$--computable,
where
\begin{align*}
\MAX\circ\Re&:
(f,u,v)\;\mapsto\; \max\big\{\Re f(x):\min(u,v)\leq x\leq\max(u,v)\big\}
\quad\text{ and }\quad \\
|\MAX|&:(f,u,v)\;\mapsto\;\max\big\{|f(x)|:\min(u,v)\leq x\leq \max(u,v)\big\}
\enspace . 
\end{align*}
\item[g)] 
As a converse to a), given a $\dyrho$--name of $f$
as well as $K\in\IN$ and an integer upper bound
$B$ on $\max\big\{|f(z)|: |z|\leq\sqrt[K]{2}\big\}$ and $j\in\IN$,
the coefficient $a_j=f^{(j)}(0)/j!$
is computable within time polynomial in $n+K+\log B+j$; formally:
the partial function 
\[ \Comegax{[-1;1]}\times\IN\times\IN
\;\ni\; (f,K,B)\;\mapsto \bar a:\quad f=f_{\bar a}, \;
\forall |z|\leq\sqrt[K]{2}: |f(z)|\leq B \]
is polytime $\big(\dyrho+\binary\langle 2^K,B\rangle,\rhody^\omega\big)$--computable.
\end{enumerate}
\end{theorem}
Note that, for real-valued $f$, 
$\min\{f(z):z\}=-\max\{-f(z):z\}$ and
$\max\{|f(z)|:z\}=\max\big(\max\{f(z):z\},-\min\{f(z):z\}\big)$.
Hence the above Items~a) to f) indeed constitute a natural 
choice of  basic primitive operations on $\Comegax{\cball(0,1)}$.

\begin{proof}[Theorem~\ref{t:PowerSeries}]
\begin{enumerate}
\item[a)]
Given $K$, calculate $r:=\sqrt[K]{2}$ 
within time polynomial in $n+K$. 
Then evaluate the first roughly $N:=n\cdot K+\log A$ terms
of the power series on the given $z$:
according to Equation~(\ref{e:Hadamard})
the tail $\sum_{j\geq N} |a_j|\cdot|z|^j
\leq A\cdot\sum_{j\geq N} (|z|/r)^j=
A\cdot(|z|/r)^N\tfrac{1}{1-|z|/r}$
is then small of order 
$(|z|/r)^{n\cdot K}=\calO(2^{-n})$
because of $|z|\leq1$.
\item[b)]
Given $(\bar a,K,A)$ and $(\bar b,L,B)$
output $(\bar c,M,C)$ where $c_j:=a_j+b_j$ 
and $C:=A+B$ and $M:=\max(K,L)$.
In view of Example~\ref{x:Param1}b),
approximating $c_j$ up to error $2^{-m}$
this is possible in time polynomial in
$m+\log(A+B)+j\cdot(K+L)$; now recall
Definition~\ref{d:NumberComplex}e)
\item[c)]
Similarly, given $(\bar a,K,A)$ and $(\bar b,L,B)$
output $(\bar c,M,C)$ where 
$c_j:=\sum_{i=0}^j a_i\cdot b_{j-i}$ 
and $C$ is some integer $\geq A\cdot B\cdot\big(1+M/(e\cdot\ln 2)\big)$
where $M:=2\max(K,L)$. Indeed, with 
$r_A\geq(2)^{\tfrac{1}{K}}$ and $r_B\geq(2)^{\tfrac{1}{L}}$
and $r=\min(r_A,r_B)\geq(2)^{\tfrac{2}{M}}$,
$1<r':=(2)^{\tfrac{1}{M}}\leq\sqrt{r}<r$ implies
$j\leq\sqrt{r^j}/(e\cdot\ln\sqrt{r})\leq
M\cdot\sqrt{r}^{j}/(e\cdot\ln2)$
according to Lemma~\ref{l:Derivative}a);
hence $|c_j|\leq\sum_{i=0}^j A/r_A^i\cdot B/r_B^{j-i}
=A\cdot B\cdot (j+1)/r^j\leq C/\sqrt{r}^j$.
\item[d)] 
Given $(\bar,K,A)$ output $(\bar a',K',A')$ with
$(a_j'):=\big((j+1)\cdot a_{j+1}\big)$ and $K':=2K$ and 
$A'$ some integer $\geq \tfrac{A}{r}\cdot(1+\tfrac{2K}{e\ln2})$.
Indeed, $1<r':=(2)^{\tfrac{1}{K'}}=\sqrt{r}<r$ implies
$j\leq2K\cdot\sqrt{r}^{j}/(e\ln2)$
according to Lemma~\ref{l:Derivative}a);
hence $(j+1)\cdot|a_{j+1}|\leq A'/\sqrt{r}^j$. \\
More generally, in view of Lemma~\ref{l:Derivative}b), 
output $\big((j+1)\cdot(j+2)\cdots(j+d)\cdot a_{j+d}\big)_{_j}$ 
and $K':=2K$ and 
$A'\geq \tfrac{A}{r^d}\cdot d^d\cdot (1+2K)^d
\geq \tfrac{A}{r^d}\cdot(1+\tfrac{2Kd}{e\ln2})\cdot
(2+\tfrac{2Kd}{e\ln2})\cdots(d+\tfrac{2Kd}{e\ln2})$.
\item[e)]
Given $(\bar a,K,A)$, in case $d=1$ output
$a_0':=0$ and $a_j':=a_{j-1}/j$ and $K':=K$ and $A'\geq A\cdot r$. 
In the general case 
$a_j'=a_{j-d}/j/(j-1)/\cdots/(j-d+1)$ and
$A'\geq A\cdot r^d$.
\item[f)]
First suppose that $f|_{[-1;1]}$ is real, i.e. $a_j\in\IR$.
Similar to a), the first $d:=\calO(n\cdot K+\log A)$ terms of the series
yield a polynomial $p\in\ID_{n+1}[X]$ of $\deg(p)<d$ with dyadic
coefficients uniformly
approximating $f$ up to error $2^{-n-1}$. In particular
it suffices to approximate the maximum of $p$ on $[u',v']$
up to $2^{-n-1}$ (for $u',v'\in\ID$ 
sufficiently close to $u$ and $v$, respectively).
This can be achieved by bisection on $y$
w.r.t. the following existentially quantified formula 
in the first-order equational theory of the reals 
with dyadic parameters which,
involving only a constant number of polynomials and quantifiers,
can be decided in time polynomial in the degree and 
binary coefficient length \mycite{Exercise~11.7}{Basu}:
\[ \Phi(u',v',p_0,\ldots,p_{d-1}) \;\;:= \;\; ``\exists x,r,s,t\in\IR: \;\;
\underbrace{x=u'+r^2}_{\geq u'} \;\wedge\; 
\underbrace{v'=x+s^2}_{x\leq v'} \;\wedge\; 
\underbrace{p(x)=y+t^2}_{p(x)\geq y} "  \]
In the general case of a complex valued $f|_{[-1;1]}$,
$|f|^2=\Re(f)^2+\Im(f)^2$ is uniformly approximated by
the real polynomial $q:=\Re(p)^2+\Im(p)^2$, thus
$\max|f|^2$ is polytime computable as above. Since both 
$\IR\ni t\mapsto t^2$ and $\IR_+\ni s\mapsto\sqrt{s}$
are monotonic and polytime computable,
the same follows\footnote{%
We thank \person{Robert Rettinger} for pointing this out 
during a meeting in Darmstadt on August 22, 2011} 
for $\max|f|=\sqrt{\max|f|^2}$.
\item[g)]
According to Cauchy's differentiation formula (\ref{e:Cauchy}),
Equation~(\ref{e:Hadamard}) is satisfied with $A:=B$.
Polytime computability of the sequence $a_j=f^{(j)}(0)/j!$
from evaluations of $f|_{[-1;1]}$ is due \cite{Mueller87};
cmp. also the proof of \mycite{Theorem~6.9}{Ko91} and
(that of) Theorem~\ref{t:Gevrey}a) below.
and note that both $\Re f$ and $\Im f$ are real analytic.
Since $|f'(x)|\leq A/(r'-1)^2\leq A\cdot(\tfrac{2K}{\ln2})^2$
according to a) and Lemma~\ref{l:Derivative}c),
we known that $\mu(n)=n+\lceil\log A+2\log\tfrac{2K}{\ln2}\rceil$
is a modulus of continuity of $f|_{[-1;1]}$
and can suffice with evaluations on the dense subset
$[-1;1]\cap\ID$.
\qed\end{enumerate}\end{proof}

\COMMENTED{%
\begin{myproposition} \lab{p:Polynomial}
\begin{enumerate}
\item[a)] minimize convex functions
\item[b)] Given $d\in\IN$ and $a_0,\ldots,a_d\in\IR$,
\note{Can we improve the running time from $A$ to $\log A$?}
the maximum $[-1;1]$ of $x\mapsto\sum_{j=0}^d a_jx^j$
is computable within time polynomial in $A+d+n$,
where $A:=\max\{|a_0|,\ldots,|a_d|\}$.
\end{enumerate}
\end{myproposition}

Note that Sch\"{o}nhage's polytime finder 
\cite{Schoenhage} requires monic polynomials and
thus does not seem applicable
to locate the zeroes of $p'$ as
putative positions of extrema.
}

\subsection{Representing, and Operating on, Analytic Functions on $[0;1]$} \lab{s:Analytic}
We now consider functions $f\in\Comega[0;1]$, that is functions
analytic on some complex neighbourhood of $[0;1]$. 
Being members of $\Lip[0;1]$, the
second-order representation $\dyrhoLip$ applies. On the other hand,
$f$ is covered by finitely many power series, each naturally
encoded via $\pi$. And Equation~\ref{e:Hadamard1} suggests yet
yet another encoding:

\begin{mydefinition} \lab{d:Analytic}
Let $\Comegax{[0;1]}$ denote the space of complex-valued functions 
analytic on some complex neighbourhood of $[0;1]$.
\begin{enumerate}
\item[a)]
Let $M\in\IN$ and $x_m\in[0;1]$ 
and $a_{m,j}\in\IC$ and $L_m\in\IN$ and $A_m\in\IN$
($1\leq m\leq M$, $j\in\IN$). \\ We say that
$\big(M,(x_m),(a_{m,j}),(L_m),(A_m)\big)$
\emph{represents} $f\in \Comegax{[0;1]}$ if it holds
\begin{equation} \label{e:Cover}
[0;1]\subseteq\bigcup\limits_{m=1}^M \big[x_m-\tfrac{1}{4L_m},x_m+\tfrac{1}{4L_m}\big]
\;\;\text{and}\;\;
f^{(j)}(x_m)=a_{m,j}\cdot j!
\;\;\text{and}\;\;
|a_{m,j}|\leq A_m\cdot L_m^j
\end{equation}
An $\alpha$--name of $f$ encodes ($M$ in unary and,
jointly in the sense Remark~\ref{r:TTE}c) 
for each $1\leq m\leq M$ the following:
a $\rhody$--name of $x_m$, a $(\rhody^2)^\omega$--name
of $(a_{m,j})_{_j}$ as well as advice $A_m$ in binary and $L_m$ in unary
with parameter $\langle \log A_1,{L_1},\log A_2,{L_2},\ldots,\log A_m,{L_m}\rangle$.
\item[b)]
Define second-order representation $\tilde\beta$ to encode $f$ via
(see Figure~\ref{f:domain}b)
\begin{itemize}\itemsep0pt%
\item a $\dyrho$--name of $\Re f|_{[0;1]}$
\item and one of $\Im f|_{[0;1]}$, together with 
\item an integer $L$ in unary such that $f\in\Comegax{\barR_L}$, \\ where
$\barR_L:=\{ x+iy \mid -\tfrac{1}{L}\leq y\leq\tfrac{1}{L}, -\tfrac{1}{L}\leq x\leq 1+\tfrac{1}{L} \}$
\item and a binary integer upper bound $B$ to $|f|$ on said $\barR_L$.
\end{itemize}
\item[c)]
A $\tilde\gamma$--name of $f\in \Comegax{[0;1]}$ 
consists of a $\dyrho$--name of $\Re f|_{[0;1]}$ and one of $\Im f|_{[0;1]}$
enriched with advice parameters $A$ (in binary) and $K$ (in unary) such that 
$|f^{(j)}(x)|\leq A\cdot K^j\cdot j!$ holds for all $0\leq x\leq 1$.
\end{enumerate}
\end{mydefinition}
Again, the binary and unary encodings have been chosen carefully:
Analogous to $K$ in Definition~\ref{d:PowerSeries} upper bounding $1/(R-1)$,
the distance of the domain $\cball(0,1)$ to a singularity,
here $1/L$ constitutes a lower bound
on the distance of $[0;1]$ to any complex singularity of $f$.
Specifically the size of a $\tilde\beta$--name $\psi$ of $f$ 
is $|\psi|(n)=\Theta(n+L+\log B)$;
and that of a $\tilde\gamma$--name is $\Theta(n+K+\log A)$.

\begin{myexample} \label{x:Poles}
\begin{enumerate}
\item[a)]
For $x_m\in[0;1]$ and $y_m>0$ ($1\leq m\leq M$) 
the function 
$z\mapsto \prod_m \big((z-x_m)^2+y_m^2\big)^{-1}$
is analytic on $[0;1]$ with complex singularities
at $x_m\pm iy_m$.
\item[b)]
The Gaussian function $g_1(x)=\exp(-x^2)$ 
employed in the proof of Example~\ref{x:Param3}f)
has $g_1^{(j)}(x)=(-1)^j\cdot H_j(x)\cdot g_1(x)$ 
with the \textsf{Hermite Polynomials} 
\[ H_0=1, \quad H_1(x)=2x, \quad H_{n+1}(x)=2x\cdot H_n(x)-2n\cdot H_{n-1}(x) \enspace . \]
A simple uniform bound on $g_1^{(j)}$ is obtained
using Equation~(\ref{e:Cauchy}):
\[ |g_1^{(j)}(x)| \;\leq\; \tfrac{j!}{2\pi} \cdot \int\nolimits_{|z-x|=1} |\exp(-z^2)|/1^j\,dz
\;\leq\; j! e \]
because $|\exp(-z^2)|=\exp\big(-\Re^2(z)+\Im^2(z)\big)\leq\exp(1)$
due to $|z-x|=1$ with $x\in\IR$.
\end{enumerate}\end{myexample}
Hence in both cases, $f|_{[0;1]}$ being large/steep or a complex
singularity residing close-by, more time is (both needed and)
granted for polytime calculations on $f$.
Similarly for the parameterized representation $\alpha$.
Also note that a $\tilde\gamma$--name of $f$ encodes only data
on the restriction $f|_{[0;1]}$ whereas 
both $\alpha$ and $\tilde\beta$ explicitly refer
to the complex differentiable $f$ with open domain;
cmp. \cite{RealAnalytic}. 

\begin{myremark} \lab{r:Nonunif}
\begin{enumerate}
\item[a)]
Both $\tilde\beta$ and $\tilde\gamma$ enrich $\dyrho$--names
with different
discrete information (and thus without affecting the nonuniform
complexity where $n$ is considered the only parameter)
of the kind commonly omitted
in nonuniform claims and are thus 
candidates for uniformly refining Example~\ref{x:Upper}.
\item[b)]
We could (and in Section~\ref{s:Gevrey} will) combine
in the definition of $\tilde\gamma$ 
the binary $A$ with unary $K$ into one single
binary $C:=A\cdot 2^K$ satisfying $\|f^{(j)}\|\leq C\cdot (\log C)^j\cdot j!$
\end{enumerate}
\end{myremark}

\begin{theorem} \lab{t:Analytic}
\begin{enumerate}
\item[a)]
On $\Comegax{[0;1]}$, $\alpha$ and $\tilde\beta$ and $\tilde\gamma$
constitute mutually (fully) polytime-equivalent (parameterized) representations.
\item[b)]
The following operations from 
Theorem~\ref{t:PowerSeries}
are in fact uniformly polytime computable
on $\Comegax{[0;1]}$: \\
Evaluation (i), addition (ii), multiplication (iii), 
iterated differentiation $(\sdone^d,f)\mapsto f^{(d)}$ (iv), 
parametric integration (v), 
and parametric maximization (vi). 
\item[c)]
Composition (vii), that is the partial operator
\[ (g,f)\mapsto g\circ f \quad\text{ for }\quad
f,g\in\Comegax{[0;1]}
\quad\text{ with }\quad
f\big|_{[0;1]}\in\Comegax{[0;1],[0;1]} \enspace , \]
is fixed-parameter computable in the sense of
Definition~\ref{d:dualuse}h). More precisely in terms of 
$(\tilde\beta\times\tilde\beta,\tilde\beta)$--computability,
given integers $A,K$ with $|f(z)|\leq A$ on $\barR_K$
and $(L,B)$ with $|g(z)|\leq B$ on $\barR_L$,
$g\circ f$ is analytic on $\barR_{2AKL}$ and thereon bounded (like $g$ itself) by $B$.
Note that (only) the unary length of $2AKL$ is exponential in 
(only) the binary length of $A$.
\\
Similarly concerning
$(\tilde\gamma\times\tilde\gamma,\tilde\gamma)$--computability,
$f$ with advice parameters $(A,K)$ 
and $g$ with advice parameters $(B,L)$ 
gets mapped to $g\circ f$ with advice parameters
$\big(\tfrac{A\cdot B\cdot L}{1+A\cdot L},K\cdot(1+A\cdot L)\big)$,
where the unary length of (only) the latter is exponential in (only)
the binary length of $A$.
\end{enumerate}
\end{theorem}
Recall that, since both $\tilde\beta$--names and $\tilde\gamma$--names
have length $\calO(n+k)$ linear in $n$ with constant parameter $k$,
second-order polynomials here boil down to ordinary polynomials
in $n+k$ (Observation~\ref{o:SecondVsFirst}a+b+c).


\begin{proof}[Theorem~\ref{t:Analytic}]
\begin{description}
\item[a)]
Since $f$ is complex analytic on some $U\supseteq[0;1]$
open in $\IC$ there exists an $L\in\IN$ as required
for a $\tilde\beta$--name. The continuous
$|f|$ is bounded on compact $\barR_L$ by some $B$.
\item[$\tilde\beta\reduceP\tilde\gamma$:]
The $\dyrho$--names as part of the desired $\tilde\gamma$--name
are already contained in the $\tilde\beta$--name.
Now observe that Cauchy's differentiation formula
(\ref{e:Cauchy}) 
implies $|f^{(j)}(x)|/j!\leq B\cdot L^j$
for all $j\in\IN$ and all $x\in[0;1]$.
Hence $A:=B$ and $K:=L$ are suitable choices.
\item[$\tilde\gamma\reduceP\alpha$:]
Given $A,K$ set 
$L_m:=K$ and $A_m:=A$ and $M:=4K+1$ and $x_m:=(m-1)/(4K)$, 
noting that 
$(M,L_1,\ldots,L_M,A_1,\ldots,A_M,x_1,\ldots,x_M)$ can be
computed within second-order polytime:
with $A_m$ encoded in binary and $L_m$ unary,
the output has length $\calO\big(M\cdot(K+\log A)\big)$ 
compared to the unary encoding length
$\Omega(K+\log A)$ of the input. We claim that also
the sequences $a_{m,j}=f^{(j)}(x_m)/j!$
can be obtained within time a
second-order polynomial in the size of the
given $\tilde\gamma$--name.
In order to apply Theorem~\ref{t:PowerSeries}g)
observe that, since $|a_{m,j}|\leq A\cdot K^j$,
the translated and scaled function
$\hat f_m(z):=f\big(x_m+\tfrac{z}{2K}\big)$ 
can be evaluated efficiently on $\ID$ 
using the $\tilde\varphi$--name of $f$; and 
is analytic on $\ball(0,2)$ with Taylor coefficients 
$\hat a_{m,j}:=\hat f_m^{(j)}(0)/j!=f^{(j)}(x_m)/(2K)^j/j!$
bounded by $A/2^j$. 
Hence $\big((\hat a_{m,j})_{_j},(\hat K:=1),A\big)$ 
constitutes an $\alpha$--name of $\hat f_m$;
from which $a_{m,j}=\hat a_{m,j}\cdot(2K)^j$ can be recovered.
\item[$\alpha\reducep\tilde\beta$:]
Let $\big(M,(x_m),(a_{m,j}),(L_m),(A_m)\big)$ 
denote an $\alpha$--name of $f$.
From $|a_{m,j}|\leq A_m\cdot L_m^j$ observe
that $f$ is analytic on 
$\bigcup_m\ball(x_m,\tfrac{1}{L_m})$;
which constitutes an open neighbourhood of
$\barR_{L}$ for $L:=2\cdot\max_m L_m$ because of
$[0;1]\subseteq\bigcup_{m=1}^M \big[x_m-\tfrac{1}{4L_m},x_m+\tfrac{1}{4L_m}\big]$;
and $f$ is bounded on $\barR_L$ 
by $\max\big\{|\sum_j a_{m,j} (z-x_m)^j|:m\leq M, z\in\ball(x_m,\tfrac{1}{4L_m}\big\}
\leq\max_{m}\sum_j A_m L_m^j \cdot \tfrac{1}{4L_m}^j=\tfrac{4}{3}\max_m A_m$.
Both $L$ and $B\geq\tfrac{4}{3}\max_mA_m$ can easily be calculated 
in time polynomial in the parameter $k=\Theta\big(\sum_m (L_m+\log A_m)\big)$. 
It thus remains to prove $\alpha\reducep\dyrho$
as follows: \\
Given $x\in[0;1]$, some (of at least two) 
$m$ with $x\in J_m:=\big[x_m-\tfrac{1}{2L_m},x_m+\tfrac{1}{2L_m}\big]$ 
can be found within time polynomial in $\sum_m\log L_m$:
due to the factor-two overlap, it suffices to know
$x$ up to error $\tfrac{1}{4L_m}$. 
Finally observe that $\hat f_m(z):=f\big(x_m+\tfrac{z}{2L_m}\big)$ 
is analytic on $\ball(0,2)$ with coefficients  
$\hat a_{m,j}=f_m^{(j)}(0)/j!=a_{m,j}/(2L_m)^j$
satisfying $|\hat a_{m,j}|\leq A_m/2^j$; 
hence can be evaluated on $\cball(0,1)$ 
by Theorem~\ref{t:PowerSeries}a)
within time polynomial in $K+\log A_m+n$ for $K:=1$.
This provides for the evaluation 
within time polynomial in $n$ and $k$ of 
$J_m\ni x\mapsto\hat f_m\big((x-x_m)\cdot 2L_m\big)=f(x)$,
i.e. of $f|_{[0;1]}$.
\item[b)]
Note that in view of a) we may freely choose among representations
$\alpha,\tilde\beta,\tilde\gamma$ output
and even any combination of them for input.
\item[i)]
Evaluation is provided by the $\dyrho$--information
contained in a $\tilde\gamma$--name of $f$
together with the Lipschitz bound
$|f'(x)|\leq AK$ of binary length
polynomial in $k=\Theta(K+\log A)$.
\item[ii)]
Given $\dyrho$--names of $f_1,f_2$ 
and $L_1,L_2$ in unary and binary $B_1,B_2$
according to $\tilde\beta$,
output a $\dyrho$--name of $f:=f_1+f_2$
(Example~\ref{x:Param3}a)
and $L:=\max(L_1,L_2)$ and $B=B_1+B_2$.
\item[iii)] Similarly with $f:=f_1\cdot f_2$
and $L:=\max(L_1,L_2)$ and $B:=B_1\cdot B_2$.
\item[iv)]
Given an $\alpha$--name
$\big(M,(x_m),(a_{m,j}),(L_m),(A_m)\big)$ of $f$ and in case $d=1$,
(the proof of) Theorem~\ref{t:PowerSeries}d) yields $A_m'$ 
such that the coefficients 
$\hat a'_{m,j}=(j+1)\cdot\hat a_{m,j+1}=(j+1)\cdot a_{m,j+1}/(2L_m)^{j+1}$ 
of $\hat f'_m(z)=\tfrac{d}{dz}f\big(x_m+\tfrac{z}{2L_m}\big)$
satisfy $|\hat a'_{m,j}|\leq A_m'/\sqrt{2}^j$.
Note that the bounds 
$|a'_{m,j}|\leq A_m'\cdot L'^j_m$ on
$a'_{m,j}=\hat a'_{m,j}\cdot(2L_m)^j$ thus obtained
apply to $L'_m=\sqrt{2}L_m$ and hence may fail the
covering property of Equation~(\ref{e:Cover}).
Nevertheless they do support efficient evaluation
of $\hat f'_m$ on $\cball(0,1)$ 
by Theorem~\ref{t:PowerSeries}a), but now with $K:=2$.
Given $x\in[0;1]$,
similarly to the above proof of 
``$\alpha\reducep\tilde\beta$''
some $J_m$ containing $x$
can be found efficiently and used to calculate
$f'(x)=\hat f'_m\big((x-x_m)\cdot 2L_m\big)\cdot2L_m$.
This provides for a $\dyrho$--name of $f'$; 
as part of a $\tilde\gamma$--name to output.
And for given $(K,A)$ satisfying $|f^{(j)}(x)|\leq A\cdot K^j\cdot j!$,
Lemma~\ref{l:Derivative}a) implies
$|(f')^{(j)}(x)|=|f^{(j+1)}(x)|\leq AK\cdot(j+1)\cdot K^j\cdot j!
\leq AK\cdot (2K)^j\cdot j!$ since $\tfrac{1}{e\ln2}\leq1$;
hence $(AK,2K)$ yields the rest of a $\tilde\gamma$--output.
In case $d>1$, similarly,
$|f^{(d+j)}(x)|\leq AK^d\cdot(j+1)\cdots(j+d)\cdot K^j\cdot j!
\leq A\cdot(dK)^d\cdot(2K)^j\cdot j!$ according to
Lemma~\ref{l:Derivative}b); hence output
$\big(A\cdot(dK)^d,2K\big)$ of binary length 
polynomial in $d$ (sic!) and $K+\log A$.
\item[vi)]
Given $0\leq u\leq v\leq 1$ and an $\alpha$--name 
$\big(M,(x_m),(a_{m,j}),(L_m),(A_m)\big)$ of $f$,
the idea is to partition $[u;v]$ into $K\leq \calO(m)$ 
sub-intervals $I_k=[u_k;v_k]$ each lying completely within 
some $J_m$ and then apply Theorem~\ref{t:PowerSeries}f)
to $\hat f_m(z):=f\big(x_m+\tfrac{z}{2L_m}\big)$
in order to obtain $y_k:=\max\{f(x):x\in I_k\}$
and finally 
$\max\{f(x):u\leq x\leq v\}=\max\{y_k:k\leq K\}$.
\\
It follows from the proof of $\tilde\gamma\reduceP\alpha$
that we may w.l.o.g. suppose 
$L_m\equiv L$ and $(x_m)$ increasing with equidistance $\leq\tfrac{1}{4L}$.
So find some $1\leq m\leq m'\leq M$ with
$u\in J_m$ and $v\in J_{m'}$.
Then $u_1:=u$ and $v_1:=\min\big(v,x_m+\tfrac{1}{2L}\big)$
yield the first interval;
the next are given by $u_{k+1}:=v_{k}$
and $v_{k+1}:=x_{m+k}+\tfrac{1}{2L}\geq u_{k+1}$ ($1\leq k\leq m'-m-1$);
and the last by $u_{m'-m}:=v_{m'-m-1}$ and
$v_{m'-m}:=\max\big(u,x_{m'}-\tfrac{1}{2L}\big)$.
\item[v)]
Similarly to vi) but now with
$y_k=\int_{u_k}^{v_k} f(x)\,dx=F(v_k)-F(u_k)$ 
a local anti-derivative of $f$
according to Theorem~\ref{t:PowerSeries}e)
and $\int_u^v f(x)\,dx =\sum_ky_k$.
\item[vii)]
Looking at $(\tilde\beta\times\tilde\beta,\tilde\beta)$--computability,
in view of Example~\ref{x:Param3}b) it remains
to obtain, given integers $(L,B)$ with $|g(z)|\leq B$ on $\barR_L$
and $A,K$ with $|f(z)|\leq A$ on $\barR_K$,
similar quantities for $g\circ f$.
To this end employ Cauchy's differentiation formula
(\ref{e:Cauchy}) to deduce
$|f'(z)|\leq A\cdot(2K)$ on $\barR_{2K}$.
Now the \textsf{Mean Value Theorem} (in $\Re z$ and $\Im z$ 
considered as two real variables!) implies
$|f(x)-f(x+z)|\leq 2AK\cdot|z|\leq 1/L$ 
for $x\in[0;1]$ and $z\in\barR_{2AKL}$. 
Together with the hypothesis of $f$ mapping $[0;1]$ to $[0;1]$,
this shows $f$ to map $\barR_{2AKL}$ to $\barR_L\subseteq\dom(g)$.
Therefore $g\circ f$ is analytic on some open neighbourhood of
$\barR_{2AKL}$ and thereon bounded (like $g$ itself) by $B$. 
\\
Concerning $(\tilde\gamma\times\tilde\gamma,\tilde\gamma)$--computability,
the proof of \mycite{Proposition~1.4.2}{RealAnalytic} using the formula
of \person{Fa\'{a} di Bruno} (cmp. below Fact~\ref{f:Lombardi}f) shows that 
$|f^{(j)}|\leq A\cdot K^j\cdot j!$ and
$|g^{(j)}|\leq B\cdot L^j\cdot j!$ imply
$|(g\circ f)^{(j)}|\leq C\cdot M^j\cdot j!$ where
$C:=\frac{A\cdot B\cdot L}{1+A\cdot L}$ 
and $M:=K\cdot(1+A\cdot L)$.
\qed\end{description}\end{proof}

\subsection{Representing, and Operating on, Entire Functions} \lab{s:Entire}
As has been kindly pointed out by \person{Torben Hagerup}, 
the above Theorems~\ref{t:PowerSeries} and \ref{t:Analytic}
both do not capture the important case of the exponential function 
on entire $\IR$ (Example~\ref{x:Param2}a) or $\IC$. 

\begin{mydefinition} \lab{d:Entire}
On the space 
$\Comega_\infty=\big\{\bar a\subseteq\IC, R(\bar a)=\infty\big\}$,
a $\entire$--name $\psi$ of $\bar a=(a_j)_{_j}\in\Comega_\infty$ 
is a mapping $\psi:\{0,1\}^*\to\{0,1\}^*$ of the form
\[ \sdone^n \;\mapsto \; \sigma_n\,\bin\big(B(n)\big) \] 
where $\bar\sigma$ is a $(\rhody^2)^\omega$--name of $\bar a$
and $B:\IN\to\IN$ a function satisfying 
\begin{equation} \label{e:Entire}
\forall M\in\IN \;\forall j\in\IN: \quad |a_j|\;\leq\; B(M)/M^j 
\enspace .
\end{equation}
\end{mydefinition}
Note that this second-order representation has names of super-linear length,
that is, employing the full power of second-order size parameters.
Indeed, every mapping $B:\IN\to\IN$ gives rise
to an entire function with Taylor coefficients
$a_j:=\inf_M B(M)/M^j$.

\begin{theorem} \lab{t:Entire}
\begin{enumerate}
\item[a)]
$\entire$ is a second-order representation,
and it holds $\entire\big|^{\Comega_1}\reduceP\pi$
\item[b\,i)]
Evaluation 
\[ \Comega_\infty\times\IC\;\ni\; (\bar a,z) \;\mapsto\;
\sum\nolimits_j a_j z^j \;\in\;\IC \]
is $\big(\entire,(\rhody^2,|\rhody^2|),\rhody^2\big)$--computable
in second-order polytime.
\item[b\,ii+iii)]
Addition and multiplication (i.e. convolution) on $\Comega_\infty$
are $\big(\entire\times\entire,\entire\big)$--computable
in second-order polytime.
\item[b\,iv+v)]
Iterated anti-/differentiation $\Comega_\infty\times\IN\;\ni\;(f,d)\mapsto f^{(\mp d)}$ 
is $\big(\entire+\unary(d),\entire\big)$--computable in second-order polytime.
\item[b\,vi)]
Parametric maximization 
according to Theorem~\ref{t:PowerSeries}f)
of entire functions on real line segments
(that is on $\Comega_\infty\times\IR^2$)
is $\big(\entire,(\rhody^d,|\rhody^2|),\rhody\big)$--computable
in second-order polytime.
\item[b\,vii)] 
Composition $\Comega_\infty\times\Comega_\infty\ni (f,g)\mapsto g\circ f\in\Comega_\infty$
is fixed-parameter $\big(\entire\times\entire,\entire\big)$--computable. 
More precisely, given $\bar a$ and $B=B(M)$ with 
$|a_j|\leq B(M)/M^j$ and $f(z):=\sum_j a_j z^j$
as well as $\bar b$ and $A=A(N)$ with 
$|b_i|\leq A(N)/N^i$ and $g(w):=\sum_i b_i w^i$,
$c_k:=\big(g\circ f\big)^{(k)}(0)/k!$ has
$|c_k|\leq 2A\big(4 B(2M)\big)/M^k$.
Note that converting $N:=4B(2M)$ from binary to
unary for invoking $A(N)$ is the (only) step incurring
exponential behaviour.
\end{enumerate}
\end{theorem}
%
In the case of the exponential function with $a_j=1/j!$,
in order to bound
$M^j/j!=\tfrac{M}{j}\cdot\tfrac{M}{j-1}\cdots\tfrac{M}{1}$
independent of $j$, it suffices to treat the case $j=M$
since additional factors with $j>M$ only decrease the product.
Now according to \textsf{Stirling's Approximation},
$B(M):=M^M/M!\leq\calO\big(\exp(M)\big)$ grows singly exponentially,
i.e. has binary length (coinciding with $|\psi|(M)-1$)
linear in the value (!) of $M$. 
Example~\ref{x:Param2}a) therefore indeed constitutes a special
case of Theorem~\ref{t:Entire}b\,i). 
%
\begin{proof}[Theorem~\ref{t:Entire}]
\begin{enumerate}
\item[a)]
Since an entire function is holomorphic on every disc $\ball(0,M+1)$,
Equation~\ref{e:Hadamard} applies to every $r:=M$; hence
$B()$ exists. Concerning the reduction,
it suffices to take $K:=1$ and $A:=B(2)$ 
in the $\pi$--name.
\item[b\,i)]
Similar to the proof of Theorem~\ref{t:PowerSeries}a),
evaluate the polynomial $\sum_{j<N} a_jz^j$ with 
$N:=n+1+\log A$ terms, where $A:=B(M)$ and $M\geq2|z|$:
clearly feasible within time polynomial in $n+|z|$ and $B()$,
that is the (second-order) length 
of the given $\entire$--name $\psi$ of $\bar a$.
\item[b\,ii)]
For addition transform $(a_j)$ and $(b_j)$ into
$(a_j+b_j)$ as well as 
$M\mapsto B(M)$ and $M\mapsto A(M)$ 
into $M\mapsto B(M)+A(M)$.
\item[b\,iii)]
$c_j:=\sum_{i=0}^j a_i\cdot b_{j-i}$ 
satisfies $|c_j|\leq\sum_i B(2M)/(2M)^i\cdot A(2M)/(2M)^{j-i}
=(j+1)\cdot B(2M)\cdot A(2M)/(2M)^j\leq C/M^j$ for all $j$,
where $C$ is some integer 
$\geq A(2M)\cdot B(2M)\cdot\tfrac{2}{e\ln 2}$
since, according to Lemma~\ref{l:Derivative}a),
$j+1\leq\tfrac{1}{e\ln 2}\cdot2^{j+1}$.
\item[b\,iv)]
Similarly, according to Lemma~\ref{l:Derivative}b),
$\big|(j+1)\cdot(j+2)\cdots(j+d)\cdot a_{j+d}\big|
\leq (j+d)^d\cdot B(2M)/(2M)^{j+d}
\leq B'(M)/M^j$ for
$\IN\ni B'(M)\geq B(2M)\cdot\big(\tfrac{2d}{e\ln 2}\big)^d/(2M)^d$.
\item[b\,v)]
Observe 
$|\tfrac{1}{j}\cdot\tfrac{1}{j-1}\cdots\tfrac{1}{j-d+1}\cdot a_{j-d+1}|
\leq B(M)/M^{j-d+1}$.
\item[b\,vi)]
As in b\,i) and the proof of Theorem~\ref{t:PowerSeries}f),
use real quantifier elimination to
approximate the maximum of the polynomial
consisting of the first $n+2+\log B(M)$ terms of $f$'s Taylor expansion.
\item[b\,vii)]
Instead of explicitly expanding 
$\sum_i b_i\big(\sum_j a_j z^j\big)^i$,
note similarly to Lemma~\ref{l:Derivative}e) that
$w:=f(z)$ has $|w|\leq 2\cdot B(2M)$
for $|z|\leq M$; and analogously 
$|g(w)|\leq 2A\big(4 B(2M)\big)=:C(M)$ 
for $|w|\leq 2B(2M)$.
Then Cauchy's Differentiation Formula (\ref{e:Cauchy})
yields the claim.
\qed\end{enumerate}\end{proof}

\section[Uniform Complexity on Gevrey's Scale from Real Analytic to Smooth Functions]{Complexity on Gevrey's Scale from Real Analytic to Smooth Functions} 
\lab{s:Gevrey}
This section explores in more detail the complexity-theoretic `jump' of
the operators of maximization and integration from smooth ($\calNP$--hard: Fact~\ref{f:Nonunif}) 
to analytic (polytime: Example~\ref{x:Upper}) functions. More precisely we present
a uniform\footnote{Referring to \mycite{D\'{e}finition~2.2.9}{Lombardi},
\mycite{Corollaire~5.2.14}{Lombardi} establishes `sequentially uniform'
polytime computability of the operators in the sense of mapping
every polytime sequence of functions to a polytime sequence.
This may be viewed as a complexity theoretic counterpart to
Banach--Mazur computability, cmp. e.g. \mycite{\S9.1}{Weihrauch}.}
refinement of \mycite{\S5.2}{Lombardi} asserting these operators to
map polytime to polytime functions on a class\footnote{We thank 
\person{Matthias Schr\"{o}der} for directing us to this class and to the publication \cite{Lombardi}.}  
much larger than $\Cinfty$ (and even than the \emph{quasi-}analytic functions)
which, historically, arose from the study 
of the regularity of solutions to partial differential equations \cite{Gevrey}:

\begin{mydefinition} \lab{d:Lombardi}
\begin{enumerate}
\item[a)]
Write $\Gevrey_{\ell,A,K}[-1;1]$ for the subclass
of functions $f\in\Cinfty[-1;1]$ satisfying
\begin{equation} \label{e:Gevrey1}
\forall |x|\leq1,\; \forall j\in\IN: \qquad 
|f^{(j)}(x)|\;\leq\; A\cdot K^j\cdot j^{j\ell}  \enspace ;
\end{equation}
$\Gevrey_\ell[-1;1]:=\bigcup_{A,K\geq1}\Gevrey_{\ell,A,K}[-1;1]$
and $\Gevrey[-1;1]:=\bigcup_{\ell\in\IN}\Gevrey_\ell[-1;1]$. 
\item[b)]
Let $\tilde\lambda$ denote the following 
second-order representation of $\Gevrey[-1;1]$: \\
A mapping $\psi$ is a $\tilde\lambda$--name of $f\in\Gevrey_{\ell}[-1;1]$ if
there exists $C\in\IN$ such that, for every $n\in\IN$,
$\psi(\sdone^n)$ is (the binary encoding of) an
$m$--tuple $(a_0,\ldots,a_{m-1})\in\ID_m$, 
$m=C\cdot n^{\ell}$,
such that $\hat g_m(x):=a_0+a_1 x+\cdots+a_{m-1}x^{m-1}$
has $\|f-\hat g_m\|\leq2^{-n}$.
\item[c)]
Let a $\tilde\gamma$--name of $f\in\Gevrey_{\ell}[-1;1]$ be a mapping
\[ \{\sdzero,\sdone\}^*\;\ni\;\vec w
\;\mapsto\;\sdone^{\log A+K+|\vec w|^\ell}\,\sdzero\,
\psi(\vec w)\;\in\{\sdzero,\sdone\}^* \enspace ,\]
where $\psi$ denotes a $\dyrho$--name of $f$
satisfying Equation~(\ref{e:Gevrey1}). We regard $\tilde\gamma$ as 
second-order representation with unary advice parameter $K+\log A$ 
in the sense of Definition~\ref{d:dualuse}c).
\end{enumerate}
\end{mydefinition}
For $\ell=1$, Definition~\ref{d:Lombardi}a) is equivalent to 
Equation~(\ref{e:Hadamard1}) by virtue of Stirling, i.e.,
$\Comega[-1;1]=\Gevrey_1[-1;1]$ holds.

\begin{myexample} \lab{x:Gevrey}
\begin{enumerate}
\item[a)]
The standard example $h(x)=\exp(-1/|x|)$ of a smooth but 
non-analytic function belongs to $\Gevrey_3[-1;1]$.
\item[b)]
The function $h_L$ from Example~\ref{x:maxNP}
is smooth but not in $\Gevrey[-1;1]$.
\item[c)]
Extending Example~\ref{x:Param3}f),
fix $\ell,N\in\IN$ and let
$g_{\ell,N,k}(x):=g_1\big(N^\ell\cdot x-k\big)/e^{1+N\cdot\ell/e}$
with $g_1(x)=\exp(-x^2)$. Then $g_{\ell,N,k}\in\Gevrey_{\ell+1,1,1}[-1,1]$.
\end{enumerate}
\end{myexample}
\begin{proof}
\begin{enumerate}
\item[a)]
Similarly to the proof of Example~\ref{x:Poles}b),
write $h^{(n)}(x)=x^{-2n}\cdot h(x)\cdot p_n(x)$.
So $p_0=1$ and, for $n\geq1$, $p_{n}(x)=x^2\cdot p'_{n-1}(x)+\big(1-2(n-1)x\big)\cdot p_{n-1}(x)$
shows $p_{n}$ to be an integer polynomial of $\deg(p_n)=n-1$
with leading coefficient equal to $(-1)^{n+1}\cdot n!$ 
and each coefficient bounded by $n!$.
In particular $\|p_n\|\leq \calO(n)^n$,
while $(0;1]\ni x\mapsto x^{-2n}\cdot g(x)$
attains its maximum at $x=\tfrac{1}{2n}$ of value
$\calO(n)^{2n}$: hence $\|h^{(n)}\|\leq\calO(n)^{3n}$.
\item[b)]
Recall that $h_L=\sum_{N\in L}h_N$,
where the $h_{N}(x)=h(2N^2\cdot x-2N)/N^{\ln N}$
have pairwise disjoint supports.
Therefore $\|h^{(j)}_{L}(x)\| 
= \sup_{N\in L} h^{(j)}_N =
\sup_{N\in L} \|h_1^{(j)}\|\cdot 2N^{2j}/N^{\ln N}$,
where $\sup_{N\in L} N^{2j}/N^{\ln N}\geq 
(2^j)^{2j}/(2^j)^{j\ln 2}\geq 2^{j^2}$.
\item[c)]
Note $\tfrac{d^j}{dx^j} g_{\ell,N,k}(x)=N^{\ell j}/\exp(1+N\ell/e)
\cdot g_1^{(j)}\big(N^\ell\cdot x-k\big)$
with $|g_1^{(j)}(y)|\leq j!\cdot e\leq j^j\cdot e$
according to Example~\ref{x:Poles}b);
and $0=\tfrac{d}{dN}\ln \big(N^{\ell j}/\exp(N\ell/e)\big)=
\tfrac{d}{dN} \ell j\ln N-N\ell/e=\ell j/N-\ell/e$ 
shows by monotonicity of $\ln$ that 
$\sup_N N^{\ell j}/\exp(N\ell/e)$ is attained
at $N=je$ of value $j^{\ell j}$.
\qed\end{enumerate}\end{proof}
As did Fact~\ref{f:Nonunif} for $\Comega$,
\mycite{Corollaire~5.2.14}{Lombardi} asserts
$\Max$ and $\int$ and $\partial$ to map polytime functions 
in $\Gevrey[-1;1]$ to polytime ones -- again nonuniformly, 
that is
for fixed $f$ and in particularly presuming $\ell,A,K$ 
according to Definition~\ref{d:Lombardi}a) to be known.
The representation $\tilde\gamma$ from Definition~\ref{d:Lombardi}c) 
on the other hand explicitly provide the values of these quantities.
More precisely by `artificially' padding names to length
$\log A+K+n^\ell$, instances $f$ with large values of these parameters 
are allotted more time to operate on in the second-order setting.
Our main result (Theorem~\ref{t:Gevrey}) will show these
parameters to indeed characterize the uniform computational complexity
of maximization in terms of Gevrey's scale of smoothness.
Nonuniformly, we record

\begin{myremark} \lab{r:Eike}
For fixed $\ell,A,K\in\IN$, a function $f\in\Gevrey_{\ell,A,K}$
is polytime $(\rhody,\rhody)$--computable iff it has a
polynomial-time computable $\tilde\gamma$--name $\psi$.
\end{myremark}
Before further justifying Definition~\ref{d:Lombardi}b)+c) 
--- including the re-use of $\tilde\gamma$ ---
let us recall some facts from Approximation Theory 
heavily used also in \cite{Lombardi,Lester}:

\begin{fact} \lab{f:Lombardi}
For a ring $R$, abbreviate with
$R[X]_m$ the $R$--module of all univariate 
polynomials over $R$ of degree $<m$.
Let $\Cheby_m\in\IZ[X]_m$ denote the $m$-th Chebyshev polynomial of the first kind,
given by the recursion formula $\Cheby_0\equiv1$,
$\Cheby_1(x)=x$, and $\Cheby_{m+1}(x)=2x\Cheby_m(x)-\Cheby_{m-1}(x)$.
\begin{enumerate}
\item[a)]
For $g\in\IR[X]_{m+1}$ it holds $\|g'\|\leq m^2\cdot\|g\|$ and
$\|g^{(k+1)}\|\leq \frac{m^2\cdot(m^2-1^2)\cdot(m^2-2^2)\cdots(m^2-k^2)}{
1\cdot 3\cdot 5\cdots (2k+1)}\cdot\|g\|$. 
\item[b)]
With respect to the scalar product 
\[ \langle f,g\rangle \;:=\; \int\nolimits_{-1}^1 f(x)\cdot g(x)\cdot(1-x^2)^{-1/2}\,dx 
\;=\; \int\nolimits_0^\pi f(\cos t)\cdot g(\cos t)\,dt \]
on $C[-1;1]$, the family $\calT=(T_m)_{_m}$
of Chebyshev polynomials forms an orthogonal system, namely
satisfying $\langle\Cheby_0,\Cheby_0\rangle=\pi$ 
and $\langle\Cheby_m,\Cheby_m\rangle=\pi/2$
and $\langle\Cheby_m,\Cheby_n\rangle=0$ for $0\leq n<m$.
The orthogonal projection (w.r.t. this scalar product) 
$\OCHEBY{m}(f)$ of $f\in C[-1;1]$ onto $\IR[X]_m$ is given by
\[ \OCHEBY{m}(f):=c_{f,0}/2+\sum\nolimits_{k=1}^{m-1} c_{k}(f)\cdot \Cheby_k, \qquad
c_{k}(f):=\tfrac{2}{\pi}\cdot\langle f,\Cheby_k\rangle \]
Moreover $\|f-\OCHEBY{m}(f)\|\leq\sum_{k\geq m} |c_k(f)|$ and
$|c_{m}(f)|\leq \tfrac{4}{\pi}\cdot\|f-g\|$ for every $g\in\IR[X]_m$.
\item[c)]
The unique polynomial $\ICHEBY{m}(f)\in\IR[X]_m$ 
interpolating $f\in C[-1;1]$ at the Chebyshev Nodes
$x_{m,j}:=\cos\big(\tfrac{\pi}{2}\tfrac{2j+1}{m}\big)$, $0\leq j<m$,
is given by
\[ \ICHEBY{m}(f):=\sum\limits_{k=0}^{m-1} y_{m,k}(f)\cdot \Cheby_k, \quad
y_{m,k}(f):=\tfrac{1}{m}\cdot f(x_{m,0})\cdot \Cheby_k(x_{m,0})+
\tfrac{2}{m}\cdot\sum\limits_{j=1}^{m-1} 
f(x_{m,j})\cdot \Cheby_k(x_{m,j}) \]
and `close' to the best polynomial approximation in the following sense:
\[ \forall g\in\IR_m: \quad \|f-\ICHEBY{m}(f)\| \;\leq\; 
\big(2+\tfrac{2}{\pi}\log m\big)\cdot\|f-g\| \enspace . \]
\item[d)]
To $f\in C^k[-1;1]$ and $k<m\in\IN$ there exists $g\in\IR[X]_m$ such that 
\begin{equation} \label{e:Jackson}
\|f-g\| \;\leq\; (\tfrac{\pi}{2})^k\cdot 
\frac{\|f^{(k)}\|}{m\cdot(m-1)\cdots(m-k+1)} \enspace . 
\end{equation}
\item[e)]
If $g_m$ converges pointwise to $f$, and if all the $g_m$ are differentiable,
and if the derivatives $g'_m$ converge uniformly to $g$, 
then $f$ is differentiable and $f'=g$.
\item[f)]
The \textsf{Formula of Fa\`{a} di Bruno} expresses
higher derivatives of function composition:
\[
(g\circ f)^{(n)}(t) \;=\hspace*{-4ex}
\sum_{\substack{k_1,\ldots,k_n\in\IN_0 \\ k_1+2k_2+\cdots+nk_n=n}} \hspace*{-1ex}
\frac{n!}{k_1!\cdot k_2!\cdots k_n!}\cdot g^{(k)}\big(f(t)\big)\cdot 
\Big(\frac{f^{(1)}(t)}{1!}\Big)^{k_1}\cdot
\Big(\frac{f^{(2)}(t)}{2!}\Big)^{k_2}\cdots
\Big(\frac{f^{(n)}(t)}{n!}\Big)^{k_n}  \]
where $k:=k_1+k_2+\cdots+k_n$.
In particular
$\sum_{k_1,\ldots,k_n}
\tfrac{k!}{k_1!\cdot k_2!\cdots k_n!}\cdot R^k=R\cdot(1+R)^{n-1}$
for $n\in\IN$ and $R>0$.
\item[g)]
According to \textsf{Stirling}, 
$\sqrt{2\pi}\cdot n^{n+1/2}\cdot e^{-n}\leq n!\leq
n^{n+1/2} \cdot e^{1-n}$.
\end{enumerate}
\end{fact}
Claim~a) are the Markov brothers' inequalities.
Claim~b) is mostly calculation plus \cite[\textsc{Theorem~4.4.5}iii]{Cheney}
and \mycite{\S5.2.2}{Lombardi}.
The nontrivial part of Claim~c) bounds the Lebesgue Constant;
cmp. \mycite{Theorem~1.2}{Rivlin}.
For d) refer, e.g., to \cite[\textsc{Jackson's Theorem~4.6.V}iii]{Cheney}.
Claim~f) can for instance be found in \mycite{\S1.3+\S1.4}{RealAnalytic}.

The following tool provides quantitative refinements to
\mycite{Th\'{e}or\`{e}me~5.2.4}{Lombardi}.

\begin{myproposition} \lab{p:Lombardi}
\begin{enumerate}
\item[a)]
For $0<r<1$ and $0<q\leq1$ and $p\geq0$ and $N\in\IN$ it holds
\[ \sum\nolimits_{n>N} r^{n^q}\cdot (n-1)^p \;\;\leq\;\;
\tfrac{1}{2q\cdot\ln1/r}\cdot
\big(2\cdot\tfrac{1-p+q}{q\cdot e\cdot \ln1/r}\big)^{(1-q+p)/{q}}
\cdot\sqrt{r}^{N^q} \]
\item[b)]
Let $0<q,r<1\leq B$. Then $B\cdot r^{m^q}\leq 2^{-n}$ holds for all
$m\geq C\cdot n^{1/q}$, 
where $C:=\big(1+\log B/\log\tfrac{1}{r}\big)^{1/q}$.  
Conversely, for $0<q<1\leq C$,
$\big(\forall m\geq C\cdot n^{1/q}:\varepsilon_m\leq 2^{-n}\Big)$
implies $\varepsilon_m\leq 2\cdot r^{m^q}$ 
with $\tfrac{1}{2}<r:=2^{-C^{-q}}<1$.
\item[c)]
Suppose $f\in\Cinfty[-1;1]$ satisfies Equation~(\ref{e:Gevrey1})
with parameters $A,K,\ell\geq1$. Then to every $m\in\IN$
there exists $g_m\in\IR[X]_m$ with
\begin{equation} \label{e:Lombardi}
\|f-g_m\|\;\leq\; B\cdot r^{m^q}
\end{equation}
where $q:=1/\ell$, $2^{-1/(2\pi)}\leq r:=2^{-(2\pi K)^{-q}}<1$, $B:=A\cdot(2+\pi K)$.
\item[d)]
Conversely suppose $0<r,q<1\leq B$ and $f:[-1;1]\to\IR$ are such 
that to every $m\in\IN$ there exists some $g_m\in\IR[X]_m$ 
satisfying Equation~(\ref{e:Lombardi}). 
Then $f$ is smooth and satisfies $\|f^{(d)}-g_m^{(d)}\|\leq B_d\cdot \sqrt{r}^{m^q}$
with $B_d:=B\cdot\big(\tfrac{6}{q\cdot e\cdot \ln1/r}\big)^{(1+2d)/q}$. 
Moreover $f$ obeys Equation~(\ref{e:Gevrey1}) with $\ell:=2/q-1$,
$A:=B\cdot\big(\tfrac{6}{q\cdot e\cdot\ln1/r})^{4/q}$,
and $K:=4\cdot\big(\tfrac{6}{q\cdot\ln1/r}\big)^{2/q}$.
\item[e)]
Suppose $f\in\Cinfty[-1;1]$ satisfies Equation~(\ref{e:Gevrey1}).
To $m\in\IN$ let $\hat g_m\in\ID_m[X]_m$ denote the polynomial 
$\ICHEBY{m}(f)$ with coefficients `rounded' to $\ID_m$.
Then it holds $\|f-\hat g_m\|\leq \hat B\cdot \sqrt{r}^{m^{1/\ell}}$
with $\hat B:=1+2B\cdot\big(\tfrac{2\ell}{e\cdot\ln1/r}\big)^{\ell}$
for $r$ and $B$ according to c).
\item[f)]
For $f\in\Gevrey_\ell[-1;1]$, $f'\in\Gevrey_\ell[-1;1]$.
More precisely, Equation~(\ref{e:Gevrey1}) implies
$\|f^{(d+j)}\|\leq A_d\cdot K_d^j\cdot j^{j\ell}$ 
for $A_d:=A\cdot L\cdot(2\cdot d^2\cdot\ell/e)^{d\ell}$ 
and $K_d:=K\cdot e\cdot (2d)^\ell$.
\item[g)]
For $f,g\in\Gevrey_\ell[-1;1]$, $f\cdot g\in\Gevrey_\ell[-1;1]$.
More precisely, 
$\|f^{(j)}\|\leq A\cdot K^j\cdot j^{j\ell}$ and
$\|g^{(j)}\|\leq B\cdot L^j\cdot j^{j\ell}$ imply
$\|(f\cdot g)^{(j)}\|\leq C\cdot M^j\cdot j^{j\ell}$ 
for $C:=A\cdot B$ and $M:=K+L$.
\item[h)]
For $f,g\in\Gevrey_\ell[-1;1]$ with $f:[-1;1]\to[-1;1]$, 
$g\circ f\in\Gevrey_\ell[-1;1]$.
More precisely, 
$\|f^{(n)}\|\leq A\cdot K^n\cdot n^{n\ell}$ and
$\|g^{(n)}\|\leq B\cdot L^n\cdot n^{n\ell}$ imply
$\|(g\circ f)^{(n)}\|\leq C\cdot M^n\cdot n^{n\ell}$ 
for $C:=B\cdot(e\cdot\ell/2)^{\ell/2}$ and 
$M:=e\cdot A\cdot K\cdot L$.
\end{enumerate}
\end{myproposition}
Note that, in case $\ell=1=q$, Claims~c) and d) together characterize 
$\Gevrey_1[-1;1]=\Comega[-1;1]$ in terms of function
approximability by polynomials \cite{Demanet}
but leave a gap in cases $q<1<\ell$; cmp.
\mycite{Remarques~5.2.5(4)}{Lombardi}.
As a consequence, the proof of Theorem~\ref{t:Gevrey}a) below will show
$\tilde\gamma\big|^{\Gevrey_\ell}\reduceP\tilde\lambda\big|^{\Gevrey_\ell}$
uniformly in $\ell$
but `only' $\tilde\lambda\big|^{\Gevrey_\ell}\reduceP\tilde\gamma\big|^{\Gevrey_{2\ell-1}}$:
with consequences to the proof of Theorem~\ref{t:Gevrey}b) below.

We postpone the proof of Proposition~\ref{p:Lombardi} in order to first state 
our main

\begin{theorem} \lab{t:Gevrey}
\begin{enumerate}
\item[a)]
Up to second-order polytime equivalence,
Definition~\ref{d:Lombardi}c) generalizes and extends
Definition~\ref{d:Analytic}c) from $\Comega[-1;1]$ to $\Gevrey[-1;1]$.
Moreover, $\tilde\gamma$ is second-order polytime equivalent to $\tilde\lambda$
constituting another second-order representation of $\Gevrey[-1;1]$.
\item[b)]
The following operations are uniformly polytime 
$\tilde\lambda$--computable on $\Gevrey_{\ell}[-1;1]$: \\
Evaluation (i), addition (ii), multiplication (iii), 
iterated differentiation (iv), 
parametric integration (v), 
and parametric maximization (vi). 
\item[c)]
Composition (vii) as the partial operator
\[ (g,f)\mapsto g\circ f \quad\text{ for }\quad
f,g\in\Gevrey_{\ell}[-1;1]
\quad\text{ with }\quad f:[-1;1]\to[-1;1]  \]
is fixed-parameter $(\tilde\gamma\times\tilde\gamma,\tilde\gamma)$--computable. 
\item[d)]
Fix $\ell\in\IN$. Then 
$(\dyrho,\rhody)$--computing $\Max$ restricted to
$\Gevrey_{\ell+1,1,1}[-1;1]$ requires time at least
$\Omega(n^\ell)$.
\end{enumerate}
\end{theorem}
\begin{proof}[Theorem~\ref{t:Gevrey}]
\begin{enumerate}
\item[a+b\,i)]
Every $f\in\Gevrey_\ell[-1;1]$ has a $\tilde\gamma$--name $\psi$
of length $|\psi|(n)=\log A+K+n^\ell$. For $f\in\Comega[-1;1]$
a second-order polynomial in $n$ and $n\mapsto\log A+K+n$ amounts,
to a bivariate polynomial in $n$ and $k:=\log A+K$;
recall Observation~\ref{o:SecondVsFirst}b). Moreover
$k$ can be recovered from $\psi$; and any bound
$A\cdot K^j\cdot j^j$ according to Equation~(\ref{e:Gevrey1})
with $k=\log A+K$ 
is in turn dominated by $A'\cdot K'^j\cdot j!$ for
$A':=2^k/\sqrt{2\pi}$ and $K':=k\cdot e$
by virtue of \textsf{Stirling} (Fact~\ref{f:Lombardi}f)
with both $A'$ in binary and $K'$ in unary computable
from $\psi$ within second-order polytime;
hence Definition~\ref{d:Lombardi}c) indeed extends
Definition~\ref{d:Analytic}c).
\\
In order to see $\tilde\gamma\reduceP\tilde\lambda$,
recall that $\|f'\|\leq A\cdot K$ implies
evaluation $\Gevrey[-1;1]\times[-1;1]\ni (f,x)\mapsto f(x)$
be second-order polytime $(\tilde\gamma,\rhody,\rhody)$--computable
according to Example~\ref{x:Representation2}b)+e)+f).
Now note that, for $B',r$ from Proposition~\ref{p:Lombardi}c+e),
$C$ according to the first part of
Proposition~\ref{p:Lombardi}b) is polynomially
bounded in $\log A+K$; hence interpolating $f$ 
on $m=C\cdot n^\ell$ Chebyshev nodes 
$x_{m,0},\ldots,x_{m,m-1}$ from Fact~\ref{f:Lombardi}d) 
is feasible within time a second-order polynomial in $n$
and $n\mapsto \log A+K+n^\ell$.
\\
Conversely, the polynomial $\hat g_m$ in a 
$\tilde\lambda$--name $\psi$ of $f$ yields approximate evaluation
of $f$ within time a second-order polynomial in the output
error $n$ and the length $|\psi|(n)\approx m^2=(C\cdot n^\ell)^2$
and in particular allows to obtain a $\dyrho$--name. 
Moreover the parameters $C$ (in unary) and $\ell$ 
are easily recovered and lead to $\ell':=2\ell-1$ 
and $A,K$ satisfying Equation~(\ref{e:Gevrey1})
according to the second part of 
Proposition~\ref{p:Lombardi}c) and Proposition~\ref{p:Lombardi}e);
in fact both the binary length of $A$ and the value 
of $K$ (i.e. in unary) are polynomially bounded in (the value of) 
$C$ and exponentially in $\ell$, hence allowing also the map
$\vec w\mapsto\sdone^{\log A+K+|\vec w|^{\ell'}}$ 
to be computed within time a second-order polynomial
in $n$ and in $n\mapsto (C\cdot n^\ell)^2$.
\item[b\,ii)]
Given sequences $\hat f_n\in\ID_{m_f(n)}[X]_{m_f(n)}$ 
with $m_f(n)=C_f\cdot n^\ell$ and $\|f-\hat f_n\|\leq 2^{-n}$
as well as
$\hat g_n\in\ID_{m_g(n)}[X]_{m_g(n)}$ 
with $m_g(n)=C_g\cdot n^\ell$ and $\|g-\hat g_n\|\leq 2^{-n}$,
output $\hat f_{n+1}+\hat g_{n+1}\in\ID_{m_{f+g}(n)}[X]_{m_{f+g}(n)}$
where $m_{f+g}(n):=\max\{C_f,C_g\}\cdot n^\ell$.
\item[b\,iii)]
Given a $\tilde\lambda$--name $\psi_f$ of $f$,
observe that $\|f\|\leq 1+\|\hat f_0\|\leq 1+|a_{0,0}|+|a_{1,0}|+\cdots+|a_{C_f,0}|
\leq 2^{\calO(|\psi_f|(0)|)}$ where $\psi_f(\sdone^n)$ 
encodes the coefficients of $\hat f_n\in\ID_{m_f(n)}[X]_{m_f(n)}$ in binary
with $m_f(n)=C_f\cdot n^\ell$.
Now in view of
$\|f\cdot g-\hat f_n\cdot\hat g_n\|\leq \|f\|\cdot\|g-\hat g_n\|+\|f-\hat f_n\|\cdot\|\hat g_n\|$
it suffices to output $\hat f_{n'}\cdot\hat g_{n'}\in\ID_{m_{f\cdot g}(n')}[X]_{m_{f\cdot g}(n')}$ 
for $n'=\calO\big(n+|\psi_f|(0)+|\psi_g|(0)\big)$,
where $m_{f\cdot g}(n)=m_f(n)+m_g(n)=(C_f+C_g)\cdot n^\ell$.
\item[b\,iv)] 
Given a sequence $\hat f_n\in\ID_{m(n)}[X]_{m(n)}$ with
$\|f-\hat f_n\|\leq2^{-n}$ for $m(n)=C\cdot n^\ell$,
combine Proposition~\ref{p:Lombardi}b+d) to see that
$\hat f^{(d)}_{N(n,\ell,C,d)}\in\ID_{M(n,\ell,C,d)}[X]_{M(n,\ell,C,d)}$ satisfies 
$\|f-\hat f^{(d)}_{N(n,\ell,C,d)}\|\leq2^{-n}$ for 
$N,M$ polynomials in $n^\ell$, $C$, $\ell^\ell$, $d^\ell$
and in particular bounded by second-order polynomials 
in $d$ and in $n\mapsto C\cdot n^\ell$, i.e., in the input size.
Indeed: $C$ leads,
according to the second part of Proposition~\ref{p:Lombardi}b),
to $B:=2$ and $r:=2^{-C^{1/\ell}}$;
then furtheron to $B_d:=2\cdot\big(\tfrac{6}{q\cdot e\cdot\ln1/r}\big)^{(1+2d)\ell}$ 
and $\sqrt{r}$ according to the first part of Proposition~\ref{p:Lombardi}d);
and finally, according to the first part of Proposition~\ref{p:Lombardi}b),
to $C_d:=(1+\tfrac{\log B_d}{\log1/\sqrt{r}})^\ell$ 
polynomially bounded in $d^\ell$, $\ell^\ell$, 
and $C\geq\Omega\big((\log C)^\ell\big)$.
\item[b\,v)]
Given a sequence $\hat f_n\in\ID_{m(n)}[X]_{m(n)}$,
$\|f-\hat f_n\|\leq2^{-n}$,
it obviously holds $\|\int f-\int\hat f_n\|\leq2^{-n}$.
\item[b\,vi)]
Similarly, maximize $\hat f_n$ as in the proof of
Theorem~\ref{t:PowerSeries}f).
\item[c)]
The estimate $\|f'\|\leq A\cdot K$ yields
$\tilde\gamma\reduceP\dyrhoLip$ and thus a 
$\dyrho$--name of $g\circ f$ according to Example~\ref{x:Param3}b).
The binary parameter $\log C+M$ for $g\circ f$ can be obtained from
$\log A+K$ of $f$ and from $\log B+L$ of $g$ due to
Proposition~\ref{p:Lombardi}h) and are independent of $n$.
\item[d)]
Similarly to the proof of 
Example~\ref{x:Param3}f),
recall the functions $g_{\ell,N,k}\in\Gevrey_{\ell+1,1,1}[-1;1]$
from Example~\ref{x:Gevrey}c) 
satisfying $g_{\ell,N,k}(k/N^\ell)=\exp(N\cdot\ell/e+1)$.
and $g_{\ell,N,k}(x)\leq\exp(N\cdot\ell/e)$
for $|x-k/N^\ell|\geq 1/N^\ell$.
Hence any algorithm computing $f\mapsto\max\{f(x)\mid0\leq x\leq 1\}$ 
on $\{0,g_{\ell,N,0},\ldots,g_{\ell,N,N-1}\}\subseteq\Gevrey_{\ell+1,1,1}[-1;1]$
up to error $2^{-n}$, $n=:(2+N\cdot\ell/e)/\ln e$,
must distinguish (every name of) the identically zero function from 
(all names of) each of the $g_{\ell,N,k}$ ($0\leq k<N^\ell$).
Yet, since the $f_{\ell,N,k}$ are `thin', any approximate 
evaluation up to error $2^{-m}$ 
at some $x$ with $|x-k/N^\ell|\geq m/N^\ell$.
could return $0$. 
\qed\end{enumerate}\end{proof}

\begin{proof}[Proposition~\ref{p:Lombardi}]
\begin{enumerate}
\item[a)]
Since $n-1\leq x\leq n$ implies
$r^{n^q}\leq r^{x^q}$ and $(n-1)^p\leq x^p$,
it follows 
$r^{n^q}\cdot (n-1)^p \leq \int\nolimits_{n-1}^n r^{x^q}\cdot x^p\,dx$ 
and 
$\sum\nolimits_{n>N} r^{n^q}\cdot (n-1)^p \leq
\int\nolimits_N^\infty r^{x^q}\cdot x^p\,dx$.
In the latter integral substitute $y:=x^q\cdot\ln(1/r)$,
ranging from $M:=N^q\cdot\ln(1/r)$ to $\infty$;
moreover $x=\big(\frac{y}{\ln1/r}\big)^{1/q}$ 
and
$\tfrac{dy}{dx}=q\cdot\ln(1/r)\cdot x^{q-1}
=q\cdot y/x=y^{1-1/q}\cdot q\cdot\ln(1/r)^{1/q}$.
The integral thus transforms into
\[ \int\nolimits_M^\infty e^{-y}
\cdot \Big(\frac{y}{\ln 1/r}\Big)^{p/q}
\cdot y^{1/q-1} \cdot \tfrac{1}{q} \cdot (\ln\tfrac{1}{r})^{-1/q}\,dy 
\;=\;
\tfrac{1}{q} \cdot \big(\ln\tfrac{1}{r}\big)^{-(1+p)/q} \cdot \int_M^\infty
e^{-y}\cdot y^{s}\,dy \]
with $s:=(1+p)/q-1\geq0$.
According to Lemma~\ref{l:Derivative}b),
$y^s\leq (\tfrac{s}{e\ln\sqrt{e}})^{s}\cdot e^{y/2}$;
yielding the bound 
\[ \int_M^\infty e^{-y}\cdot y^{s}\,dy
\;\leq\; \big(\tfrac{2s}{e}\big)^s \cdot e^{M/2}/2, \qquad
e^{M/2}=\sqrt{r}^{N^q} \enspace . \]
\item[b)]
Taking binary logarithms on both sides shows the claim equivalent to
$m\geq\big(\tfrac{n+\log B}{\log 1/r}\big)^{1/q}$, 
and the latter is $\leq C\cdot n^{1/q}$. \\
Given $m$, considering $n':=\lfloor(m/C)^q\rfloor\geq (m/C)^q-1$,
the largest $n$ with $m\geq C\cdot n^{1/q}$, implies
$\varepsilon_m\leq 2^{-n'}\leq 2^{1-(m/C)^q}=2\cdot r^{m^q}$.
\item[c)]
Following \cite[p.324]{Lombardi},
Fact~\ref{f:Lombardi}d) implies
$\|f-g_m\|\leq A\cdot \big(\tfrac{\pi K}{2m}\cdot k^\ell\big)^k$
for every $k\in\IN$. Now choosing $k:=(\tfrac{m}{2\pi K})^{1/\ell}$
would yield the bound $A\cdot r^{m^q}$; but $k$ must be integral.
To this end observe that the claim does hold in case $k=1$ 
by means of $B\geq A$;
and for non-integral $k>1$ 
(where necessarily $\tfrac{\pi K}{2m}\cdot k^\ell<1$),
there exists an integer $k'$ between $k-1$ and $k$:
yielding $(\tfrac{\pi K}{2m}\cdot k'^\ell)^{k'}
\leq\big(\tfrac{\pi K}{2m}\cdot k^\ell\big)^{k-1}=2r^{m^q}$.
\item[d)]
Observe that triangle equality yields,
in connection with Fact~\ref{f:Lombardi}a), 
$\|g^{(d)}_m-g^{(d)}_{m-1}\|\leq (m-1)^{2d}\cdot 2B\cdot r^{m^q}$; 
hence, for $N\geq n$,
\[ \|g^{(d)}_{N}-g^{(d)}_n\|\;=\; \Big\|\sum\nolimits_{m=n+1}^N g^{(d)}_m-g^{(d)}_{m-1}\Big\|
\;\leq\; 2B\cdot \sum\nolimits_{m>n} (m-1)^{2d}\cdot r^{m^q} \enspace . \]
Now apply a) to conclude
$\|g^{(d)}_{N}-g^{(d)}_n\|\leq B_d\cdot\sqrt{r}^{m^q}$ independent of $N$,
showing $(g^{(d)}_m)_{_m}$ to be a Cauchy sequence in $C[-1;1]$ 
with uniform limit $f^{(d)}$ according to Fact~\ref{f:Lombardi}e).
\\
On the other hand, Fact~\ref{f:Lombardi}b) asserts
$|c_{m}(f)|\leq B\cdot\tfrac{4}{\pi}\cdot r^{m^q}$;
from which a) together with, again, Fact~\ref{f:Lombardi}b) implies 
$f=\lim_m\OCHEBY{m}(f)=c_0(f)/2+\sum_{m\geq1} c_{m}(f)\cdot \Cheby_m$
with respect to uniform convergence, and thus (Fact~\ref{f:Lombardi}e)
\[
\|f^{(j)}\| \;=\; \big\|\sum\nolimits_{m\geq j} c_{m}(f)\cdot \Cheby_m^{(j)}\big\|
\;\leq\; \sum\nolimits_{m\geq j} |c_m(f)|\cdot m^{2j}/j! \]
according to Fact~\ref{f:Lombardi}a),
observing $1\cdot 3\cdot 5\cdots (2j-1)\geq j!$.
Now with $m^{2j}/j!\leq \big(2\cdot(m-1)\big)^{2j}/j!+1$ for $m,j\in\IN$
we can invoke a) for $p=2j$ and for $p=0$ to continue bounding 
\begin{eqnarray*} \|f^{(j)}\| &\leq&
B\cdot\tfrac{4}{\pi} \cdot 4^j \cdot\tfrac{1}{2q\cdot\ln1/r}
\cdot\Big(\big(2\tfrac{1-p+2j}{q\cdot e\cdot\ln1/r}\big)^{({1-q+2j})/{q}}/j!
+\big(2\tfrac{1-p}{q\cdot e\cdot \ln1/r}\big)^{(1-q)/{q}}\Big)
\cdot\sqrt{r}^{(j+1)^q} \\
&\leq& B\cdot C\cdot\tfrac{4e}{3\pi}\cdot 4^j\cdot
(C\cdot j)^{(2-2q+2j)/{q}}/j!,
\qquad C:=6/\big(qe\ln\tfrac{1}{r}\big)
\end{eqnarray*}
since $1-q\leq 1\leq j\leq 3j$ and by generously replacing the
sum of the two large terms with their product.
Moreover $C\cdot C^{(2-2q)/q}\leq C^{2/q}$,
$j^{(2-2q)/q}\leq j^{2/q}\leq \big(\tfrac{2}{qe}\big)^{2/q}\cdot e^j$
by virtue of Lemma~\ref{l:Derivative}b),
and $j!\geq\sqrt{2\pi}\cdot j^{j+1/2}\cdot e^{-j}$
due to \textsf{Stirling} (Fact~\ref{f:Lombardi}f).
\item[e)]
By c) and Fact~\ref{f:Lombardi}c),
$\|f-\hat g_m\|\leq B\cdot r^{m^{1/\ell}}\cdot
(2+\tfrac{2}{\pi}\log m)+m\cdot 2^{-m}$
since the rounding changes $m$ coefficients by $\leq2^{-m}$.
Now $2+\tfrac{2}{\pi}\log m\leq 2\cdot\log(2m)\leq 2m
\leq 2\cdot\big(\tfrac{2\ell}{e\cdot\ln1/r}\big)^{\ell}\cdot (1/\sqrt{r})^{m^{1/\ell}}$
because $n:=m^{1/\ell}$ has $m=n^{\ell}\leq \big(\tfrac{\ell}{e\cdot\ln1/\sqrt{r}}\big)^{\ell}\cdot\sqrt{r}^{n}$
according to Lemma~\ref{l:Derivative}b).
Similarly, Lemma~\ref{l:Derivative}a) yields $m\leq(2\sqrt{r})^m$ 
because $r\geq2^{-1/(2\pi)}$ implies $\tfrac{1}{e\cdot\ln(2\sqrt{r})}\leq1$;
hence $m\cdot2^{-m}\leq \sqrt{r}^m$.
\item[f)]
By hypothesis
$\|f^{(d+j)}\|\leq A\cdot L^{d+j}\cdot (d+j)^{(d+j)\ell}$;
and $d+j\leq 2dj$ implies 
$(d+j)^{(d+j)\ell}\leq j^{d\ell}\cdot (2d)^{d\ell}\cdot \big((2d)^\ell\big)^j$
where $j^{d\ell}\leq (d\ell/e)^{d\cdot\ell}\cdot e^j$
by virtue of Lemma~\ref{l:Derivative}b).
\item[g)]
According to the \textsf{General Leibniz Rule},
\[ 
\|(f\cdot g)^{(j)}\| 
\;\leq\; 
A\cdot B\cdot \underbrace{\sum\nolimits_{k=0}^j \binom{j}{k} K^k\cdot L^{j-k}}_{=(K+L)^j}
\cdot \underbrace{k^{k\ell}\cdot (j-k)^{(j-k)\ell}}_{\leq j^{k\ell}\cdot j^{(j-k)\ell}=j^{j\ell}} \]
\item[h)]
Following \mycite{\S I.2.1}{Gevrey}, we first record that induction on $n$ 
in the Fa\`{a} di Bruno's Formula (Fact~\ref{f:Lombardi}f) 
shows (i) the coefficients $\displaystyle\tfrac{n!}{k_1!\cdot 1!^{k_1}\cdot k_2!\cdot 2!^{k_2}
\cdots k_n!\cdot n!^{k_n}}$ to all be nonnegative integers. 
Secondly (ii), according to Stirling,
\begin{multline*}
\sum\nolimits_{k_1,\ldots,k_n}
\tfrac{k^k}{k_1!\cdot k_2!\cdots k_n!}\cdot R^k \cdot 
({1^1}/{1!})^{k_1}\cdot
({2^2}/{2!})^{k_2}\cdots
({n^n}/{n!})^{k_n}  \\
\leq\quad
\sum\nolimits_{k_1,\ldots,k_n}
\tfrac{k!}{k_1!\cdot k_2!\cdots k_n!}\cdot (e\cdot R)^k/\sqrt{2\pi} 
\cdot (e^1/\sqrt{2\pi})^{k_1}\cdot
\cdot (e^2/\sqrt{2\pi})^{k_2}\cdots
\cdot (e^n/\sqrt{2\pi})^{k_n} \\
=\quad
e^n/\sqrt{2\pi}\cdot (e\cdot R/\sqrt{2\pi})\cdot (1+e\cdot R/\sqrt{2\pi})^{n-1}
\;\leq\; (e\cdot R)^n
\end{multline*}
for $R\geq1$ by virtue of Fact~\ref{f:Lombardi}f).
Now we can bound $\|(g\circ f)^{(n)}\|$ with
\begin{eqnarray*}
\lefteqn{\sum_{k_1,\ldots,k_n} 
\tfrac{n!}{k_1!\cdot k_2!\cdots k_n!}\cdot
(B\cdot L^k\cdot k^{\ell k})
\cdot (A\cdot K^1\cdot \tfrac{1^{1\ell}}{1!})^{k_1}
\cdot (A\cdot K^2\cdot \tfrac{2^{2\ell}}{2!})^{k_2}
\cdots (A\cdot K^n\cdot \tfrac{n^{n\ell}}{n!})^{k_n}} \\
&=& 
B\cdot K^n\cdot \sum_{k_1,\ldots,k_n} 
\tfrac{n!}{k_1!\cdot 1!^{k_1}\cdot k_2!\cdot 2!^{k_2}
\cdots k_n!\cdot n!^{k_n}}
\cdot \big(\sqrt[\ell]{A\cdot L}^k
\cdot k^k \cdot 
({1^1})^{k_1}\cdot
({2^2})^{k_2} \cdots
({n^n})^{k_n}\big)^{\ell}  \\
&\overset{\text{(i)}}{\leq}&
B\cdot K^n\cdot \sum_{k_1,\ldots,k_n} 
\big(\tfrac{n!}{k_1!\cdot 1!^{k_1}\cdot k_2!\cdot 2!^{k_2}
\cdots k_n!\cdot n!^{k_n}}\big)^\ell
\cdot \big(\sqrt[\ell]{A\cdot L}^k
\cdot k^k \cdot 
({1^1})^{k_1}\cdot
({2^2})^{k_2} \cdots
({n^n})^{k_n}\big)^{\ell}  \\
&\leq&
B\cdot K^n\cdot \Big(
\sum_{k_1,\ldots,k_n} 
\tfrac{n!}{k_1!\cdot 1!^{k_1}\cdot k_2!\cdot 2!^{k_2}
\cdots k_n!\cdot n!^{k_n}}
\cdot \sqrt[\ell]{A\cdot L}^k
\cdot k^k \cdot 
({1^1})^{k_1}\cdot
({2^2})^{k_2} \cdots
({n^n})^{k_n}\Big)^{\ell}  \\
&\overset{\text{(ii)}}{\leq}& 
B\cdot K^n\cdot \big(n!\cdot 
(e\cdot \sqrt[\ell]{A\cdot L})^n\big)^\ell
\end{eqnarray*}
where $n!^{n\ell}\leq n^{n\ell}\cdot (n^{1/2}\cdot e)^\ell\cdot e^{-n\ell}$
and $n^{\ell/2}\leq \sqrt{\ell/(2e)}^{\ell}\cdot e^n$ 
according to Stirling and Lemma~\ref{l:Derivative}b).
\qed\end{enumerate}\end{proof}

\section{Conclusion and Perspectives}
We have constructed a parameterized and two second-order representations 
of the space $\Comegax{[0;1]}$ of analytic functions on $[0;1]$;
shown them second-order polytime equivalent; and to render
the basic primitives second-order polytime computable. 
In view of Remark~\ref{r:Nonunif}
this subsumes the known nonuniform results from Example~\ref{x:Upper};
but now 
\begin{enumerate}
\item[a)] explicitly specifies the additional discrete information employed
\item[b)] and an asymptotic running time analysis taking into account 
  both the dependence on the output precision $n$ and 
  $f$ (i.e. parameters according to a)
\item[c)] which turns out to be (ordinary) polynomial ---
  except for composition which is fixed-parameter tractable
  but may increase one parameter exponentially.
\end{enumerate}
These results entail a-priori judgement of whether, and on which inputs $f$,
the algorithms underlying Theorems~\ref{t:PowerSeries} and \ref{t:Analytic}
may be efficient in practice;
see below. And a) suggests concrete data structures for
representing real analytic functions in, say, \iRRAM 
as well as for 
\texttt{C++} interface declarations of actual
implementations of the above operators (i) to (vii).
Concerning future work we record

\begin{goals}
\begin{enumerate}
\item[a)]
Investigate the complexity of division,
that is the (partial) operator $f\mapsto 1/f$.
\item[b)]
Investigate the complexity of inversion,
that is the (partial) operator $f\mapsto f^{-1}$.
\item[c)]
Generalize all above results to the multivariate case.
\end{enumerate}
\end{goals}
Item~b) may suggest a parameterization in terms
of the \emph{modulus of unicity}; cmp.,
e.g., \mycite{Theorem~4.6}{Ko91} and
\mycite{\S16.1}{Kohlenbach}.

\subsection{Concerning Practical Efficiency}
The algorithms underlying Theorem~\ref{t:PowerSeries}
look practical and promising to implement and evaluate on.

For instance each invocation of c) and d) can essentially double
the value of $K$ and thus also the running time of subsequent 
operations. However the same applies already to iterations
of ordinary integer multiplication (repeated squaring) and hence 
seems unavoidable in the worst case. We expect some improvement 
on `typical' cases, though, from refining the bound
in Equation~(\ref{e:Hadamard}) to the more general
form $|a_j|\leq A(j)/r^j$ with $A:\IN\to\IN$ from an
appropriate function class \cite{Bitterlich}.

It may be advisable to avoid in general
invoking (the algorithm realizing)
Theorem~\ref{t:PowerSeries}g) whenever possible. For instance concerning
the proof of Theorem~\ref{t:Analytic}b\,v+vi) it may be less elegant 
but more practical to handle also non-equidistant $x_m$ and different
$L_m$. To this end consider the following alternative approach towards
identifying intervals $I_k$ partitioning $[u;v]$ and each contained
within some $J_m$:
After choosing $m$ with $u\in J_m$ and $m'$ with $v\in J_{m'}$,
output $I:=[u,u']$ where $u':=\min\big(v,x_m+\tfrac{1}{2L_m}\big)$ as before,
but also initialize $M:=\{m\}$.
While $m\neq m'$ holds (otherwise we are done covering $[u;v]$),
iteratively set $u:=u'$ and choose to this new $u$ some 
$m\not\in M$ with $u\in J_m$; again output 
$I:=[u,u']$ for $u':=\min\big(v,x_m+\tfrac{1}{2L_m}\big)$
and let $M:=M\cup\{m\}$. The restriction to $m\not\in M$ 
avoids possible cycles in degenerate cases like
$x_k+\tfrac{1}{2L_k}=x_\ell+\tfrac{1}{2L_\ell}$ for some $k\neq\ell$.
On the other hand $m\not\in M$ with $u\in J_m$ always exists:
On the one hand we have $\bigcup_m J_m\supseteq[0;1]$;
and on the other hand induction on the number of loop 
iterations shows $x_m+\tfrac{1}{2L_m}\leq u$ to hold
for all $m\in M$.

\subsection{Quantitative Complexity Theory of Operators}
Polytime computability is generally considered appropriate
a formalization of efficient tractability in practice.
However asymptotic running times like $\calO(n\log n)$ 
and $\calO(n^{999})$, although both polynomial, obviously
make a considerable difference. 
Actually specifying the exponent of polynomial growth thus
allows for a more precise estimate of the (range of) practicality 
of an algorithm: softly linear, quadratic, cubic etc.
Such an approach seems infeasible for second-order polynomial
time bounds, though, because they cannot naturally be ordered 
linearly with respect to asymptotic growth:

\begin{myexample} \label{x:Second}
With respect to $n\to\infty$,
$P(n,\ell)=\ell\big(\ell(n^2)\cdot n\big)\cdot n$
grows asymptotically slower than
$Q(n,\ell)=\big(\ell(n^3)\big)^2\cdot n^9$
in case $\ell(n)\leq\calO\big(n^{(5+\sqrt{89})/4}\big)$
and faster otherwise.  
More precisely, for $\ell$ a polynomial in $n$ of degree $d$, 
it holds $\deg P(n,\ell)=2d^2+d+1$ and $\deg Q(n,\ell)=6d+9$. 

Put differently: When algorithms $\calA$ and $\calB$ exhibit
second-order polynomial running times $P$ and $Q$, respectively,
then none is superior to the other \emph{globally}; but for a given family 
of inputs of known size $\ell$, one can predict whether $\calA$
or $\calB$ is asymptotically preferable.
\end{myexample}
This suggests to quantify the asymptotic growth of a
second-order polynomial $P(n,\ell)$ in terms of its 
degree with respect to $n$ as a (polynomial) function of $\deg(\ell)$.
To this end, the following tool may turn out as useful:

\begin{mylemma} \label{l:Composition}
\begin{enumerate}
\item[a)]
Let $\calM_1$ denote an ordinary Turing machine calculating
a partial map $f:\subseteq\{\sdzero,\sdone\}^*\to\{\sdzero,\sdone\}^*$
within time $\leq t_1(n)$ producing outputs of length at most
$s_1(n)\leq t_1(n)$; similarly for another machine
$\calM_2$ calculating $g$ 
of length at most $s_2(n)$
within time $\leq t_2(n)$.
Then there is a machine $\calM_3$
calculating the composition $g\circ f$
within time $\leq t_3(n):=t_1(n)+t_2\big(s_1(n)\big)$ 
of length at most $s_3(n):=s_2\big(s_1(n)\big)$.
\item[b)]
Let $\calM_1^?$ denote an oracle Turing machine calculating
a partial mapping $F:\subseteq\Reg\to\Reg$ within second-order
time $\leq T_1=T_1(n,\ell)$ producing, for every length-monotone
oracle $\psi$ and inputs $\vec u\in\{\sdzero,\sdone\}^n$, 
outputs of length at most $S_1(n,|\psi|)\leq T_1(n,|\psi|)$;
similarly for another machine $\calM_2^?$ calculating $G$
within time $\leq T_2(n,\ell)$ of size at most $S_2(n,\ell)$.
Then there is a machine $\calM_3^?$ calculating $G\circ F$
within time $\leq T_3$ and of size at most $S_3$, where
\[ S_3(n,\ell):=S_1\big(n,S_2(\cdot,\ell)\big) \;\;\text{and}\;\;
 T_3(n,\ell):=\tilde T(n,\ell)\cdot T_2\big(\tilde T(n,\ell),\ell\big)
, \;\; \tilde T(n,\ell):=T_1\big(n,S_2(\cdot,\ell)\big) \]
\item[c)]
Consider second-order polynomials $P=P(n,\ell)$ in first-order and 
second-order variables $n$ and $\ell$, respectively. Abbreviating
$d:=\deg(\ell)$, define the ordinary polynomial $\deg(P)\in\IN[d]$
by structural induction as follows:
$\deg(1):=0$, $\deg(n):=1$, $\deg\big(\ell(P\big)):=d\cdot\deg(P)$, 
$\deg(P+Q):=\max\big(\deg(P),\deg(Q)\big)$, and
$\deg(P\cdot Q):=\deg(P)+\deg(Q)$. 
\\
Then $P(n,\ell)=\calO\Big(n^{\deg\big(P\big)(d)}\Big)$ as $n\to\infty$. 
Moreover $\deg\Big(P\big(Q(n,\ell),\ell\big)\Big)=\deg(P)\cdot\deg(Q)$
and $\deg\Big(P\big(n,Q(\cdot,\ell)\big)\Big)=\deg(P)\circ\deg(Q)$.
In particular with the notation from b) it follows
$\deg(S_3)=\deg(S_1)\circ\deg(S_2)$ and
$\deg(\tilde T)=\deg(T_1)\circ\deg(S_2)$ and
$\deg(T_3)=\deg(\tilde T)+\deg(T_2)\cdot\deg(\tilde T)$.
\end{enumerate}
\end{mylemma}
\begin{proof}
\begin{enumerate}
\item[b)]
When presented with a length-monotone $\psi:\{\sdzero,\sdone\}^*\to\{\sdzero,\sdone\}*$
of size $\ell=|\psi|$ and on inputs $\vec a\in\{\sdzero,\sdone\}^m$, 
$\calM_2^{\psi}$ by hypothesis makes $\leq T_2(m,\ell)$ steps
and produces some output $\vec b$ of length $S_2(m,\ell)$.
In the composition
$\calM_3^{\psi}:=\calM_1^{\calM_2^\psi}$, 
$\calM_2^{\psi}$ thus serves as a string function of size
$\ell':=S_2(\cdot,\ell)$; for which $\calM_1$ converts inputs $\vec x$
of length $n$ into outputs $\vec y$ of length $S_1(n,\ell')$;
compare Figure~\ref{f:TTE2}b).
This calculation of $\calM_1$ takes at most $m:=\tilde T(n,\ell')$ steps.
In particular, $\calM_1$ can make no more than that many oracle queries
of length at most $m$, each; hence each such query, 
now fed to $\calM_2$, takes it $\leq T_2(m,\ell)$ steps to answer.
\qed\end{enumerate}\end{proof}

\subsection{Optimality Questions}
The literature on TTE provides some categorical constructions of natural
representations for certain spaces --- and can prove them optimal.
For instance $\myrho$ is known to be, up to computational equivalence, 
the only reasonable choice for the space $\IR$ \cite{Hertling};
similarly for $C[0;1]$ \mycite{\S6.1}{Weihrauch}.

Strengthening from computability to complexity on $\Comegax{[0;1]}$,
the above representations $\alpha,\tilde\beta,\tilde\gamma$
all render common primitive operations polytime computable ---
and have turned out as mutually fully polytime equivalent.
One might therefore conjecture that they are optimal in the 
sense that any second-order representation making these
operations polytime computable is in turn polytime 
reducible to $\alpha,\tilde\beta,\tilde\gamma$.

However consider the following (artificial) 
\begin{mydefinition}
A length-monotone string function $\psi:\{\sdzero,\sdone\}^*\to\{\sdzero,\sdone\}^*$ 
constitutes a $\tilde\delta$--name of $f\in\Cinfty[0;1]$ 
if, for every $d\in\IN$, the mappings
\[ \vec w\mapsto \psi(\sdzero\,\sdone^d\,\sdzero\,\vec w) \quad
\text{ and }\quad
\vec w\mapsto\psi(\sdone\,\sdone^d\,\sdzero\,\vec w) \]
constitute (padded) $\big(\rhody^{\ID^3}+\binary(\Lip)\big)$--names of
$[0;1]^2\ni (u,v)\mapsto \MAX\big(f^{(+d)},u,v\big)$
and of
$[0;1]^2\ni (u,v)\mapsto \MAX\big(f^{(-d)},u,v\big)$,
respectively.
Here $f^{(d)}$ denotes the $d$-th derivative of $f$ in case $d\geq1$;
and $f^{(-d)}$ the $d$-fold antiderivative with $f^{(-d)}(0)=0$.
\end{mydefinition}
Noting $f(x)=\MAX(f,x,x)$ and in view of Example~\ref{x:Param3}a),
it is then immediate that this second-order representation
renders the very operations from Theorem~\ref{t:Analytic}b) 
second-order polytime computable. In particular it permits 
to find integers $A_j$ with $\|f^{(j)}\|\leq A_j$;
but not to continuously (not to mention polytime computably,
and even restricted to $f\in\Comegax{[0;1]}$) deduce integers
$(K,A)$ satisfying $\forall j:\|f^{(j)}\|\leq A\cdot K^j\cdot j!$.


\end{document}